\documentclass[11pt]{article}
\usepackage{amsmath,amsthm,amssymb}
\usepackage{enumerate}
\usepackage{graphicx}
\usepackage[noadjust]{cite}
\usepackage{comment}
\usepackage{oands}
\usepackage{tikz}
\usetikzlibrary{calc,math}
\usepackage{changepage}
\usepackage{bbm}
\usepackage{mathtools}
\usepackage{enumitem}
\usepackage[margin=1in]{geometry}
\usepackage{hyperref}
\usepackage{authblk}
\usepackage[pagewise,mathlines]{lineno}
\usepackage{appendix}
\usepackage{multicol}
\usepackage{stmaryrd}
\usepackage{todonotes}
\usepackage{placeins}
\usepackage{xcolor}
\usepackage{pgfplots}
\pgfplotsset{compat=1.18}

\newcommand\sbullet[1][.5]{\mathbin{\vcenter{\hbox{\scalebox{#1}{$\bullet$}}}}}

\newcommand{\cR}{\mathcal{R}}
\newcommand{\cG}{\mathcal{G}}
\newcommand{\cT}{\mathcal{T}}
\newcommand{\cE}{\mathcal{E}}
\newcommand{\ZZ}{\mathbbm{Z}}
\newcommand{\cW}{\mathcal{W}}
\newcommand{\EE}{\mathbb{E}}
\usepackage{cleveref}
\theoremstyle{plain}
\newtheorem{thm}{Theorem}[section]
\newtheorem*{thm*}{Theorem}

\newtheorem{lem}[thm]{Lemma}
\newtheorem{prop}[thm]{Proposition}
\newtheorem{conj}[thm]{Conjecture}

\def\@rst #1 #2other{#1}
\newcommand\MR[1]{\relax\ifhmode\unskip\spacefactor3000 \space\fi
  \MRhref{\expandafter\@rst #1 other}{#1}}
\newcommand{\MRhref}[2]{\href{http://www.ams.org/mathscinet-getitem?mr=#1}{MR#2}}

\definecolor{prettyblue}{RGB}{20,80,160}
\definecolor{citegreen}{RGB}{50,120,90}
\hypersetup{
  colorlinks=true,
  linkcolor=prettyblue,
  urlcolor=prettyblue,
  citecolor=citegreen,
  linktocpage=true
}
 
\theoremstyle{definition}
\newtheorem{defn}[thm]{Definition}
\newtheorem{remark}[thm]{Remark}

\numberwithin{equation}{section}

\newcommand{\dsb}{\begin{adjustwidth}{2.5em}{0pt}
\begin{footnotesize}}
\newcommand{\dse}{\end{footnotesize}
\end{adjustwidth}}

\newcommand{\ssb}{\begin{adjustwidth}{2.5em}{0pt}}
\newcommand{\sse}{\end{adjustwidth}}

\newcommand{\aryb}{\begin{eqnarray*}}
\newcommand{\arye}{\end{eqnarray*}}
\def\alb#1\ale{\begin{align*}#1\end{align*}}
\def\allb#1\alle{\begin{align}#1\end{align}}
\newcommand{\eqb}{\begin{equation}}
\newcommand{\eqe}{\end{equation}}
\newcommand{\eqbn}{\begin{equation*}}
\newcommand{\eqen}{\end{equation*}}

\newcommand{\PP}{\mathbb{P}}
\newcommand{\cN}{\mathcal{N}}
\newcommand{\cB}{\mathcal{B}}
\newcommand{\NN}{\mathbb{N}}
\newcommand{\cM}{\mathcal{M}}
\newcommand{\cH}{\mathcal{H}}
\newcommand{\CC}{\mathbb{C}}
\newcommand{\HH}{\mathbb{H}}
\newcommand{\RR}{\mathbb{R}}
\newcommand{\BB}{\mathbbm}
\newcommand{\ol}{\overline}

\newcommand{\op}{\operatorname}

\newcommand{\frk}{\mathfrak}

\newcommand{\ep}{\varepsilon}

\newcommand{\bdy}{\partial}

\newcommand{\Fsym}[4]{
    \hyperref[def:F]{F_{#1,#2}^{#3,#4}}
}
\newcommand{\Ssym}[2]{
    \hyperref[def:non-constancy]{S_{#1,#2}}
}
\newcommand{\Sstarsym}[2]{
    \hyperref[def:non-constancy]{S^{*}_{#1,#2}}
}
\newcommand{\Tsym}[2]{
    \hyperref[def:TH]{\cT_{(#1,#2)}}
}
\newcommand{\Hsym}[2]{
    \hyperref[def:TH]{\cH_{(#1,#2)}}
}
\newcommand{\DKPZ}[2]{
    \hyperref[def:deltas]{\Delta_{\mathrm{KPZ}}(#1,#2)}
}
\newcommand{\DbKPZ}[2]{
    \hyperref[def:deltas]{\Delta^{\partial}_{\mathrm{KPZ}}(#1,#2)}
}
\newcommand{\DEuc}[2]{
    \hyperref[def:deltas]{\Delta_{\mathrm{Euc}}(#1,#2)}
}
\newcommand{\DLQG}[2]{
    \hyperref[def:deltas]{\Delta_{\mathrm{LQG}}(#1,#2)}
}
\newcommand{\Gsymlink}[1]{
    \hyperref[def:G-k]{\mathcal{G}_{#1}}
}
\newcommand{\Ggeqsymlink}[1]{
    \hyperref[def:G-geqk0]{\mathcal{G}_{>#1}}
}
\newcommand{\Gkasymlink}[2]{
    \hyperref[def:Gkalpha]{\mathcal{G}_{#1}^{#2}}
}
\newcommand{\Gtildesymlink}[1]{
    \hyperref[def:g-tilde]{\tilde{\mathcal{G}}_{#1}}
}
\newcommand{\Gtildegeqsymlink}[1]{
    \hyperref[def:ggeq-tilde]{\tilde{\mathcal{G}}_{>#1}}
}
\newcommand{\Gkatildesymlink}[2]{
    \hyperref[def:gka-tilde]{\tilde{\mathcal{G}}_{#1}^{#2}}
}
\newcommand{\Adeltasymlink}[1]{
    \hyperref[def:A-delta]{A_{#1}}
}
\newcommand{\Wsymlink}[1]{
    \hyperref[def:wk]{\mathcal{W}_{#1}}
}
\newcommand{\Wgeqsymlink}[1]{
    \hyperref[def:W-geqk0]{\mathcal{W}_{>#1}}
}
\newcommand{\Wkasymlink}[2]{
    \hyperref[eq:55]{\mathcal{W}_{#1}^{#2}}
}
\newcommand{\PPPlink}[1]{
    \hyperref[def:ppp]{\Pi_{#1}}
}

\let\originalleft\left
\let\originalright\right
\renewcommand{\left}{\mathopen{}\mathclose\bgroup\originalleft}
\renewcommand{\right}{\aftergroup\egroup\originalright}

\title{Dimension lower bounds in random geometry via Lipschitz functions}
\date{}
\author[1,2]{Manan Bhatia}
\author[3]{Ewain Gwynne}
\author[3]{Brin Harper}
\affil[1]{Massachusetts Institute of Technology}
\affil[2]{The University of Hong Kong}
\affil[3]{The University of Chicago}

\begin{document}

\maketitle

\begin{abstract}
  We prove lower bounds for the Hausdorff dimensions of various natural sets associated with the Liouville quantum gravity (LQG) metric. We prove that the set of 3-star points (i.e., starting points of three disjoint geodesics) has Hausdorff dimension at least two with respect to the LQG metric, which is conjectured to be optimal. Our proof works for a general class of planar length metrics which also includes, e.g., Kendall's Poisson roads metric. We additionally prove a dimension lower bound of one for the set of 2-star points intersected with the boundary and for the metric net intersected with the boundary, as well as a dimension lower bound of two for the intersection of two metric nets. In the particular setting of LQG, we obtain sharper lower bounds for the Hausdorff dimensions of the set of 2-star points and the LQG metric net, with respect to both the Euclidean metric and the LQG metric. Our proofs are primarily topological. The key idea is to express the sets of interest in terms of non-constancy sets of Lipschitz functions.
\end{abstract}

\tableofcontents

\section{Introduction}
The field of random geometry investigates natural models obtained by distorting distances in Euclidean spaces by a random noise. Chief and particularly well-studied among these are planar models such as first passage percolation (FPP) and random planar triangulations. In recent years, there has been particularly rapid progress in constructing candidate continuum limits of the above models, such as the directed landscape \cite{dov-dl}, as well as the Liouville quantum gravity ($\gamma$-LQG) metrics \cite{dddf-lfpp,gm-uniqueness} and the special case of the Brownian map \cite{legall-geodesics,miermont-brownian-map} (corresponding \cite{lqg-tbm1} to $\gamma=\sqrt{8/3}$).
Recently \cite{geosdonotpausenroute}, there has also been interest in another model -- namely, Kendall's Poisson roads metric \cite{kendallrandomlinesmetricspaces}, which is constructed from a family of infinite ``roads'' on the plane with differing ``speed limits.'' A notable common property of all these models is that of geodesic confluence (or coalescence) \cite{BSS19,gm-confluence,mq-strong-confluence,geosdonotpausenroute,BK25,DEP24}, wherein geodesics between nearby points tend to merge with each other; this phenomenon is the source of highly interesting fractal behavior.

As an example, owing to geodesic confluence, typical points in the above-mentioned continuum models only have a single geodesic emanating out ``locally.'' Indeed, defining a point $z$ in a metric space to be $k$-star if there exist $k$ distinct nontrivial geodesics $P_1,\dots, P_k$ such that $\bigcap_{i=1}^kP_i=\{z\}$, almost every point in the above geometries is a $1$-star. However, for $k\geq 2$, there may exist an exceptional set of points $T_{k\textrm{-star}}$ consisting of $k$-stars. Indeed, beginning with the Brownian map \cite{miermont-brownian-map,mq-strong-confluence,legall-geodesic-stars} and then followed by the directed landscape \cite{Bha22,Dau25}, the set $T_{k\textrm{-star}}$ has been deeply investigated and its Hausdorff dimension has been precisely computed for all values of $k$. Still, for other models of random geometry, a precise computation of the dimension of $T_{k\textrm{-star}}$ remains an open question. For instance, the more general $\gamma$-LQG models appear to lack sufficient integrability to enable a precise computation of such dimensions, and as a result, the sets $T_{k\textrm{-star}}$ are still not well-understood (see \cite{gwynne-geodesic-network} for the best known results). 

An interesting approach, developed \cite{BGH21,BGH22} in the literature on random geometries in the Kardar-Parisi-Zhang (KPZ) universality class \cite{kpz-fluctuation}, is to relate the set $T_{\textup{2-star}}$ to the ``non-constancy set'' of the function $\phi(z) = D(z,a) - D(z,b)$ defined as the difference of distances from two fixed points $a$ and $b$. An analysis of the H\"older continuity properties of such difference profiles yields precise lower bounds on the dimension of the set $T_{\textup{2-star}}$. Inspired by this approach, in this paper, we consider similar difference profiles defined for a general class of metric spaces, and analyze their non-constancy sets. Using this approach, for metric spaces satisfying a certain list of conditions\footnote{this includes coalescent metric spaces such as $\gamma$-LQG and Kendall's Poisson roads metric}, we establish (\Cref{3-star_points_general}) that the set $T_{\textup{3-star}}$  satisfies $\dim(T_{\textup{3-star}})\geq 2$, where $\dim(\cdot)$ denotes the Hausdorff dimension with respect to the random metric being considered. We note that \cite{Dau25} conjectures that for $\gamma$-LQG, we have $\dim(T_{\textup{3-star}})=2$ for all values of $\gamma$, and our result applied to $\gamma$-LQG yields a tight lower bound (\Cref{thm:lqgres}\ref{it:LQG-3-lb}). Additionally, we conjecture (Conjecture \ref{conj:2star-bdry}) that for all $\gamma$, a $\gamma$-LQG surface $(X,D)$ with boundary satisfies $\dim(\partial X\cap T_{\textup{2-star}})=1$ almost surely, and we prove (Theorem \ref{thm:lqgbbdres} \ref{it:LQG-2-lb}) that the lower bound here does hold.

Apart from $k$-stars, one can consider various other interesting fractal sets defined in terms of the random geometry. A particularly interesting one, especially from the point of view of $\gamma$-LQG, is the so-called metric net. Indeed, in these coalescent planar random geometries, a metric ball $\cB_r(z)=\{x:D(z,x)<r\}$ around a point $z$ is often not simply connected but rather has multiple ``holes'' (see \Cref{fig-metric-ball}), and one can consider the metric net $\cN^\infty(z)$ defined as the union of the outer-most boundaries of the balls $\cB_r(z)$ as the radius $r$ varies from $0$ to $\infty$ (see Section \ref{basic-defns-ms}). Using that the metric net can be expressed as the non-constancy set of a certain natural Lipschitz function, we apply the approach discussed above to obtain (\Cref{intersection_metric_nets_general}, \Cref{intersection_metric_nets}) new lower bounds on the Hausdorff dimensions of metric nets and of the intersections of metric nets centered around distinct points for a class of ``nice" coalescent metric spaces, which includes both LQG and the Poisson roads metric. 

In the latter half of the paper (see Section \ref{section-lqg-bounds}), we shift from the setting of a relatively general metric space to that of $\gamma$-LQG, and use the special structure therein to prove stronger dimension lower bounds for non-constancy sets of Lipschitz functions of the LQG metric. These sets will include, in particular, $2$-stars (\Cref{thm:lqgres}\ref{it:LQG-2-lb-wholeplane}, \Cref{thm:lqgbbdres}\ref{it:LQG-2-lb}) and metric nets (\Cref{thm:lqgres}\ref{it:LQG-1net-lb}, \Cref{thm:lqgbbdres}\ref{it:LQG-net-lb-boundary}), both in the ``bulk" and on the boundary. As opposed to the first half of the paper (see Section \ref{section-general-bounds}), where the techniques used are topological and more general, the latter half makes essential use of certain ``thickness'' estimates for the GFF which are specific to the case of LQG.

\begin{figure}[t]
\begin{center}
\includegraphics[width=0.45\textwidth]{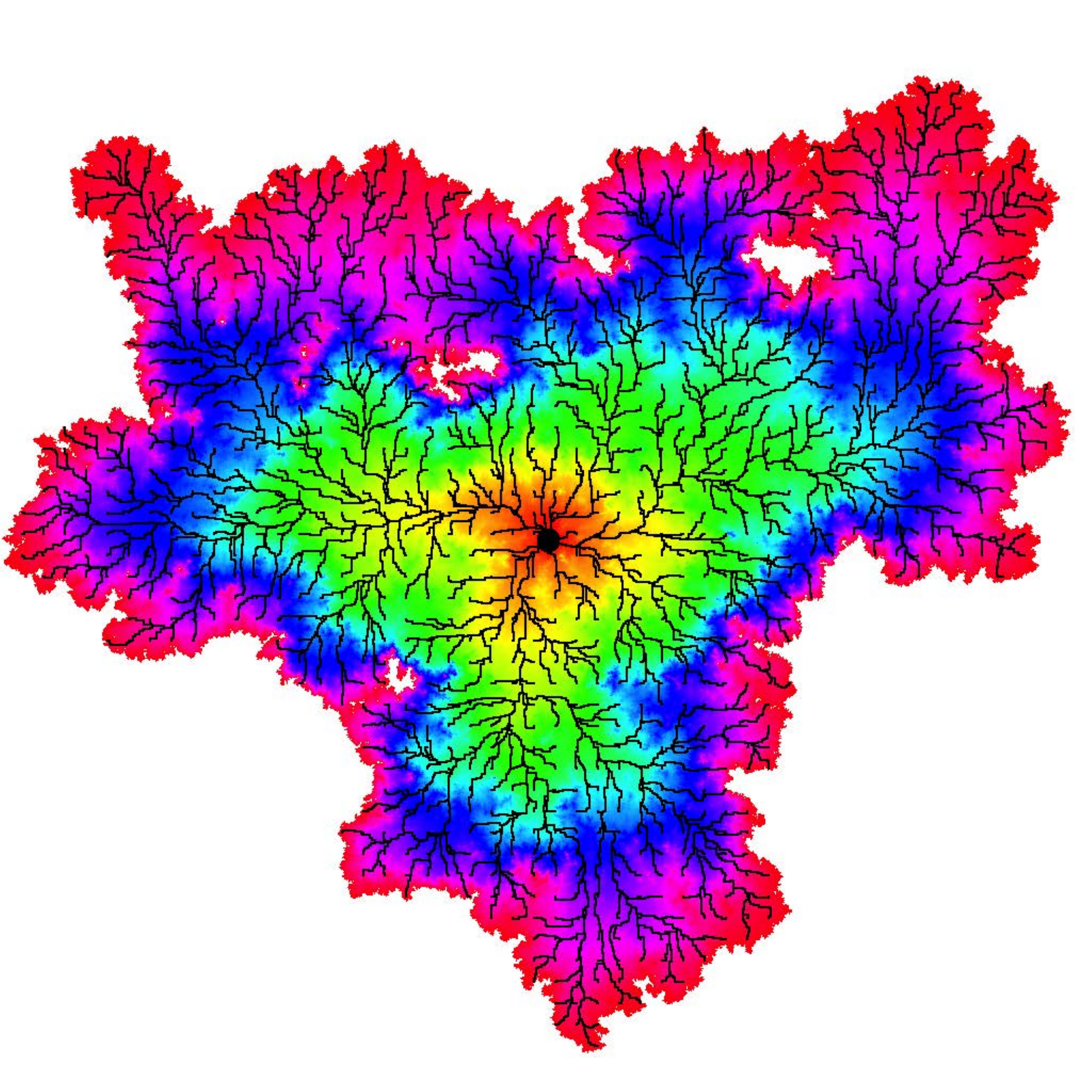}  
\caption{\label{fig-metric-ball}  
(We thank Jason Miller for providing us with the code used to generate this figure.) Simulation of an LQG metric ball $\cB_r(z)$, with $\gamma = \sqrt{8/3}$. The color gradient indicates distance from the center $z$, and the overlaid black lines depict geodesics from the center of the ball to other points. Notice the ``holes" inside the metric ball, i.e. the regions disconnected from infinity by the ball boundary; an LQG metric ball has infinitely many such holes. Also notice the branching behavior of the geodesics. These phenomena are consequences of geodesic confluence.}
\end{center}
\end{figure}

Before stating our main results, we collect a few basic definitions regarding metric spaces.

\subsection{Basic definitions regarding metric spaces}\label{basic-defns-ms}
\begin{itemize}
\item A topological space $X$ is said to be $\boldsymbol{\sigma}$\textbf{-compact} if it is the countable union of compact sets. 
\item A metric space $(X,D)$ is called \textbf{boundedly compact} if the closure of any metric ball is compact. If $D$ is a metric on $\mathbb{C}$ which induces the Euclidean topology, then $(\mathbb{C},D)$ is boundedly compact if and only if $\lim_{z\rightarrow\infty}D(w,z) = \infty$ for some (equivalently, every) $w \in \mathbb{C}$. 
\item For a metric space $(X,D)$, points $z$ and $w$, and a path $\eta\colon [a,b]\rightarrow X$ from $z$ to $w$,  define $\mathrm{len}(\eta;D)=\sup_{a=t_0<\dots<t_n=b}\sum_{i=0}^{n-1}D(\eta(t_i),\eta(t_{i+1}))$. The space $(X,D)$ is a \textbf{length space} if for all $z,w\in X$, we have $D(z,w)=\inf_{\eta: z\rightarrow w}\mathrm{len}(\eta;D)$, where the infimum is over all paths connecting $z$ and $w$; we then call $D$ a \textbf{length metric}. 
\item Any path $\eta$ attaining the above infimum is called a \textbf{geodesic} from $z$ to $w$ and is denoted by $\Gamma_{z,w}$. We will always parametrize geodesics according to their length; that is, we will have $\Gamma_{z,w}\colon[0,D(z,w)]\rightarrow X$ and $\mathrm{len}(\Gamma_{z,w}\lvert_{[s,t]};D)=t-s$ for all $0\leq s < t \leq D(z,w)$. We will often also use $\Gamma_{z,w}$ to also denote the image of the geodesic as a subset of $X$; for example, we can write $z,w\in \Gamma_{z,w}$.
\item If a length metric space admits a geodesic between every pair of points, then it is called a \textbf{geodesic space}. 
\item For a metric space $(X,D)$ and a subset $V \subset X$, let
\begin{equation*}
H_{\delta}^{d}(V) := \inf\left\{ \sum_{j=1}^{\infty}\op{diam}(U_j)^{d} : V \subset \bigcup_{j=1}^{\infty}U_{j}, \, \op{diam}(U_{j}) < \delta \right\}.
\end{equation*}
Then we define the $d$-dimensional \textbf{Hausdorff measure} of $V$ to be
\begin{equation*}
\mathcal{H}^{d}(V) := \lim_{\delta\to 0}H_{\delta}^{d}(V),
\end{equation*}
and the \textbf{Hausdorff dimension} of $V$ to be 
\begin{equation*}
\dim(V) := \op{inf}\{d \geq 0: \mathcal{H}^{d}(V) = 0\}.
\end{equation*}
We may sometimes write $\dim_{D}(V)$ instead of $\dim(V)$ to emphasize that we are computing the Hausdorff dimension of $V$ with respect to the metric $D$.
\item Given $k$ geodesics $\{\Gamma_{z_i,w_i}\}_{i=1}^k$, we say that these geodesics are \textbf{almost disjoint} if the sets $\Gamma_{z_i,w_i}\setminus \{z_i,w_i\}$ are mutually disjoint.
\item A point $z\in D$ is said to be a $\boldsymbol{k}$\textbf{-star} if there exist points $\{z_i\}_{i=1}^{k}$ and almost disjoint geodesics $\{\Gamma_{z,z_i}\}_{i=1}^k$.
\item For a general metric space $(X,D)$, the \textbf{filled metric ball} $\cB^{z,\bullet}_{r}(w)$ of radius $r$ around $w\in X$ and targeted at $z\in X$ is defined as follows. For $r< D(z,w)$, $\cB^{z,\bullet}_{r}(w)$ is the union of $\overline{\cB_{r}(w)}$ along with the set of points $x\in X$ which are disconnected by $\overline{\cB_{r}(w)}$ from $z$, in the sense that any continuous path from $z$ to $x$ must intersect $\overline{\cB_{r}(w)}$. For $r\geq D(z,w)$, $\cB^{z,\bullet}_{r}(w)$ is simply defined to be equal to $X$. We will refer to the set $\partial \cB^{z,\bullet}_{r}(w)$ as the \textbf{outer boundary} of the metric ball $\cB_{r}(w)$. Now, the \textbf{metric net} $\cN^{z}(w)$ targeted at $z$ is defined by $\cN^{z}(w)=\bigcup_{r>0}\partial \cB^{z,\bullet}_{r}(w)$.\footnote{For metrics on the complex plane $\CC$ or the upper half plane $\overline{\HH}$, one can also analogously define metric nets $\cN^z(a)$ for $z=\infty$.}
\end{itemize}

\subsection{Main results: a summary}
\label{sec:main-results:-summ}
As mentioned earlier, some of our results are applicable to rather general metric spaces, while others are specific to LQG. For the reader primarily interested in applications to LQG and the Poisson roads metric, we now concisely state all of the main results of this paper in this setting. Since the LQG metric and the Poisson roads metric are metrics on the plane (see \Cref{sec:preliminaries-2} for a brief working introduction to these metrics), one can define Hausdorff dimensions of exceptional sets with respect to both the random metric and the standard Euclidean metric. To make clear the metric with respect to which the dimension is being taken, we will use a subscript; for example, for LQG, $\dim_{D_h}(\cdot)$ will denote dimension with respect to the LQG metric, and $\dim_{\mathrm{Euc}}(\cdot)$ will indicate dimension with respect to the Euclidean metric.

To simplify notation in the following results, we introduce the following quantities. Throughout, $d_{\gamma}$ will refer to the almost sure Hausdorff dimension of $\gamma$-LQG (see \Cref{sec:preliminaries-2}). 
\hypertarget{deltas}{\begin{defn}\label{def:deltas}
For $\gamma\in (0,2)$ and $\beta>0$, we define
  \begin{align*}
    \Delta_{\mathrm{KPZ}}(\gamma,\beta)&=\beta\frac{d_\gamma (\gamma^2 + 4) - 2 \beta\gamma^2 - 2 \gamma \sqrt{4 d_\gamma^2 - (4\beta + \beta\gamma^2)d_\gamma + \beta^2\gamma^2 } }{2 d_\gamma^2},\\
    \Delta_{\mathrm{KPZ}}^{\partial}(\gamma,\beta) & := \beta\frac{d_{\gamma}(\gamma^{2}+4) -4\beta\gamma^{2} - 2\gamma\sqrt{4d_{\gamma}^{2} - (8\beta + 2\beta\gamma^{2})d_{\gamma} + 4\beta^{2}\gamma^{2}}}{2d_{\gamma}^{2}},\\
    \Delta_{\mathrm{Euc}}(\gamma, \beta) &:= \beta\frac{d_{\gamma}(\gamma^{2}+4) - 2\beta\gamma^{2} - 2\gamma\sqrt{2d_{\gamma}^{2}-(4\beta+\beta\gamma^{2}) d_{\gamma}+\beta^{2}\gamma^{2}}}{2d_{\gamma}^{2}}, \\
    \Delta_{\mathrm{LQG}}(\gamma,\beta) &:= \frac{2d_{\gamma}(\gamma^{2}+4) - 4\beta\gamma^{2} - 4\gamma\sqrt{2d_{\gamma}^{2}-(4\beta+\beta\gamma^{2}) d_{\gamma}+\beta^{2}\gamma^{2}}}{16+\gamma^{4}}.
\end{align*}
\end{defn}}

The first two of these quantities will be useful to us later when converting back and forth between Euclidean and metric dimensions in the context of LQG. For instance, the quantity $\DKPZ{\gamma}{\beta}$ appears when computing the minimum possible Euclidean dimension of a set known to have $D_h$-dimension at least $\beta$ (see Proposition \ref{prop:worstkpz}); indeed, this is a worst-case version (proved in \cite{GP22}) of the Knizhnik-Polyakov-Zamolodchikov (KPZ) \cite{KPZ88} formula, which is why we use the abbreviation KPZ. The quantity $\DbKPZ{\gamma}{\beta}$ appears in a ``boundary'' version of the aforementioned estimate (see Lemma \ref{bdy-kpz}). The latter two quantities $\DEuc{\gamma}{\beta}$ and $\DLQG{\gamma}{\beta}$ are quite specific to the setting of this paper, and will appear in certain lower bounds for Euclidean and LQG dimensions respectively. With the notation of \Cref{def:deltas} in place, we now summarize the new results of this paper for LQG.

\begin{thm}
  \label{thm:lqgres}
Fix $\gamma\in (0,2)$.  Let $h$ be a whole-plane GFF, and let $D_h$ be the associated $\gamma$-LQG metric on $\CC$. Then the following hold almost surely.
  
  \begin{enumerate}[label=(\arabic*)]
  \item\label{it:LQG-3-lb} $\dim_{D_h}\left(T_{\textup{3-star}}\right)\geq 2$, $\dim_{\mathrm{Euc}}(T_{\textup{3-star}})\geq \DKPZ{\gamma}{2}$.
  \item\label{it:LQG-2-lb-wholeplane} $\dim_{D_h}(T_{\textup{2-star}}) \geq 1 + \DLQG{\gamma}{1}$, $\dim_{\mathrm{Euc}}(T_{\textup{2-star}}) \geq 1 + \DEuc{\gamma}{1}$.
  \item\label{it:LQG-1net-lb} Simultaneously for all $a\in \CC$ and $z \in \CC\setminus\{a\}\cup\{\infty\}$, the metric net $\cN^z(a)$ satisfies $\dim_{D_h}(\cN^z(a)) \geq 1 + \DLQG{\gamma}{1}$ and $\dim_{\mathrm{Euc}}(\cN^z(a)) \geq 1 + \DEuc{\gamma}{1}$.
  \item\label{it:LQG-net-int-lb} Simultaneously for all points $a\neq b\in \CC$ and points $z\in \CC\cup\{\infty\}$ not lying on any $D_h$-geodesic from $a$ to $b$, we have $\dim_{D_h}(\cN^z(a)\cap \cN^z(b))\geq 2$ and $\dim_{\mathrm{Euc}}(\cN^z(a)\cap \cN^z(b))\geq \DKPZ{\gamma}{2}$.
  \end{enumerate}
\end{thm}

\begin{thm}
  \label{thm:lqgbbdres}
  Fix $\gamma\in (0,2)$. Let $h$ be a Neumann (free-boundary) GFF on the upper half plane $\mathbb{H}$, and let $D_h$ be the associated $\gamma$-LQG metric on $\overline{\mathbb{H}}$. Then the following hold almost surely.

  \begin{enumerate}[label=(\arabic*)]
  \item \label{it:LQG-2-lb} $\dim_{D_h}(T_{\textup{2-star}}\cap \RR)\geq 1$, $\dim_{\mathrm{Euc}}(T_{\textup{2-star}}\cap \RR) \geq \DbKPZ{\gamma}{1}$. 
  \item \label{it:LQG-net-lb-boundary} Simultaneously for all $a\in \overline{\mathbb{H}}$ and $z\in \overline{\mathbb{H}}\setminus \{a\}\cup\{\infty\}$, the metric net $\cN^z(a)$ satisfies $\dim_{D_h}(\cN^z(a)\cap \RR)\geq 1, \dim_{\mathrm{Euc}}(\cN^z(a)\cap \RR) \geq \DbKPZ{\gamma}{1}$.
  \end{enumerate}
\end{thm}

Note that for the particular case of the Brownian map (where $\gamma = \sqrt{8/3}$ and $d_{\gamma}$ = 4), we have $\DKPZ{\sqrt{8/3}}{2} \approx 0.0572$, $\DbKPZ{\sqrt{8/3}}{1} \approx 0.0286$, $\DEuc{\sqrt{8/3}}{1} \approx 0.378$ and $\DLQG{\sqrt{8/3}}{1} \approx 1.047$. For the Brownian map, it is known that the dimension of the set of $k$-stars is equal to $5 - k$ for $k \leq 5$ (see \cite{mq-strong-confluence,legall-geodesic-stars}), and it is expected that the dimension of the metric net is equal to 3 (\cite{dg-confluence}). For all other values of $\gamma$, however, the bounds appearing in \Cref{thm:lqgres} and \Cref{thm:lqgbbdres} of this paper are the first nontrivial bounds for the dimensions of these sets. 

We emphasize that our lower bounds for the LQG dimension of the set of 2-stars and the metric net are nontrivial for all values of $\gamma$; in fact, they are greater than 2 for all $\gamma$. We also note that all of $\DbKPZ{\gamma}{1}$, $\DEuc{\gamma}{1}$, $\DLQG{\gamma}{1}$ converge to $1$ as $\gamma\rightarrow 0$, and that $\DKPZ{\gamma}{2}$ converges to $2$ as $\gamma\rightarrow 0$ (see \Cref{fig:delta-plots}); in the context of \Cref{thm:lqgres} and \Cref{thm:lqgbbdres}, this aligns with the heuristic that the LQG metric behaves more and more ``Euclidean'' as $\gamma\rightarrow 0$.

\begin{figure}[t]
\centering
\usepgfplotslibrary{groupplots}
\pgfplotsset{compat=1.18}

\hspace*{-.75cm}
\begin{tikzpicture}

\begin{groupplot}[
  group style={
    group size=4 by 1,
    horizontal sep=0.7cm,
  },
  width=0.32\textwidth,
  height=0.22\textwidth,
  xmin=0, xmax=2,
  samples=250,
  domain=0.01:1.99,
  axis lines=left,
  xlabel={$\gamma$},
  grid=major,
  grid style={gray!25},
  tick style={black},
  tick label style={font=\scriptsize},
]

\nextgroupplot[
title={$\Delta_{\mathrm{KPZ}}(\gamma,2)$},
ymin=0, ymax=2,
ytick={0,1,2},
]

\addplot[blue, very thick] {
2*(
((2+x^2/2+x/sqrt(6))*(x^2+4)-4*x^2
-2*x*sqrt(4*(2+x^2/2+x/sqrt(6))^2
-(8+2*x^2)*(2+x^2/2+x/sqrt(6))+4*x^2))
/
(2*(2+x^2/2+x/sqrt(6))^2)
)
};

\nextgroupplot[
title={$\Delta_{\mathrm{KPZ}}^\partial(\gamma,1)$},
ymin=0, ymax=1,
ytick={0,0.5,1}
]

\addplot[blue, very thick] {
((2+x^2/2+x/sqrt(6))*(x^2+4)-4*x^2
-2*x*sqrt(4*(2+x^2/2+x/sqrt(6))^2
-(8+2*x^2)*(2+x^2/2+x/sqrt(6))+4*x^2))
/
(2*(2+x^2/2+x/sqrt(6))^2)
};

\nextgroupplot[
title={$\Delta_{\mathrm{Euc}}(\gamma,1)$},
ymin=0.35, ymax=1,
]

\addplot[blue, very thick] {
((2+x^2/2+x/sqrt(6))*(x^2+4)-2*x^2
-2*x*sqrt(2*(2+x^2/2+x/sqrt(6))^2
-(4+x^2)*(2+x^2/2+x/sqrt(6))+x^2))
/
(2*(2+x^2/2+x/sqrt(6))^2)
};

\nextgroupplot[
title={$\Delta_{\mathrm{LQG}}(\gamma,1)$},
ymin=1, ymax=1.05,
]

\addplot[blue, very thick] {
(2*(2+x^2/2+x/sqrt(6))*(x^2+4)-4*x^2
-4*x*sqrt(2*(2+x^2/2+x/sqrt(6))^2
-(4+x^2)*(2+x^2/2+x/sqrt(6))+x^2))
/
(16+x^4)
};

\end{groupplot}

\end{tikzpicture}
\vspace{-.75cm}
\caption{Graphs of the quantities $\DKPZ{\gamma}{2}$, $\DbKPZ{\gamma}{1}$, $\DEuc{\gamma}{1}$, and $\DLQG{\gamma}{1}$ which appear in \Cref{thm:lqgres} and \Cref{thm:lqgbbdres}, using the approximation $d_\gamma \approx 2+\gamma^2/2+\gamma/\sqrt6$ (this approximation is consistent with the best known bounds for $d_{\gamma}$; see \cite{gp-lfpp-bounds}).}\label{fig:delta-plots}
\end{figure}
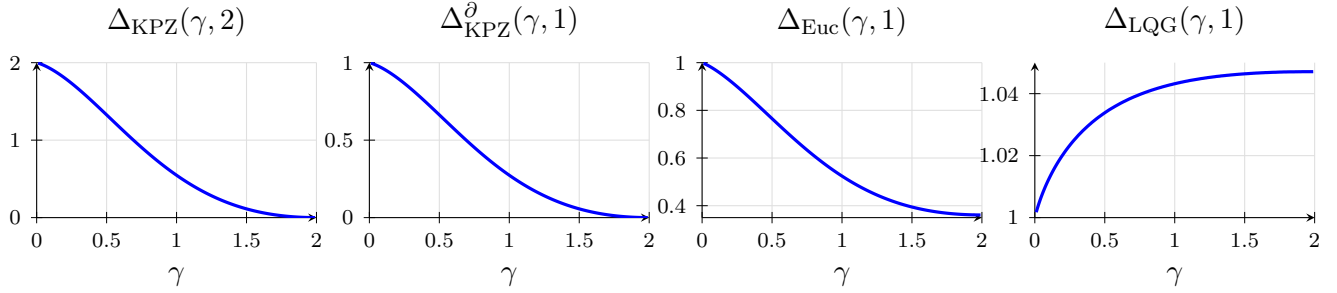

We remark that, while the results in Theorem \ref{thm:lqgres} and Theorem \ref{thm:lqgbbdres} are stated just for LQG surfaces corresponding to the whole-plane and Neumann GFF, they in fact hold for the interior and boundary of any of the standard LQG surfaces (such as quantum disks, quantum wedges and quantum cones, see e.g. \cite[Section 7]{berestycki-lqg-notes}), as can be seen by a local absolute continuity argument. We will not discuss this aspect further in the paper.

We reiterate that the proofs of \Cref{thm:lqgres} and \Cref{thm:lqgbbdres} involve a combination of two broad approaches: one valid for general metric spaces satisfying certain coalescence-related conditions, and one specific to LQG. Importantly, we note that the estimates attained using our LQG-specific approach (these are the estimates involving the $\Delta$ quantities introduced in Definition \ref{def:deltas}) are not expected to be optimal. However, we emphasize that these are first nontrivial estimates on the respective quantities in the literature. 

In contrast, we do expect the bound $\dim_{D_h}\left(T_{\textup{3-star}}\right)\geq 2$ of \Cref{thm:lqgres} \ref{it:LQG-3-lb} to be optimal. Indeed, as mentioned earlier, the dimension of the set of 3-star points for the LQG metric has been conjectured by Dauvergne to be \textit{equal} to 2 \cite[equation (8)]{Dau25} for all $\gamma\in (0,2)$, and we have been able to prove the tight lower bound. It is already known that the dimension is exactly equal to 2 for the $\gamma = \sqrt{8/3}$ case (see \cite{mq-strong-confluence,legall-geodesic-stars}), but for all other values of $\gamma$ the equality is only conjectural. In this paper, we additionally conjecture that for all values of $\gamma$ and LQG surfaces with boundary, the set of 2-stars on the boundary is always one dimensional. We now state a precise version of this for the case of the Neumann GFF.
\begin{conj}
  \label{conj:2star-bdry}
Let $h$ be a Neumann GFF and consider the LQG surface $(\overline{\mathbb{H}}, D_h)$. Then we have $\dim_{D_h}(T_{\textup{2-star}}\cap \RR)= 1$.
\end{conj}
In fact, we believe that the above should hold for any ``natural'' coalescent planar metric with boundary, though we do not attempt to make this precise in this paper. Note that Theorem \ref{thm:lqgbbdres} \ref{it:LQG-2-lb} proves the lower bound in this conjecture.

Finally, we state the main results of this paper for Kendall's Poisson roads metric. For a brief introduction to this metric, we refer the reader to Section \ref{sec:kendall-intro}.
\begin{thm}
  \label{thm:poissonroad}
  Fix $\beta>2$ and consider the corresponding Poisson roads metric $(\mathbb{R}^2,D)$. Then the following hold almost surely.
  \begin{enumerate}[label=(\arabic*)]
      \item \label{it:pr-3-lb} $\dim_D(T_{\textup{3-star}})\geq 2$.
      \item \label{it:pr-net-int-lb} Simultaneously for all points $a\neq b\in \RR^{2}$ and points $z\in \RR^{2}\cup\{\infty\}$ not lying on any $D$-geodesic from $a$ to $b$, we have $\dim_{D}(\cN^z(a)\cap \cN^z(b))\geq 2$.
    \end{enumerate}  
\end{thm}

\subsection{Results for general coalescent metric spaces}
\label{sec:main-results}
In this section, we will work in the setting of relatively general metric spaces. Working under the assumption that certain basic regularity conditions hold, we will state dimension lower bounds for various fractal sets associated to coalescent geometries. The results of this section are proved by topological arguments, and do not involve the usage of structures or integrability specific to any particular model of random geometry. In order to state our results, we will crucially require the following definition.

\hypertarget{def-non-constancy-set}{\begin{defn}[Non-constancy set]\label{def:non-constancy}Given a metric space $(X,D)$, a function $f=(f_1,\dots,f_n)\colon X\rightarrow \RR^n$, and a set $V\subseteq \RR^n$, we will often refer to the \textbf{non-constancy set} $S_{f,V}$, given by
\begin{equation}
  \label{eq:1}
  S_{f,V}=\{x \in f^{-1}(V) : \textrm{ for all } i, \: \nexists \epsilon > 0 \text{ s.t. } f_i \text{ is constant on } \cB_{\epsilon}(x)\}.
\end{equation}
In this setting we will also write $S_{f_{i},V}$ as a shorthand for the set
\begin{equation}
  S_{f_{i},V}=\{x \in f_{i}^{-1}\left(\pi_{i}\left(V\right)\right) : \nexists \epsilon > 0 \text{ s.t. } f_i \text{ is constant on } \cB_{\epsilon}(x)\}.
\end{equation}
where $\pi_{i}$ is the projection map projecting a point in $\mathbb{R}^{n}$ to its $i$th coordinate. Note that $S_{f,V} = \bigcap_{i=1}^n S_{f_i,V}$.
\end{defn}}
 
Our first result provides a general dimension lower bound for any non-constancy set of a locally Lipschitz function (i.e., a function which is Lipschitz on compact sets) for a rather general class of metric spaces.

\begin{thm}\label{intersection_universal_bounds}
  Let $(X,D)$ be a $\sigma$-compact metric space. For every $n\in \NN$ and every locally Lipschitz function $f\colon X\rightarrow \RR^n$, and every set $V\subseteq \RR^n$, we have
  \begin{equation}
    \label{eq:2}
    \max\{\dim(\Ssym{f}{V}),n-1\}\geq \dim(f(X)\cap V).
  \end{equation}
\end{thm} 

This result will be proven by an elementary argument at the beginning of \Cref{section-general-bounds}. In practice we will often apply the above for Lebesgue measurable sets $V\subseteq \RR^n$ satisfying $\mathrm{Leb}(f(X)\cap V)>0$, which we note automatically implies that $\dim(f(X)\cap V)=n$, thereby yielding the bound $\dim(\Ssym{f}{V})\geq n$.

By appropriately choosing the function $f$, we can probe the structure of different fractal sets associated to coalescent geometries. A prime example of this is the metric net (defined in \Cref{basic-defns-ms}), as is illustrated by the following result. 

\begin{thm}\label{intersection_metric_nets_general} Let $D$ be a geodesic metric on $\mathbb{C}$ that induces the Euclidean topology, and suppose that  $(\mathbb{C},D)$ is boundedly compact. Then for all points $a\neq b\in \CC$ and points $z\in \CC\cup\{\infty\}$ not lying on any $D$-geodesic from $a$ to $b$, we have $\dim_{D}(\cN^z(a)\cap \cN^z(b))\geq 2$. 
\end{thm}

The above results will later be used to obtain some of the results regarding metric nets stated in \Cref{thm:lqgres} and \Cref{thm:lqgbbdres}.

We now move on to results regarding star points. These results will be valid for metric spaces satisfying certain mild regularity conditions. For LQG and the Poisson roads metric, these regularity conditions will be verified to hold by using coalescence, and thus one can intuitively interpret these conditions as being a proxy for geodesic confluence. We now define a useful pair of sets which will help us formulate our results.
\begin{defn}[The sets $\cT_{(X,D)}$ and $\cH_{(X,D)}$]
  \label{def:TH}
  For a geodesic metric space $(X,D)$, we define
  \begin{equation*}
  \label{eq:16}
  \cT_{(X,D)}= \{ (u,v,w)\in X^3: \exists \textrm{ geodesic } \Gamma_{u,w}\textrm{ and } \exists x\in \Gamma_{u,w} \textrm{ not contained in any geodesics } \Gamma_{u,v}, \Gamma_{v,w}\}.
\end{equation*}
Also, we define $\cH_{(X,D)}$ to be the set of all $(u,v)\in X^2$ for which $u\neq v$ and the following property holds:
\begin{enumerate}[label=\Alph*.]
\item \label{it:assum_A}Let $S$ be the set of $s \in \mathbb{R}$ for which there exists a geodesic $\Gamma^s$ started from $u$ and a nontrivial interval $I_s$ for which $\Gamma^s\lvert_{I_s}\subseteq \{x \in X : D(x, u) + s = D(x,v)\}$. Then $S$ has zero Lebesgue measure.\footnote{Note that, since any geodesic $\Gamma$ has finite length, there are at most countably many $s$ for which $\Gamma$ spends a positive Lebesgue measure of time in the set $\{x\in X: D(x,u)+s=D(x,v)\}$. However, in assumption \ref{it:assum_A}, we are allowed to choose different geodesics $\Gamma^s$ for \emph{different} values of $s$.}
\end{enumerate}
\end{defn}

With the above definitions in hand, we can now state the following result for $3$-star points.

\begin{thm}\label{3-star_points_general}
Let $(X,D)$ be a connected $\sigma$-compact geodesic metric space, which is not a singleton, and whose first homology group (with integer coefficients) is trivial. Then for any $(a,b,c)\in \Tsym{X}{D}$, the function $f\colon X\rightarrow \RR^2$ defined by $f(x)=(D(x,a)-D(x,b), D(x,a)-D(x,c))$ satisfies $\dim (\Ssym{f}{\RR^2})\geq 2$.
\noindent Further, if $X$ additionally satisfies the following  three assumptions:
\begin{enumerate}[label=(\alph*)]
    \item \label{it:bddcpt}$X$ is boundedly compact, in the sense that its closed metric balls are compact.
    \item \label{it:tree} $X$ is not a tree, by which we mean a space with a unique path connecting every two points.
    \item \label{it:as} There exists a countable dense set $A \subset X$ such that $A^2\subseteq \Hsym{X}{D}$ and such that there is a unique geodesic $\Gamma_{u,v}$ for all $u,v\in A$.
\end{enumerate}
then we have $\dim\left(T_{\textup{3-star}}\right) \geq 2$.
\end{thm}

In practice, in the context of metric spaces obtained as scaling limits of discrete random geometries, \Cref{it:as} is a rather mild regularity condition; we check this condition for our metric spaces of interest fairly easily using the particular coalescence properties known to hold there (e.g., \Cref{prop:conf1} in the case of LQG). 

While Theorem \ref{3-star_points_general} concerns $3$-stars, one may also consider $2$-star points, and for these we obtain the following result. 
\begin{thm}
  \label{boundary_thm}
  Let $(X,D)$ be a $\sigma$-compact geodesic metric space and let $A\subseteq X$ be a closed connected set equipped with the restriction of the metric $D$. Then for all distinct $a,b\in A$, with $f\colon A\rightarrow \RR$ defined as $f(x)=D(x,a)-D(x,b)$, we have $\dim(\Ssym{f}{\RR})\geq 1$. Further, if the set $\Hsym{X}{D}\cap A^2$ is non-empty, then with $T_{\textup{2-star}}$ denoting the set of 2-stars for $(X,D)$, we have $\dim\left(A\cap T_{\textup{2-star}}\right) \geq 1$.
\end{thm}
In \Cref{boundary_thm}, the most natural choice is to consider a metric space $(X,D)$ with a connected boundary, and then take the set $A$ to be the above boundary. In this setting, the above result can be viewed as lower bounding the dimension of $2$-stars on the boundary; this should be compared with Theorem \ref{3-star_points_general}, which gives an analogous lower bound for $3$-stars in the ``bulk.''

\subsection{Results specific to LQG}
\label{sec:main-results-ii}
While the previous section stated results valid for non-constancy sets in rather general metric spaces, in this short section, we state results specific to non-constancy sets in LQG. These results make use of properties of the GFF in their proofs, and thus they do not have versions applicable to general metric spaces. In particular, we have the following lower bounds for the dimensions of non-constancy sets of Lipschitz functions with respect to the LQG metric. 

\begin{thm}\label{lqg_lower_bounds_thm}
    Let $h$ be a whole-plane GFF and let $\gamma\in (0,2)$. Let $f\colon \CC\rightarrow \RR$ be a non-constant locally Lipshitz function with respect to $D_h$. Then almost surely, for every set $V \subseteq f(\mathbb{C})$ and every $\beta>0$ such that $\beta\leq\dim(V)$, we have the following bounds for the non-constancy set $\Ssym{f}{V}$.
    \begin{align*}
        \dim_{D_h}(\Ssym{f}{V}) &\geq  \beta + \DLQG{\gamma}{\beta},\\
        \dim_{\mathrm{Euc}}(\Ssym{f}{V}) &\geq 1 + \DEuc{\gamma}{\beta}.
    \end{align*}
\end{thm}

Taking $h$ to be a Neumann GFF and considering the LQG surface $\left(\overline{\mathbb{H}},D_{h}\right)$, a slight modification of the proof of \Cref{lqg_lower_bounds_thm} also gives the following theorem about the dimension of the restriction of the non-constancy set to the boundary.

\begin{thm}\label{lqg_line_lower_bounds_thm}
 Fix $\gamma\in (0,2)$. Let $h$ be a Neumann GFF and consider the surface $(\overline{\mathbb{H}},D_h)$. Let $f\colon (\RR,D_h\lvert_{\RR\times \RR})\rightarrow \RR$ be a non-constant locally Lipschitz function. Then almost surely, for every set $V \subseteq f(\mathbb{R})$ and every $\beta>0$ such that $\beta\leq\mathrm{dim}(V)$, we have the following bounds for the set $\Ssym{f}{V}$.
 \begin{align*}
   \dim_{D_{h}}(\Ssym{f}{V}) &\geq \beta,\\
   \dim_{\mathrm{Euc}}(\Ssym{f}{V}) &\geq \DbKPZ{\gamma}{\beta}.
\end{align*}
\end{thm}

As we shall see later, \Cref{lqg_lower_bounds_thm} and \Cref{lqg_line_lower_bounds_thm} with suitable choices of Lipschitz function $f$ will yield many of the results stated earlier in \Cref{sec:main-results:-summ}.

\subsection{Acknowledgments} We thank Lingfu Zhang for first suggesting to us the idea of looking at non-constancy sets for Lipschitz functions of the LQG metric. We also thank Wei Qian and Riddhipratim Basu for helpful discussions. M.B.\ acknowledges the partial support of the NSF grant DMS-2153742 and the MathWorks fellowship. E.G.\ was partially supported by the NSF grant DMS-2245832 and B.H.\ was partially supported by the NSF Graduate Research Fellowship Program under Grant No. 2140001.

\section{Preliminaries}
\label{sec:preliminaries-2}

In this short section, we give a quick working introduction to two particular random metrics we will consider in this paper, namely the LQG metric and Kendall's Poisson roads random metric. 

\subsection{Gaussian free field} We now define the whole-plane Gaussian free field (GFF) $h$; for a detailed introduction to the GFF, we refer the reader to the texts \cite{berestycki-lqg-notes}, \cite{pw-gff-notes}, or \cite{shef-gff}. For $v,w\in \CC$, define the kernel $G_{\CC}(v,w)$ by
\begin{equation}
  \label{eq:21}
 G_{\CC}(v,w)=\log\frac{ \max \{|v|,1\} \max\{|v|,1\}}{|v-w|}.
\end{equation}
Informally, one wishes to define $h$ as a centered Gaussian process with covariance kernel given by $G_{\CC}$. However, $G_{\CC}$ blows up along the diagonal, and therefore one cannot define $h$ as an a.s.\ well-defined function. However, if one considers the set $\cM_{\CC}$ of signed measures $\rho$ on $\CC$ such that $\int G_{\CC}(v,w)d\rho(v) d\rho(w)<\infty$, then one can define the whole-plane GFF as a centered Gaussian process $\{(h,\rho)\}_{\rho\in \cM_{\CC}}$ such that for $\rho_1,\rho_2\in \cM_{\CC}$, we have
\begin{equation}
  \label{eq:23}
  \mathrm{Cov}( (h,\rho_1), (h,\rho_2))= \int G(v,w)d\rho_1(v) d\rho_2(w).
\end{equation}
We note that for $z\in \CC, r>0$, the uniform probability measure $\rho_{z,r}$ on the circle of radius $r$ around $z$ satisfies $\rho_{z,r}\in \cM_{\CC}$. In the literature, one often uses $h_{r}(z)$ (commonly referred to as the circle average of the GFF) to denote $(h,\rho_{z,r})$. We note also that the normalization in \eqref{eq:21} is chosen so as to almost surely have $h_{1}(0)=0$ (see e.g. \cite{vargas-dozz-notes}).
\subsection{The LQG measure \texorpdfstring{$\mu_h$}{mu h}} Fix $\gamma\in (0,2)$. Then the LQG measure $\mu_h$ is a random measure on $\CC$ defined \cite{kahane,DS11} so as to intuitively satisfy, for each $S\subseteq \CC$,
\begin{equation}
  \label{eq:30}
  ``\mu_h(S)=\int_S e^{\gamma h(z)}d^2z."
\end{equation}
Note that the above does not make literal sense since $h$ is not a function. However, the definition of the LQG measure can be made rigorous by a regularization procedure, and this is made precise via the theory of Gaussian multiplicate chaos. We refer the reader to \cite{rhodes-vargas-review} for a survey. 
\subsection{The LQG metric} Fix $\gamma\in (0,2)$. Informally, with $d_\gamma$ being a constant which turns out to be the Hausdorff dimension of $(\CC,D_h)$, the $\gamma$-LQG metric is defined so as to intuitively satisfy, for each $u,v\in \CC,$
\begin{equation}
  \label{eq:24}
  ``D_h(u,v)=\inf_{P\colon u\rightarrow v}\int_0^1 e^{(\gamma/ d_\gamma) h(P(t))}|P'(t)|dt,"
\end{equation}
where the infimum is over all smooth unit-time paths from $u$ to $v$. Note that the above equation is highly informal since point-wise values of $h$ are not well-defined. Over the course of a series of works culminating in \cite{dddf-lfpp,gm-uniqueness}, it was established that one can make sense of \eqref{eq:24} via a ``renormalization'' procedure, and that given a whole-plane GFF $h$, there is a well-defined metric $D_h$ associated to it which satisfies a certain list of natural and rigorous axioms suggested by the informal expression \eqref{eq:24}. In this paper, we do not list these axioms but instead refer the reader to the survey \cite{ddg-metric-survey}. 

For a given $\gamma\in (0,2)$, the Hausdorff dimension of the $\gamma$-LQG metric space is almost surely equal (\cite[Corollary 1.7]{DG20}) to the deterministic constant $d_\gamma$ appearing in \eqref{eq:24}; this was defined rigorously in \cite{DG20} and is heuristically the scaling exponent relating area and distance on a $\gamma$-LQG surface. The precise numerical value of $d_{\gamma}$ is still unknown, except in the particular case of $\gamma = \sqrt{8/3}$ (also called the Brownian map), where $d_{\sqrt{8/3}} = 4$. Throughout the paper, we shall set
\begin{align}\label{def:xi-and-Q}
\xi&= \gamma/d_\gamma, \\
    Q&= \gamma/2+2/\gamma. \nonumber
\end{align}

The metric space $(\CC,D_h)$ is almost surely a $\sigma$-compact geodesic metric space which induces the Euclidean topology. Further, for any fixed points $v,w \in \mathbb{C}$ and thus simultaneously for all rational points $v,w\in \mathbb{Q}^{2}$, there is \cite[Theorem 1.2]{MQ20} almost surely a unique $D_{h}$-geodesic $\Gamma_{v,w}$ from $v$ to $w$. 

The following result establishes optimal exponents for which the LQG metric is almost surely locally H\"older continuous with respect to the Euclidean metric and vice versa.

\begin{prop}[{{\cite[Theorem 1.7]{lqg-metric-estimates}}}]
\label{lem:estholder} Let $U \subseteq \mathbb{C}$ be open and bounded. Almost surely, the identity map from $U$, equipped with the Euclidean metric, to $(U,D_{h})$ is locally H\"older continuous with any exponent smaller than $\xi(Q-2)$, and is not locally H\"older continuous with any exponent larger than $\xi(Q-2)$, where $\xi$ and $Q$ are as defined in \eqref{def:xi-and-Q}. Furthermore, the inverse of this map is a.s. locally H\"older continuous with any exponent smaller than 
$\xi^{-1}(Q+2)^{-1}$ and is not locally H\"older continuous with any exponent larger than $\xi^{-1}(Q+2)^{-1}$.
\end{prop}

An interesting phenomenon in LQG is the heterogeneity of the space, in the sense that there is no fixed $\alpha$ for which Euclidean distance of order $\epsilon$ translates to a $D_h$-distance of order $\epsilon^{\alpha}$. Instead, the above $\alpha$ is different around different points in $\CC$, owing to the presence of ``thick'' points for the GFF, or regions where the GFF is atypically large or small. The following lemma capturing this phenomenon will be useful to us. 
In what follows, we will use $B_r(z)$ to denote the Euclidean ball of radius $r$ around $z$, while we will use $\cB_r(z)$ to denote the $D_h$-ball of radius $r$ around $z$.

\begin{prop}[{{\cite[Lemma 3.19]{lqg-metric-estimates}}}]
  \label{lem:est1}
Let $h$ be a whole-plane GFF. Fix a compact set $K\subseteq \CC$. Using $\mathrm{diam}_{D_h}(\cdot)$ to denote the LQG diameter of a set, for each $\alpha\in (0,\xi Q)$, for all $\epsilon$ small enough and for each $z \in K$, we have
  \begin{equation}
    \label{eq:25}
        \mathbb{P}\left[\mathrm{diam}_{D_{h}}(B_{\epsilon}(z)) > \epsilon^{\alpha} \right] \leq \epsilon^{(Q-\alpha/\xi)^{2}/2}.
  \end{equation}
\end{prop}
In order to control the LQG distance between two concentric circles around a point, we will later need the following result.
\begin{prop}[{{\cite[Lemma 3.21]{lqg-metric-estimates}}}]
  \label{lem:est2}
  Let $h$ be a whole-plane GFF. Let $K_1,K_2\subseteq \CC$ be disjoint non-empty connected and compact sets which are both not singletons. Then, with superpolynomial probability as $A\rightarrow \infty$, uniformly in $r>0$ and $z\in \CC$, we have
  \begin{equation}
    \label{eq:26}
    A^{-1}r^{\xi Q} e^{\xi h_r(z)} \leq D_h( rK_1 + z, rK_2 + z)\leq Ar^{\xi Q} e^{\xi h_r(z)}.
  \end{equation}
\end{prop}

Also, we will need the following almost sure result on the $\mu_h$-volume of LQG metric balls.
\begin{prop}[{{\cite[Theorem 1.1]{afs-metric-ball}}}]
  \label{lem:est3}
Let $h$ be a whole-plane GFF. For any compact set $K\subseteq \CC$ and $\delta>0$, we almost surely have
  \begin{equation}
    \label{eq:29}
    \sup_{s\in (0,1)} \sup_{z\in K} \frac{\mu_h(\cB_s(z))}{s^{d_\gamma-\delta}}<\infty \textrm{ and } \inf_{s\in (0,1)}\inf_{z\in K} \frac{\mu_{h}(\cB_s(z))}{s^{d_\gamma+\delta}}>0.
  \end{equation}
  Further, the Minkowski dimension of $\gamma$-LQG is a.s.\ equal to $d_\gamma$, in the sense that almost surely, for any Borel set $S\subseteq \CC$ with nontrivial interior, using $N_\epsilon$ to denote the minimum number of LQG balls with radius $\epsilon$ needed to cover $S$, we have $\lim_{\epsilon\rightarrow 0} \frac{\log N_{\epsilon}}{\log \epsilon^{-1}}=d_\gamma$.
\end{prop}
We now introduce a lemma which formalizes the phenomenon of geodesic confluence for the LQG metric.\footnote{In fact, a stronger form of \Cref{prop:conf1} holding simultaneously for all $u,v\in \CC$ has now been proved (see \cite{BK25}). The reason that we do not use this stronger version in this paper is to illustrate that the regularity condition \ref{it:as} of \Cref{3-star_points_general} and, thereby, the resulting lower bound on the 3-star dimension, can still be obtained even if the strongest forms of confluence are not yet known. This might be useful when working with other random geometries, where strong forms of confluence are yet to be established.}

\begin{prop}[{\cite[Lemma 3.11]{gwynne2022geodesicsmetricballboundaries}}]
  \label{prop:conf1}
Let $h$ be a whole-plane GFF and let $D_h$ be the associated $\gamma$-LQG metric. Fix a point $u\in \CC$. Then almost surely, for all points $v\in \CC$, for any $D_h$-geodesic $\Gamma$ from $u$ to $v$, we have the following. For any neighborhoods $U\ni u, V\ni v$ and any sequence $\left(\Gamma^{(n)}\right)$ of $D_h$-geodesics converging uniformly to $\Gamma$, we have $\Gamma\setminus \Gamma^{(n)} \; \cup \; \Gamma^{(n)}\setminus \Gamma\subseteq U\cup V$ for all sufficiently large $n$.
\end{prop}
Finally, we will require a version of the KPZ \cite{KPZ88} formula relating Euclidean and quantum dimensions in LQG. There have been several mathematical works \cite{DS11, RV11} proving different versions of this formula; we will require the following worst-case version from \cite{GP22}.

\begin{prop}[{{\cite[Theorem 1.8 (1.14)]{GP22}}}]
\label{prop:worstkpz}
Let $h$ be a whole-plane GFF and let $D_h$ be the associated $\gamma$-LQG metric. Then, almost surely, for every Borel set $A\subseteq \CC$, and every $s \leq \dim_{D_{h}}(A)$, we have
    \begin{equation}
        \mathrm{dim}_{\mathrm{Euc}}(A) \geq \DKPZ{\gamma}{s}.
      \end{equation}
\end{prop}
We note that in the source \cite[Theorem 1.8 (1.14)]{GP22}, the above result is presented as an upper bound for $\dim_{\mathrm{LQG}}(A)$ given $\dim_{\mathrm{Euc}}(A)$; simple algebraic manipulations inverting this expression yield Proposition \ref{prop:worstkpz}. Later on, when investigating fractal sets on the boundary of an LQG surface, we will also require a boundary version of Proposition \ref{prop:worstkpz}--- since we could not find this in the literature, we develop this later as Lemma \ref{bdy-kpz}.
\subsection{LQG surfaces with a boundary and the boundary measure \texorpdfstring{$\nu_h$}{nu h}} There are also natural LQG surfaces having a boundary, and we begin by describing a canonical example. Let $h$ be a Neumann GFF on the upper half plane $\mathbb{H}$ normalized to have average zero on the semicircle $\partial B_{1}(0)\cap \mathbb{H}$, by which we mean a Gaussian process with the kernel $G_{\mathbb{H}}\colon \mathbb{H}^2\rightarrow (0,\infty)$ given by
\begin{equation}
  \label{eq:27}
  G_{\mathbb{H}}(v,w)= \log \frac{\max(|v|,1)^2\max(|w|,1)^2}{|v-w||v-\overline{w}|},
\end{equation}
where $\overline{w}$ denotes the complex conjugate of $w$. Just as in the whole-plane case, one can canonically define the $\gamma$-LQG distance $D_h(v,w)$ for all $v,w\in \mathbb{H}$ and, in fact, this metric also extends continuously to all boundary points $\RR\subseteq \overline{\mathbb{H}}$ (\cite[Proposition 1.7]{hm-metric-gluing}). Note that $(\overline{\mathbb{H}}, D_h)$ is a $\sigma$-compact geodesic metric space and again a.s.\ admits a unique geodesic $\Gamma_{v,w}$ for any fixed $v,w\in \overline{\mathbb{H}}$ (this follows by the same argument as the proof of \cite[Lemma 4.2]{ddg-metric-survey}, which gives the analogous geodesic uniqueness statement for the whole-plane case). Just as in the whole-plane case, one can define the ``bulk'' LQG measure $\mu_h$ when $h$ is a Neumann GFF, and this is a Borel measure on $\overline{\mathbb{H}}$. Further, one can also define \cite{DS11} a boundary LQG measure $\nu_h$ which is informally defined for Borel sets $S\subseteq \RR$ by ``$\nu_h(S)=\int_Se^{(\gamma/2) h(z)}dx$.''

Just as in the whole-plane setting, geodesic confluence holds in the boundary setting as well. The following analogue of Proposition \ref{prop:conf1} captures coalescence in the boundary setting; this was established in \cite{Bha25} by adapting corresponding arguments from the bulk case.

\begin{prop}[{\cite[Lemma 79]{Bha25}}]
  \label{prop:conf2}
Let $h$ be a Neumann GFF and let $D_h$ be the associated $\gamma$-LQG metric on $\overline{\HH}$. Fix a point $u\in \overline{\HH}$. Then almost surely, for all points $v\in \overline{\HH}$, for any $D_h$-geodesic $\Gamma$ from $u$ to $v$, we have the following. For any neighborhoods $V\ni v, U\ni u$ and any sequence $\left(\Gamma^{(n)}\right)$ of $D_h$-geodesics converging uniformly to $\Gamma$, we have $\Gamma\setminus \Gamma^{(n)} \; \cup \; \Gamma^{(n)}\setminus \Gamma\subseteq U\cup V$ for all sufficiently large $n$.
\end{prop}

\subsection{Kendall's Poisson roads metric}
\label{sec:kendall-intro}
The planar Poisson roads metric \cite{kendallrandomlinesmetricspaces} is a family of random metrics on $\mathbb{R}^2$ wherein geodesics are constrained to a family of random straight lines with speed limits; we now give a quick definition. First, consider the space $\mathbb{L}$ of affine lines in $\mathbb{R}^2$. It can be shown that there exists a unique (up to a multiplicative constant) locally finite Borel measure $\mu$ on $\mathbb{L}$ which is additionally invariant under rotations and translations of the plane (see \cite{kendallrandomlinesmetricspaces}, Definition 1.2). After fixing a parameter $\beta>2$, we consider a Poisson process $\Pi$ with intensity proportional to $\mu \otimes v^{-\beta} dv$ on $\mathbb{L}\times (0,\infty)$. Now, each $(\ell,v)\in \Pi$ is to be viewed as a ``road," with $v$ being the ``speed limit'' corresponding to the path $\ell$. For each $x \in \RR^2$, we define the speed limit $V(x)=\max\{v: x\in (\ell,v)\in \Pi\}$, with the convention $\sup \emptyset = 0$.
Finally, for $u,v\in \RR^2$, one can define the distance $D(u,v)$ by
\begin{align}
  \label{eq:28}
  D(u,v)= \inf\{\tau: \exists &\textrm { a continuous path }P\colon [0,\tau]\rightarrow \RR^2 \textrm{ from } u \textrm { to } v \textrm{ satisfying }\nonumber\\
  &|P(t)-P(s)|\leq \int_{s}^{t}V(P(t))dt \textrm{ for all } [s,t]\subseteq [0,\tau]\}.
\end{align}
Any path $P$ attaining the above minimum is called a geodesic from $u$ to $v$. With the above definition, it turns out that $(\RR^2,D)$ is a $\sigma$-compact geodesic metric space. In fact, as established in \cite[Theorem 4.4]{kendallrandomlinesmetricspaces}, there is a unique geodesic $\Gamma_{u,v}$ between any fixed points $u,v\in \mathbb{R}^2$ (and thus for all rational points).

\section{Dimension lower bounds for general coalescent metric spaces}\label{section-general-bounds}
In this section, we provide the proofs of the results introduced in \Cref{sec:main-results}. We emphasize that the majority of these results are valid for rather general metric spaces and do not require the special structure of LQG or the Poisson roads metric. Eventually in \Cref{sec:appl-lqg-poiss}, we will specialize to the settings of LQG and the Poisson roads metric and use these results to obtain Theorems \ref{thm:lqgres}\ref{it:LQG-3-lb}, \ref{thm:lqgres}\ref{it:LQG-net-int-lb}, and \ref{thm:poissonroad}.

\subsection{Universal lower bound}

\begin{proof}[Proof of \Cref{intersection_universal_bounds}]
  Fix an arbitrary $V \subseteq \mathbb{R}^{n}$. The result is trivial if $\dim(f(X)\cap V)\leq n-1$, so let us assume that $\dim(f(X)\cap V)>n-1$. As a consequence of the standard fact that any locally Lipschitz map cannot increase Hausdorff dimension, it suffices to prove that 
  \begin{equation}\label{lipschitz-sfv-dimension-equality}
      \dim \left(f(\Ssym{f}{V})\right)=\dim\left(f(X)\cap V\right),
  \end{equation}
  as this will imply that $\dim(\Ssym{f}{V})\geq\dim \left(f(X)\cap V\right)$. 

  We will now prove that \eqref{lipschitz-sfv-dimension-equality} holds. To simplify notation, for a set $A\subseteq X$, we will use the shorthand
  \begin{equation}
    \label{eq:33}
    A^*=f^{-1}(V)\setminus A.
  \end{equation}
  Since $X$ is $\sigma$-compact, it is also separable in the sense that it has a countable dense set. Thus, writing $f=(f_1,\dots,f_n)$, the set $f_i(\Sstarsym{f_i}{V})$ is at most countable for each $1\leq i\leq n$. To see this, note that by the definition of the non-constancy set, for any distinct values $r, r' \in f_i(\Sstarsym{f_i}{V})$ there exist $z, z' \subset V$ and $\epsilon, \epsilon' > 0$ such that $f_{i}$ is identically equal to $r$ on $B_{\epsilon}(z)$ and is identically equal to $r'$ on $B_{\epsilon'}(z')$. As a result, we must necessarily have $B_{\epsilon}(z) \cap B_{\epsilon'}(z') = \emptyset$. However, since $X$ has a countable dense set, there can exist only countably many such disjoint balls, and thus only countably many elements of $f_i(\Sstarsym{f_i}{V})$. Now, with $\pi_i$ denoting the map projecting $\RR^n$ to its $i$th coordinate, since $\Ssym{f}{V}=\bigcap_{i=1}^n \Ssym{f_i}{V}$, we have
  \begin{equation}
    \label{eq:4}
    f(\Sstarsym{f}{V})\subseteq \bigcup_{i=1}^n\pi_i^{-1}( f_i( \Sstarsym{f_i}{V})).
  \end{equation}
  The right hand side of \eqref{eq:4} is a finite union of sets each having Hausdorff dimension at most $n-1$, since $f_i(\Sstarsym{f_i}{V})$ is countable for all $i$ and $\pi_{i}^{-1}(p)$ has dimension $n-1$ for all $p \in \mathbb{R}$. Thus $\dim(f(\Sstarsym{f}{V}))\leq n-1$. However, we have
  \begin{equation}
    \label{eq:50}
    f(X)\cap V= f(\Ssym{f}{V})\cup f(\Sstarsym{f}{V}),
  \end{equation}
  and since we have assumed $\dim(f(X)\cap V)>n-1$, we must have $\dim(f(\Ssym{f}{V}))=\dim(f(X)\cap V)$. Thus, we have shown that \eqref{lipschitz-sfv-dimension-equality} holds, and this completes the proof.
\end{proof}

Before moving on, let us also record an easy lemma valid for real-valued functions on connected spaces which will be useful to us later. 
\begin{lem}
  \label{lem:conn}
Let $(X,D)$ be a connected metric space. For any non-constant Lipschitz function $f\colon X\rightarrow \RR^{n}$ and every set $V\subseteq \RR^{n}$, we have $f(X)\cap V= f(\hyperlink{def-non-constancy-set}{S_{f,V}})$.
\end{lem}
\begin{proof}
  It suffices to show that for any $v\in f(X)\cap V$, we can find an $x\in X$ such that $f(x)=v$ and such that $f$ is not constant in a neighbourhood of $x$. Assume for the sake of contradiction that this is false; then the non-empty set $f^{-1}(\{v\})$ is open. However, since $f$ is continuous and since $\{v\}$ is closed, $f^{-1}(\{v\})$ is necessarily closed as well and thus clopen. Since $X$ is connected, the only clopen sets are $\emptyset$ and $X$. We already know that $f^{-1}(\{v\})\neq \emptyset$, and also that $f^{-1}(\{v\})\neq X$ since $f$ is not constant. This yields a contradiction and completes the proof.
\end{proof}

\subsection{Lower bounds for metric net intersections}

In this subsection, we will prove \Cref{intersection_metric_nets_general} and \Cref{intersection_metric_nets} (which we shall state soon), which give lower bounds on the dimensions of metric nets and their intersections for rather general metric spaces. Later on in \Cref{sec:appl-lqg-poiss}, applying these results to the specific settings of LQG and the Poisson roads metric, we will obtain Theorems \ref{thm:lqgres}\ref{it:LQG-net-int-lb} and \ref{thm:poissonroad}\ref{it:pr-net-int-lb} as short corollaries. We now state \Cref{intersection_metric_nets}.

\begin{thm}\label{intersection_metric_nets}
  Let $(X,D)$ be a connected $\sigma$-compact geodesic metric space. For $n\in \NN$, fix distinct points $u_1,\dots, u_n\in X$ and $z\in X$. With $\pi_i$ projecting $\RR^n$ to its $i$th coordinate, consider the Lipschitz function $f\colon X\rightarrow \RR^n$ defined by
  \begin{equation}
    \label{eq:7}
    \pi_i(f(x))= \inf\{r: x\in \cB_r^{\bullet,z}(u_i)\}
  \end{equation}
  for $1\leq i \leq n$. Then we have $\max(\dim(\bigcap_{i=1}^n\cN^{z}(u_i)),n-1)\geq \dim( f(X))$.
\end{thm}

Before proving \Cref{intersection_metric_nets}, we will first prove the following simple but helpful lemma, which guarantees that the infimum in \eqref{eq:7} is always a minimum.
\begin{lem}
  \label{lem:infmin}
  Let $(X,D)$ be a metric space. Given points $a,b,z\in X$, let $r_0=\inf\{r: b\in \cB_r^{\bullet,z}(a)\}$. Then we have $b\in \cB_{r_0}^{\bullet,z}(a)$.
\end{lem}
\begin{proof}
  Recall that the filled metric ball $\cB_r^{\bullet,z}(a)$ is a closed set which contains the closed (non-filled) metric ball $\overline{\cB_r(a)}$ (see \Cref{basic-defns-ms}). Since $\cB_{r}^{\bullet,z}(a)=X$ for all $r\geq D(a,z)$, we can assume that $r_0<D(a,z)$. With the aim of obtaining an eventual contradiction, assume that $b\notin \cB_{r_0}^{\bullet,z}(a)$. As a result, there is a continuous path $\eta\colon [s_1,s_2]\rightarrow X$ from $b$ to $z$ such that $\eta([s_1,s_2])\subseteq X\setminus \overline{\cB_{r_0}(a)}$. Now, the graph $\eta([s_1,s_2])$ is necessarily compact, and as a result, we have $D(\overline{\cB_{r_0}(a)}, \eta([s_1,s_2]))>0$. Subsequently, for all $\epsilon$ small enough, we must have $\eta([s_1,s_2])\subseteq X \setminus \overline{\cB_{r_0+\epsilon}(a)}$. This contradicts the definition of $r_0$ and completes the proof.
\end{proof}
We now provide the proof of \Cref{intersection_metric_nets}.
\begin{proof}[Proof of \Cref{intersection_metric_nets}]
  Note that, with $f = (f_{1},\dots,f_{n})$ as defined in \eqref{eq:7}, the metric net $\cN^{z}(u_{i})$ is exactly the non-constancy set $\Ssym{f_{i}}{\mathbb{R}^{n}}$. Thus, in view of \Cref{intersection_universal_bounds}, we need only check that $f$ is Lipschitz. In order to prove this, we first show that for distinct points $u,z\in X$, if we have $x\in \cB_r^{\bullet,z}(u)$ for some $r$, then we necessarily have $x'\in \cB_{r+D(x,x')}^{\bullet,z}(u)$. To see this, first note that if $x\in \overline{\cB_r(u)}$, then by the triangle inequality, we immediately have $x'\in \overline{\cB_{r+D(x,x')}(u)}\subseteq \cB_{r+D(x,x')}^{\bullet,z}(u)$. Now, the other case is when $x\in \cB_r^{\bullet,z}(u)$ but $x\not\in \overline{\cB_r(u)}$. In this case, $x$ must necessarily be disconnected from $z$ by $\partial \cB_r(u)$. Now, if $x'$ is also disconnected from $z$ by $\partial \cB_r(u)$, then we automatically have $x'\in \cB_r^{\bullet,z}(u)\subseteq \cB_{r+D(x,x')}^{\bullet,z}(u)$. Alternatively, if $x'$ and $z$ lie in the same connected component of $X\setminus \partial \cB_r(u)$, then we consider a geodesic $\Gamma_{x,x'}$ from $x$ to $x'$. Note that we must have $\Gamma_{x,x'}\cap \partial \cB_r(u)\neq \emptyset$. As a result, we have
  \begin{equation}
    \label{eq:9}
    x'\in \cB_{D(x,x')}(\partial \cB_{r}(x))\subseteq \cB_{r+D(x,x')}(x)\subseteq \cB_{r+D(x,x')}^{\bullet,z}(x),
  \end{equation}
  and this proves the assertion.

As a result of the above assertion and \Cref{lem:infmin}, note that for any $x,x'\in X$, and any $1\leq i \leq n$, we have
  \begin{equation}
    \label{eq:8}
    |\pi_if(x)-\pi_if(x')| =| \min\{r: \pi_i(x)\in \cB_r^{\bullet,z}(u_i)\}- \min\{r: \pi_i(x')\in \cB_r^{\bullet,z}(u_i)\}|\leq D(x,x').
  \end{equation}
  This shows that $f$ is Lipschitz, thereby completing the proof.
\end{proof}

We will now prove \Cref{intersection_metric_nets_general}, which, roughly speaking, says that for a sufficiently ``nice'' planar metric $D$, the intersection of two metric nets has Hausdorff dimension at least 2 (with respect to $D$).

\begin{figure}[t]
\begin{center}
\includegraphics[width=\textwidth]{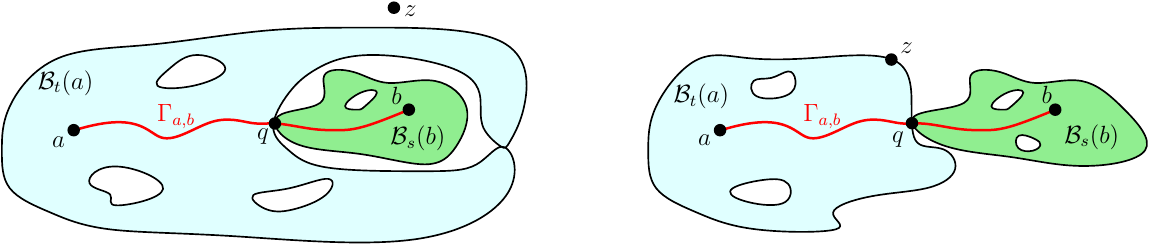}  
\caption{Illustration of the proof of \Cref{intersection_metric_nets_general}.  \textbf{Left:} The case where $b$ is cut off from $\infty$ before $\cB_{t}(a)$ absorbs $z$. \textbf{Right:} The case where $z$ is absorbed by $\cB_{t}(a)$ before $b$ is cut off from $\infty$, corresponding to $z \in \partial\cB_{t}(a)$ and $z \neq q$. (Note that in both cases, we have chosen to illustrate $b\notin \partial \cB_{t}(a)$ and thus $s>0$.)\label{fig-metric-net-intersection}}
\end{center}
\end{figure}

\begin{proof}[Proof of \Cref{intersection_metric_nets_general}]

  First, we observe that if we obtain the desired result for all finite $z$ not lying on any $D$-geodesic from $a$ to $b$, then we can choose a sequence of points $\{z_{i}\}$ diverging to $\infty$ and thereby also obtain the result for $z=\infty$ as well. So, from now on, we will assume that $z\in \CC$. In view of \Cref{intersection_metric_nets}, it suffices to establish that for the function $f\colon \CC\rightarrow \RR^2$ defined by
  \begin{equation}
    \label{eq:10}
    f(x)=( \min\{r: x\in \cB_{r}^{\bullet,z}(a)\}, \min\{r: x\in \cB_{r}^{\bullet,z}(b)\}),
  \end{equation}
  we have $\dim(f(\CC))=2$. We will do this by showing that, in fact, $\mathrm{Leb}(f(\CC))>0$. (The presence of minimums as opposed to infimums in \eqref{eq:10} is justified by \Cref{lem:infmin}.)
  
  Now, since we are working with a geodesic metric space, there exists a $D$-geodesic from $a$ to $b$; we fix a choice of such geodesic $\Gamma_{a,b}$. Writing $f=(f_1,f_2)$, define the time
  \begin{equation}
    \label{eq:34}
    t=f_1(b),
  \end{equation}
  and fix a point
  \begin{equation}
  q\in \Gamma_{a,b}\cap \partial \cB_{t}(a).
  \end{equation}
  We observe that $D(a,q)=t$, and as a result, the point $q$ must be unique depending on the already fixed choice of geodesic $\Gamma_{a,b}$. Also define the time
  \begin{equation}
  s=D(q,b),
  \end{equation}
  and note that $s+t=D(a,b)$. (It is possible that $b\in \partial \cB_t(a)$; in this case we simply have $s=0$ and $q=b$.) Note that there are two cases: the case where $b$ is cut off from $\infty$ before $\cB_{t}(a)$ absorbs $z$, and the case where $z$ is absorbed by $\cB_{t}(a)$ before $b$ is cut off from $\infty$ (corresponding to $z \in \partial\cB_{t}(a)$ and $z \neq q$). Our proof will work the same way in each case, but for clarity, we illustrate both possibilities in \Cref{fig-metric-net-intersection}. 

  Our aim now is to show that we can choose an $\epsilon > 0$ small enough so that for all $(r_1,r_2)\in [t-\epsilon,t)\times [s+\epsilon,s+2\epsilon]$, we have
  \begin{equation}
    \label{eq:37}
    \partial \cB_{r_1}^{\bullet,z}(a)\cap \partial \cB_{r_2}^{\bullet,z}(b)\neq \emptyset.
  \end{equation}
 This shall complete the proof, as for any $x\in \partial \cB_{r_1}^{\bullet,z}(a)\cap \partial \cB_{r_2}^{\bullet,z}(b)$ we necessarily have $f(x)=(r_1,r_2)$, and the set $[t-\epsilon,t)\times [s+\epsilon,s+2\epsilon]$ has positive Lebesgue measure. 
 
 In order to prove \eqref{eq:37}, we will first prove that if $\epsilon > 0$ is sufficiently small, then for all $(r_1,r_2)\in [t-\epsilon,t)\times [s+\epsilon,s+2\epsilon]$ the following three conditions hold:
  \begin{enumerate}
  \item \label{it:pt1} $b\notin  \cB_{r_1}^{\bullet,z}(a)$
  \item \label{it:pt2} $a\notin \cB_{r_2}^{\bullet,z}(b)$
  \item \label{it:pt3} $\cB_{r_1}^{\bullet,z}(a)\cap \cB_{r_2}^{\bullet,z}(b)\neq \emptyset$.
  \end{enumerate}

  That item \eqref{it:pt1} holds for all $r_1<t$ is immediate, since $t=f_1(b)$ is the minimum radius for which $b \in \cB_{r}^{\bullet,z}(a)$. It remains now to establish items \eqref{it:pt2} and \eqref{it:pt3}. 

To show item \eqref{it:pt2}, we first show that
\begin{equation}
\label{eq:both-outside}
a,z\in (\cB_{s}^{\bullet,\infty}(b))^c
\end{equation}
To see the above, first note that for all points $u\in \partial \cB_{t}(a)$, by the triangle inequality, we must have $D(b,u)\geq D(a,b)-t=s$. As a result, we necessarily have $\cB_{s}(b)\subseteq \CC\setminus \cB_{t}^{\bullet,b}(a)$, and thus $a$ lies in the unbounded connected component of $\CC\setminus \overline{\cB_{s}(b)}$, or in other words, $a\in (\cB_{s}^{\bullet,\infty}(b))^c$.

We now show that $z\in (\cB_{s}^{\bullet,\infty}(b))^c$. To see this, first note that by the definition of $t$, we have $z \in \cB_{t}^{\bullet,b}(a)$. Now, any point $x\in \overline{\cB_s(b)}\cap \cB_{t}^{\bullet,b}(a)= \partial \cB_s(b)\cap \partial \cB_{t}^{\bullet,b}(a)$ necessarily satisfies $D(b,x)=s,D(a,x)=t$ and thus lies on a geodesic from $a$ to $b$ (since $D(a,b)=s+t$). As a result, since $z$ does not lie on any geodesic between $a$ and $b$, and since we already know $z\in \cB_t^{\bullet,b}(a)$, we must have $z\in (\overline{\cB_s(b)})^c\cap \cB_t^{\bullet,b}(a)$. Thus, $z$ lies in the same connected component of $\CC\setminus \overline{\cB_{s}(b)}$ as $a$ does, and since we have already shown that $a$ lies in the unbounded component, this completes the proof of \eqref{eq:both-outside}.

Using \eqref{eq:both-outside}, we now complete the proof of item \eqref{it:pt2}. Since $a$ and $z$ both lie in the unbounded connected component $U$ of $\CC\setminus \overline{\cB_{s}(b)}$, we can choose a continuous path $\eta: [0,1] \rightarrow U$ from $a$ to $z$. Since $\eta([0,1])$ is compact and disjoint from $\overline{\cB_{s}(b)}$, its distance to $\overline{\cB_{s}(b)}$ is strictly positive, and so for all sufficiently small $\epsilon$, $\eta([0,1])$ is also disjoint from $\overline{\cB_{s+2\epsilon}(b)}$. So, we can always choose a small enough value of $\epsilon$ such that $\overline{\cB_{s+2\epsilon}(b)}$ does not disconnect $a$ from $z$, thereby yielding that $a\notin \cB_{s+2\epsilon}^{\bullet,z}(b)$. This yields \eqref{it:pt2}.
  
To see that item \eqref{it:pt3} holds, first note that since $q\in \overline{\cB_s(b)}$, we must have $\overline{\cB_{\epsilon}(q)}\subseteq \overline{\cB_{r_2}(b)}$ for all $r_2\in [s+\epsilon, s+2\epsilon]$. Now, we need only show that $\overline{\cB_{t-\epsilon}(a)}\cap \overline{\cB_{\epsilon}(q)}\neq \emptyset$ and this is immediate because $q\in \partial \overline{\cB_t(a)}$. This establishes that $\overline{\cB_{r_1}(a)} \cap \overline{\cB_{r_2}(b)} \neq \emptyset$ which is a stronger statement than item \eqref{it:pt3}.

We now use items \eqref{it:pt1}, \eqref{it:pt2}, and \eqref{it:pt3} to establish \eqref{eq:37} and thereby conclude the proof. By assumption, $D$ induces the Euclidean topology and $(\mathbb{C},D)$ is boundedly compact, or equivalently $\lim_{z\rightarrow\infty}D(w,z) = \infty$ for some (equivalently, every) $w \in \mathbb{C}$. So for all $r>0$, the boundary of every filled $D$-metric ball $\cB^{\bullet,z}_{r}(a)$ and every filled $D$-metric ball $\cB^{\bullet,z}_{r}(b)$ is a Jordan curve (\cite[Proposition 2.1]{tbm-characterization}, \cite[Lemma 2.4]{gwynne2022geodesicsmetricballboundaries}). Further, as a consequence of items \eqref{it:pt1} and \eqref{it:pt2}, if $\epsilon > 0$ is sufficiently small then for all $(r_1,r_2)\in [t-\epsilon,t)\times [s+\epsilon,s+2\epsilon]$ we must have
\begin{equation}
  \label{eq:35}
  \cB_{r_1}^{\bullet,z}(a)\not \subseteq \cB_{r_2}^{\bullet,z}(b), \cB_{r_2}^{\bullet,z}(b)\not\subseteq \cB_{r_1}^{\bullet,z}(a).
\end{equation}

We now view $\cB_{r_1}^{\bullet,z}(a), \cB_{r_2}^{\bullet,z}(b)$ as subsets of the Riemann sphere with boundaries given by the Jordan curves $\partial \cB_{r_1}^{\bullet,z}(a), \partial \cB_{r_2}^{\bullet,z}(b)$. By item \eqref{it:pt3}, note that for all $(r_1,r_2)\in [t-\epsilon,t)\times [s+\epsilon,s+2\epsilon]$, $\cB_{r_1}^{\bullet,z}(a)\cap \cB_{r_2}^{\bullet,z}(b)\neq \emptyset$, and as a consequence of this and \eqref{eq:35}, $\partial \cB_{r_2}^{\bullet,z}(b)$ must nontrivially intersect both the connected components of the complement of the Jordan curve $\partial \cB_{r_1}^{\bullet,z}(a)$. As a result, by the Jordan curve theorem, we must have $\partial \cB_{r_1}^{\bullet,z}(a)\cap \partial \cB_{r_2}^{\bullet,z}(b)\neq \emptyset$ for all $(r_1,r_2)\in [t-\epsilon,t)\times [s+\epsilon,s+2\epsilon]$, thereby establishing \eqref{eq:37} and completing the proof.
\end{proof}

\subsection{Lower bound for the set of 3-stars}\label{3-star-section}

The goal in this subsection is to prove \Cref{3-star_points_general}. We recall \Cref{def:TH} and fix a triple $(a,b,c)\in \Tsym{X}{D}$. Throughout the proof we will consider the Lipschitz function $f(x)=(D(x,a)-D(x,b), D(x,a)-D(x,c))$ from the theorem statement. Later on in \Cref{sec:appl-lqg-poiss}, applying \Cref{3-star_points_general} to the specific settings of LQG and the Poisson roads metric, we will obtain Theorems \ref{thm:lqgres}\ref{it:LQG-3-lb} and \ref{thm:poissonroad}\ref{it:pr-3-lb} as short corollaries.

\subsection*{I. Positive Lebesgue measure condition}
The lower bound on $\dim(\hyperlink{def-non-constancy-set}{S_{f,\RR^2}})$ will eventually be obtained by invoking Theorem \ref{intersection_universal_bounds}. To do so, we will require a lower bound of $2$ on the Hausdorff dimension of the set $f(X)$. In fact, as we state in the following proposition, we will show that the above set has positive two-dimensional Lebesgue measure, and this will be a key ingredient in the proof of Theorem \ref{3-star_points_general}.

\begin{prop} \label{prop-pos-leb}
Let $(X,D)$ be a geodesic metric space whose first homology group (with integer coefficients) is trivial. 
Then for any $(a,b,c)\in \Tsym{X}{D}$, the set
\eqb \label{eqn-pos-leb-set}
Q := \left\{ (r,s) \in \BB R^2 : \text{$\exists z\in X$ such that $D(z,a) +r = D(z,b) + s = D(z,c)$} \right\} . 
\eqe 
 has positive two-dimensional Lebesgue measure. 
\end{prop}

We emphasize that the condition $(a,b,c) \in \Tsym{X}{D}$ here cannot be removed: indeed, if, say, $(X,D)$ is a tree with finitely many branches, then $\Tsym{X}{D}$ is empty and the set $Q$ as in \eqref{eqn-pos-leb-set} has zero two-dimensional Lebesgue measure for every triple of points $(a,b,c) \in X^{3}$. We will later see in Proposition~\ref{prop-tree} that, roughly speaking, if $(a,b,c) \not\in \Tsym{X}{D}$ for all triples of points $a,b,c$ in a suitably nice dense subset of $X$, then $X$ in fact must be a tree.

Throughout this section, we consider $a,b,c \in \Tsym{X}{D}$.
The proof of Proposition~\ref{prop-pos-leb} is not too difficult, but does require several preparatory lemmas. First, for points $u,v\in X$, $r \in \mathbb{R}$, and $\#\in \{\leq, <, \geq, > , =\}$, we define
{\eqb  \label{def:F}
F_{u,v}^{r,\#} := \left\{ z\in X : D(z,v)-D (z,u)~ \# ~r \right\}.
\eqe}

The main step in the proof of Proposition~\ref{prop-pos-leb} is to show that there is a positive Lebesgue measure set $R \subset (-D(a,c) , D(a,c))$ such that for $r\in R$, the function $z\mapsto D(z,c) - D(z,b)$ is non-constant on $\Fsym{a}{c}{r}{=}$. This will eventually be done in Lemma~\ref{lem-non-constant}, with Lemma~\ref{lem-good-pt} as an intermediate step. Once we have the positive Lebesgue measure set $R$, we will let $J_r$ be the interval between the maximum and minimum values attained by $D(\cdot,c) - D(\cdot,b)$ on $\Fsym{a}{c}{r}{=}$. A topological argument (based on the connectivity of $\Fsym{a}{c}{r}{=}$; see Lemma~\ref{lem-connected}) shows that, for $Q$ as in~\eqref{eqn-pos-leb-set}, we have $(r,s) \in Q$ whenever $s \in J_r$. This shows that the two-dimensional Lebesgue measure of $Q$ is bounded below by the integral over $R$ of a positive function and is therefore positive. 

The proof that $D(\cdot,c) - D(\cdot,b)$ is non-constant on $\Fsym{a}{c}{r}{=}$ for a positive Lebesgue measure set of values of $r$ (see Lemma~\ref{lem-non-constant}) hinges on the assumption $(a,b,c) \in \Tsym{X}{D}$. To set the stage, we now begin by proving some basic properties of the sets defined in \eqref{def:F}.

\begin{lem}
  \label{lem-inside}
  Let $u,v\in X$, let $r\in (-D(u,v),D(u,v))$, and let $z\in \Fsym{u}{v}{r}{\geq}$ (resp.\ $\Fsym{u}{v}{r}{>}$). Then all geodesics from $z$ to $u$ stay inside $\Fsym{u}{v}{r}{\geq}$ (resp.\ $\Fsym{u}{v}{r}{>}$).
\end{lem}

\begin{proof}
  We only prove the statement for $\Fsym{u}{v}{r}{\geq}$ and an analogous argument works for $\Fsym{u}{v}{r}{>}$ as well. Let $\Gamma_{z,u}$ be a geodesic from $z$ to $u$ and suppose for contradiction that there exists $t \in [0,D(z,u)]$ such that $\Gamma_{z,u}(t) \notin \Fsym{u}{v}{r}{\geq}$. By the triangle inequality and the definition of $\Fsym{u}{v}{r}{\geq}$, 
\eqbn
D(z,v) \leq D(z, \Gamma_{z,u}(t)) + D(\Gamma_{z,u}(t) , v)  <  D(z,\Gamma_{z,u}(t)) + D(\Gamma_{z,u}(t) , u) + r = D(z,u) + r,
\eqen
where the step with the strict inequality is obtained by using that $\Gamma_{z,u}(t)\notin \Fsym{u}{v}{r}{\geq}$, and the last equality uses that $\Gamma_{z,u}$ is a geodesic from $z$ to $u$. The above contradicts the assumption that $z\in \Fsym{u}{v}{r}{\geq}$, and therefore, our assumption that $\Gamma_{z,u}(t)\notin \Fsym{u}{v}{r}{\geq}$ must be incorrect, thereby completing the proof.
\end{proof}

\begin{figure}[t]
\begin{center}
\includegraphics[width=0.7\textwidth]{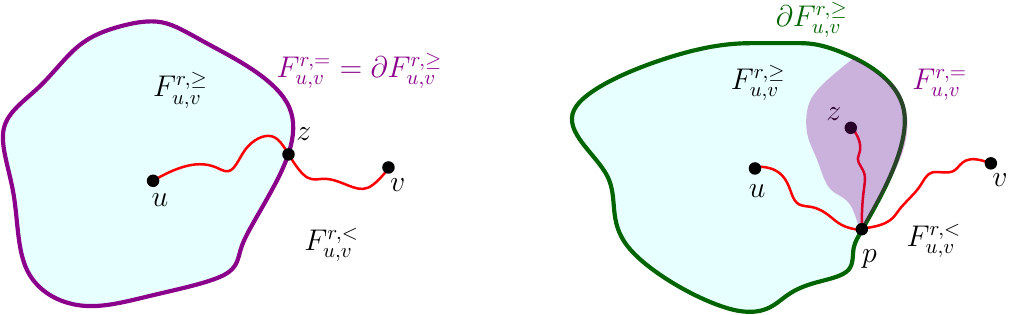}  
\caption{\label{fig-nonempty-interior}The set $\partial \Fsym{u}{v}{r}{\geq}$ may be be a proper subset of $\Fsym{u}{v}{r}{=}$, so it is indeed necessary to consider points $z \in \Fsym{u}{v}{r}{=} \setminus \partial \Fsym{u}{v}{r}{\geq}$ in the proof of \Cref{lem-connected}. \textbf{Left:} a scenario where $\Fsym{u}{v}{r}{=}$ has trivial interior, so that the equality $\partial \Fsym{u}{v}{r}{\geq} = \Fsym{u}{v}{r}{=}$ does hold. Here for every point $z \in \Fsym{u}{v}{r}{=}$, the geodesics (red) from $z$ to $u$ and from $z$ to $v$ do not spend a nontrivial amount of time inside $\Fsym{u}{v}{r}{=}$. \textbf{Right:} a scenario where, owing to confluence, $\Fsym{u}{v}{r}{=}$ has nontrivial interior, so that the equality $\partial \Fsym{u}{v}{r}{\geq} = \Fsym{u}{v}{r}{=}$ does not hold. Here the geodesics to $u$ and $v$ from interior points $z$ of $\Fsym{u}{v}{r}{=}$ all pass through a common ``confluence point'' $p$, and $\partial \Fsym{u}{v}{r}{\geq}$ (green) is a proper subset of $\Fsym{u}{v}{r}{=}$ (the union of the purple and green regions).}
\end{center}
\end{figure}

\begin{lem} \label{lem-connected}
For any points $u,v\in X$ and $r \in (-D (u,v) , D(u,v))$, the sets $\Fsym{u}{v}{r}{\geq}$, $\Fsym{u}{v}{r}{>}$, and $\Fsym{u}{v}{r}{=}$ are connected.
\end{lem}
\begin{proof}
  By Lemma \ref{lem-inside}, geodesics from any point in $\Fsym{u}{v}{r}{\geq}$ (resp.\ $\Fsym{u}{v}{r}{>}$) to $u$ stay inside $\Fsym{u}{v}{r}{\geq}$ (resp.\ $\Fsym{u}{v}{r}{>}$). As a result, we can path-connect two points in $\Fsym{u}{v}{r}{\geq}$ (resp.\ $\Fsym{u}{v}{r}{>}$) by first drawing a geodesic to $u$ and then to the latter point. Thus, both $\Fsym{u}{v}{r}{\geq}$ and $\Fsym{u}{v}{r}{>}$ are necessarily path-connected.

  Recall that since $X$ is a geodesic metric space, it must be path-connected. It must also be locally path-connected since every metric ball is necessarily path-connected (as can be seen by drawing geodesics to the center of the metric ball). Since we have also assumed that the first homology group (with integer coefficients) of $X$ is trivial, a topological fact (see, e.g.,~\cite[Lemma A.1]{gwynne-geodesic-network}) shows that the connectivity of $\Fsym{u}{v}{r}{\geq}$ and $ X\setminus \Fsym{u}{v}{r}{\geq}= \Fsym{v}{u}{-r}{>}$ implies that the boundary $\bdy \Fsym{u}{v}{r}{\geq}$ is connected. 
    
  We now deduce the connectivity of $\Fsym{u}{v}{r}{=}$. By the continuity of $D$, we have $\bdy \Fsym{u}{v}{r}{\geq} \subset \Fsym{u}{v}{r}{=}$. Now let $z \in \Fsym{u}{v}{r}{=} \setminus \bdy \Fsym{u}{v}{r}{\geq}$ (see \Cref{fig-nonempty-interior}) and let $\Gamma_{z,v}$ be a $D$-geodesic from $z$ to $v$. Since $r > -D(u,v)$, we have $v \notin \Fsym{u}{v}{r}{\geq}$, and so $\Gamma_{z,v}$  must hit $\bdy \Fsym{u}{v}{r}{\geq}$ since it is a continuous curve. Furthermore, the quantity $D(\cdot,v) - D(\cdot,u)$ cannot increase along a geodesic to $v$, so $\Gamma_{z,v}$ must remain in $\Fsym{u}{v}{r}{=}$ up until hitting $\bdy \Fsym{u}{v}{r}{\geq}$. Therefore, every point of $\Fsym{u}{v}{r}{=}\setminus \bdy \Fsym{u}{v}{r}{\geq}$ can be joined by a path in $\Fsym{u}{v}{r}{=}$ to a point of $\bdy \Fsym{u}{v}{r}{\geq}$. Since $\bdy \Fsym{u}{v}{r}{\geq}$ is connected, it follows that $\Fsym{u}{v}{r}{=}$ is also connected. 
\end{proof}

\begin{lem} \label{lem-geo-min}
Consider $u,v\in X$ and let $\Gamma_{v,u}$ be a $D$-geodesic from $v$ to $u$. Then, for each $t\in [0,D(u,v)]$, we have $\Gamma_{v,u}(t) \in \Fsym{u}{v}{2t - D(u,v)}{=}$ and 
\eqb \label{eqn-geo-min}
D(z,v) \geq t = D( \Gamma_{v,u}(t) , v) ,\quad \forall z\in \Fsym{u}{v}{2t-D(u,v)}{=}.
\eqe
\end{lem}
\begin{proof}
By the definition~\eqref{def:F} of $\Fsym{u}{v}{2t-D(u,v)}{=}$ and the definition of a $D$-geodesic, we have $\Gamma_{v,u}(t) \in \Fsym{u}{v}{2t - D(u,v)}{=}$. 
Now consider $z\in \Fsym{u}{v}{2t-D(u,v)}{=}$ and suppose for contradiction that $D(z,v) < t$. Then by the triangle inequality and the definition of $\Fsym{u}{v}{2t - D(u,v)}{=}$, we deduce
\eqbn
D(u,v) \leq D(z,u) + D(z,v) = 2D(z,v) - 2t + D(u,v) < D(u,v) ,
\eqen
which is a contradiction. 
\end{proof}

We will now prove two lemmas which crucially use the assumption $(a,b,c)\in \Tsym{X}{D}$; then, we will use these lemmas to complete the proof of Proposition \ref{prop-pos-leb}. For the following lemmas, we keep in mind the simple equality $D(\Gamma_{c,a}(t) , c) - D(\Gamma_{c,a}(t) , a)=2t-D(a,c)$.

\begin{figure}[t]
\begin{center}
\includegraphics[width=0.8\textwidth]{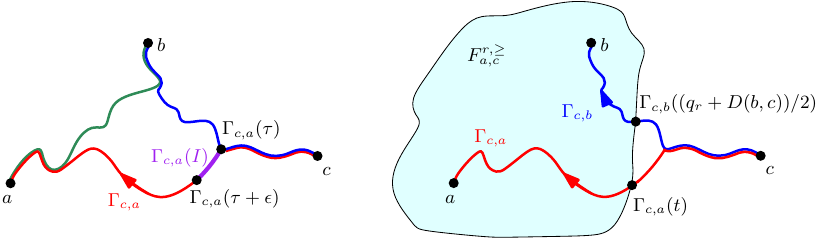}  
\caption{\label{fig-non-constant} \textbf{Left:} Illustration of the proof of Lemma~\ref{lem-good-pt}. The blue and green curves are $D$-geodesics from $c$ and $a$ to $b$. The set $\Gamma_{c,a}(T)$ where $D(\Gamma_{c,a}(t) ,a) + D(\Gamma_{c,a}(t) ,b) > D(a,b)$ and $D(\Gamma_{c,a}(t) ,b) + D(\Gamma_{c,a}(t) ,c) > D(\Gamma_{c,a}(t) ,c)$ consists of points on $\Gamma_{c,a}$ which are not hit by any $D$-geodesics from $c$ or $a$ to $b$. The ``good" set $I$ (purple) is a small open interval close to the left endpoint of $T$ (note that $\Gamma_{c,a}$ goes from $c$ to $a$). 
\textbf{Right:} Illustration of the proof of Lemma~\ref{lem-non-constant}. Note that $t = (r + D(a,c))/2 \in I$.}
\end{center}
\end{figure}

\begin{lem} \label{lem-good-pt} 
Let $(a,b,c)\in \Tsym{X}{D}$. There exists a geodesic $\Gamma_{c,a}$ from $c$ to $a$ and an open interval $I\subset (0,D(a,c))$ such that for each $t\in I$, 
\eqb  \label{eqn-good-pt-tri}
D(\Gamma_{c,a}(t) , b ) + D(\Gamma_{c,a}(t) , c) > D(b,c)  
\eqe 
and
\eqb \label{eqn-good-pt-diff} 
D(\Gamma_{c,a}(t) , c) - D(\Gamma_{c,a}(t) , a) < D(b,c) - D(b,a) .
\eqe 
\end{lem}

\begin{proof}[Proof of Lemma \ref{lem-good-pt}]
See Figure~\ref{fig-non-constant}, left, for an illustration of the proof.
Since $(a,b,c)\in \Tsym{X}{D}$, we can choose a geodesic $\Gamma_{c,a}$ from $c$ to $a$ and a point $x$ on $\Gamma_{c,a}$ such that $x$ does not lie on any geodesic from $b$ to $a$ or any geodesic from $b$ to $c$. Note that this is equivalent to the strict triangle inequality condition 
\begin{equation}
  \label{eq:3}
  D(x,a) + D(x,b) > D(a,b) \quad \text{and} \quad D(x,b) + D(x,c) > D(b,c) .
\end{equation}
If $t\in [0,D(a,c)]$ such that $D(\Gamma_{c,a}(t) , b) + D(\Gamma_{c,a}(t) , c) = D(b,c)$, then for any $s\in [0,t]$, the triangle inequality together with the fact that $\Gamma_{c,a}$ is a geodesic started from $c$ give
\alb
D(\Gamma_{c,a}(s) , b) + D(\Gamma_{c,a}(s) , c) 
&\leq [ D(\Gamma_{c,a}(t) ,b) + D(\Gamma_{c,a}(s) , \Gamma_{c,a}(t)) ] + s \\
&= D(\Gamma_{c,a}(t) , b) + t
= D(\Gamma_{c,a}(t) , b) + D(\Gamma_{c,a}(t) , c) 
= D(b,c) .
\ale
Hence, the set of $t\in [0,D(a,c)]$ such that $D(\Gamma_{c,a}(t) , b) + D(\Gamma_{c,a}(t) , c) = D(b,c)$ is a closed interval containing 0. 
Similarly, the set of $t\in [0,D(a,c)]$ such that $D(\Gamma_{c,a}(t) , a) + D(\Gamma_{c,a}(t) , b) = D(a,b)$ is a closed interval containing $D(a,c)$. By the above and the presence of the point $x\in \Gamma_{c,a}$ satisfying \eqref{eq:3}, the set
\eqbn
T = \left\{t\in [0,D(a,c)] : D(\Gamma_{c,a}(t) ,a) + D(\Gamma_{c,a}(t) ,b) > D(a,b) \: \text{and} \: D(\Gamma_{c,a}(t) ,b) + D(\Gamma_{c,a}(t) ,c) > D(b ,c) \right\}   
\eqen
is non-empty and is an open interval contained in $(0,D(a,c))$. Thus, with $\tau$ being the left endpoint of the interval $T$, we must have 
\eqb \label{eqn-endpt-c}
D(\Gamma_{c,a}(t) ,b) + D(\Gamma_{c,a}(t) ,c) > D(b,c) ,\quad \forall t \in (\tau , D(a,c)) 
\eqe
and
\eqb  \label{eqn-endpt-tri}
  D(\Gamma_{c,a}(\tau) , a) + D(\Gamma_{c,a}(\tau) , b) > D(a,b)  \quad \text{and} \quad D(\Gamma_{c,a}(\tau) ,b) + D(\Gamma_{c,a}(\tau) ,c) = D(b,c)   .
\eqe 
From~\eqref{eqn-endpt-tri}, we obtain
\begin{equation}
  \label{eqn-endpt-diff}
D(\Gamma_{c,a}(\tau)) , c) - D(\Gamma_{c,a}(\tau) , a)< D(b,c) - D(a,b).
\end{equation}
By~\eqref{eqn-endpt-diff} and the continuity of $t\mapsto D(\Gamma_{c,a}(t)) , c) - D(\Gamma_{c,a}(t) , a)$, there exists $\ep > 0$ such that~\eqref{eqn-good-pt-diff} holds for all $t \in (\tau -\ep , \tau+\ep)$. Combining this with~\eqref{eqn-endpt-c} gives the conclusion of the lemma with $I = (\tau ,\tau+\ep)$.  
\end{proof}

\begin{lem} \label{lem-non-constant}
Let $(a,b,c)\in \Tsym{X}{D}$. Let $\Gamma_{c,a}$ and $I \subset (0,D(a,c))$ be the geodesic from $c$ to $a$ and the open interval whose existence is given by Lemma~\ref{lem-good-pt}. 
Let 
\eqbn
R := \left\{ 2t - D(a,c) : t \in I \right\}  \subset (-D(a,c) , D(a,c)).
\eqen
For each $r\in R$, there exists $z_r,w_r\in \Fsym{a}{c}{r}{=}$ such that
\eqbn
D(z_r,c) - D(z_r,b) < D(w_r,c) - D(w_r,b) .
\eqen
\end{lem}

\begin{proof}[Proof of Lemma \ref{lem-non-constant}]
See Figure~\ref{fig-non-constant}, right, for an illustration of the proof.
Suppose by way of contradiction that there exists $r\in R$ such that $D(\cdot,c) - D(\cdot,b)$ is constant on $\Fsym{a}{c}{r}{=}$. Define $q_r$ such that
\eqbn
q_r := D(z,c) - D(z,b) ,\quad \text{for every $z\in \Fsym{a}{c}{r}{=}$.} 
\eqen
By the definitions of $\Fsym{a}{c}{r}{=}$ and $R$, for $t= \frac{r + D(a,c)}{2}\in I$, we have
\eqb \label{eqn-pt-on-geo0}
  \Gamma_{c,a}(t) \in \Fsym{a}{c}{r}{=}.
\eqe 
By ~\eqref{eqn-pt-on-geo0} along with ~\eqref{eqn-good-pt-diff} in the defining property of $I$, we have 
\eqbn
r = D(\Gamma_{c,a}(t) , c) - D(\Gamma_{c,a}(t) , a)  < D(b,c) - D(b,a) .
\eqen
Therefore, $b \in \Fsym{a}{c}{r}{\geq}$. Let $\Gamma_{c,b}$ be a $D$-geodesic from $c$ to $b$.  
Since $b\in \Fsym{a}{c}{r}{\geq}$ and $c\notin \Fsym{a}{c}{r}{\geq}$, we get that $\Gamma_{c,b}$ must pass through $\Fsym{a}{c}{r}{=}$. By the definition of $q_r$, the only point of $\Gamma_{c,b}$ which can belong to $\Fsym{a}{c}{r}{=}$ is $ \Gamma_{c,b}((q_r + D(b,c))/2) $. By Lemma~\ref{lem-geo-min}, for any $\alpha\in \Fsym{a}{c}{r}{=}$, we have
\begin{equation}
  \label{eq:12}
  D(\alpha,c)\geq D(\Gamma_{c,a}(t),c)=t.
\end{equation}
Now by \eqref{eq:12}, we obtain
\eqb \label{eqn-qr-relation}
q_r + D(b,c) =2  D( \Gamma_{c,b}((q_r + D(b,c))/2)  , c) \geq 2   D( \Gamma_{c,a}(t) , c) =  2t.
\eqe 

\noindent By using the definition of $q_r$ and since $\Gamma_{c,a}(t) \in \Fsym{a}{c}{r}{=}$, by \Cref{eqn-qr-relation}, we have
\eqbn
 D( \Gamma_{c,a}(t)  , c) - D( \Gamma_{c,a}(t) , b) = q_r \geq 2t  - D(b,c) .
\eqen 
Since $D(\Gamma_{c,a}(t) , c) = t$, this is equivalent to
\eqb  \label{eqn-diff-on-geo}
D(b,c)\geq D(\Gamma_{c,a}(t),c) + D(\Gamma_{c,a}(t),b).
\eqe

\noindent This contradicts~\eqref{eqn-good-pt-tri} in the defining property of $I$, thereby completing the proof.
\end{proof}

Having proved Lemmas \ref{lem-good-pt} and \ref{lem-non-constant}, we can now complete the proof of the main result of this subsection, Proposition~\ref{prop-pos-leb}.
 
\begin{proof}[Proof of Proposition~\ref{prop-pos-leb}] 
By \Cref{lem-non-constant}, adopting the notation of the lemma, for each $r\in R$, there exist $z_r ,w_r  \in \Fsym{a}{c}{r}{=}$ such that 
\eqb  \label{eqn-non-constant-diff}
D(z_r ,c) - D(z_r ,b) < D(w_r ,c) - D(w_r ,b)  .
\eqe 
For $r \in R$, we define the set $G_{r,s}$ by 
\begin{equation}
  \label{eq:14}
  G_{r,s}=\{z\in \Fsym{a}{c}{r}{=}: D(z,c)-D(z,b)\geq s\},
\end{equation}
and we note that $G_{r,s}$ is closed since $\Fsym{a}{c}{r}{=}$ is also closed.
Therefore, if for $r\in R$, we define $J_r$ by
\eqbn
J_r := \left( D(z_r ,c) - D(z_r ,b ) ,  D(w_r ,c) - D(w_r ,b) \right),
\eqen
then for all $s\in J_r$, the sets $G_{r,s}$ and $\overline{\Fsym{a}{c}{r}{=}\setminus G_{r,s}}$ are both non-empty subsets of $\Fsym{a}{c}{r}{=}$ which are also closed. Since $\Fsym{a}{c}{r}{=}$ is connected (Lemma~\ref{lem-connected}), it follows that $G_{r,s} \cap \ol{\Fsym{a}{c}{r}{=} \setminus G_{r,s}}$ is non-empty. If $z$ belongs to this last intersection, then by the continuity of $D$ and the definition of $G_{r,s}$, we must have $D(c,z)-D(c,b)=s$. As a result, with $Q$ as in~\eqref{eqn-pos-leb-set}, 
\eqbn
\bigcup_{ r \in R} (\{r\} \times J_r) \subset Q . 
\eqen
Thus, with $\mathrm{Leb}$ denoting the Lebesgue measure in $\RR^2$, we have
\begin{equation}
  \label{eq:38}
  \mathrm{Leb}(Q)\geq \mathrm{Leb}(\bigcup_{ r \in R} (\{r\} \times J_r))=\int_R ( (D(w_r,c)-D(w_r,b)) - (D(z_r,c)-D(z_r,b)))dr>0,
\end{equation}
where in the above, we have used that the integral of a strictly positive function (see \eqref{eqn-non-constant-diff}) over the interval $R$ (which is nontrivial by \Cref{lem-good-pt}) must be strictly positive. This completes the proof. 
\end{proof}
 
\subsection*{II. Total geodesic overlap for a countable dense set implies that $X$ is a tree}

In order to complete the proof of Theorem \ref{3-star_points_general}, we will need to prove one more intermediate result. Recall the assumption \ref{it:tree} in Theorem \ref{3-star_points_general}, where it is assumed that the space $(X,D)$ is not a tree\footnote{Note that a tree refers to a space wherein there exists a unique path (necessarily the geodesic) between any two points.}. The goal now is to show that if $X$ is not a tree, then there is no countable dense subset $A$ such that geodesics between pairs of points in $A$ are unique and $A^3$ completely avoids $\Tsym{X}{D}$.

\begin{prop} \label{prop-tree}
Let $(X,D)$ be a boundedly compact geodesic metric space. Assume that there is a countable dense set $A\subset X$ with the following properties.
\begin{enumerate}
\item For any $a,b\in A$, there is a unique $D$-geodesic $\Gamma_{a,b}$.
\item $A^3\cap\Tsym{X}{D}=\emptyset$.
\end{enumerate}  
Then $(X,D)$ is a tree.
\end{prop}

Here is an outline of the proof of Proposition~\ref{prop-tree}. First, for any finite collection of points $z_0,\dots,z_n\in A$, we can use assumption on the set $A$ to show that the union of the $D$-geodesics between pairs of points in $\{z_0,\dots,z_n\}$ is a tree (the technical tool to achieve the above is Lemma~\ref{lem-geo-conc}). Further, it is known (see \cite{epw-real-tree}) that the local Gromov-Hausdorff limit of trees is a tree (Lemma~\ref{lem-tree-conv}). Thus, due to the density of $A$, as $n\rightarrow \infty$, the trees corresponding to $\{z_0,\dots,z_n\}$ must converge to $(X,D)$ in the local Gromov-Hausdorff sense (see Definition~\ref{def-gh-loc}), and thus $(X,D)$ itself must be a tree. 

In order to be self-contained, we now define local Gromov-Hausdorff convergence. The following definition appears, in a more general form, in~\cite[Section 1.3.2]{gwynne-miller-uihpq}; see also~\cite[Section 8.1]{bbi-metric-geometry}. 

\begin{defn}[Local Gromov-Hausdorff convergence] \label{def-gh-loc}
Recall that for non-empty compact subsets $A$ and $B$ of a metric space $(X,D)$, we define the \textbf{Hausdorff distance} between them by
\eqbn
d_{H}^{X}(A,B)=\max\{\sup_{a\in A}d_{B}(a),\sup_{b\in B}d_{A}(b)\},
\eqen
where for $V \subset X$ we denote $d_{V}(z) = \inf_{v \in V}D(v,z)$. We will write $\BB d^{\op{GH}}$ for the \textbf{Gromov-Hausdorff distance} on compact metric spaces with a marked point, defined by
\eqbn
\BB d^{\op{GH}}((X,D_{X},x),(Y,D_{Y},y)) = \inf_{Z,\phi,\psi}\max\{d_{H}^{Z}\left(\phi(X),\psi(Y)\right),D_{Z}\left(\phi(x),\psi(y)\right)\},
\eqen
where $\phi$ and $\psi$ range over all isometric embeddings of $(X,D_{X},x), (Y,D_{Y},y)$ into a metric space $(Z,D_{Z})$. 

Now for a metric space with a marked point $(X,D,x)$ and $r > 0$, let $\frk B_r(X,D,x)$ be the metric space consisting of the closed $D$-ball of radius $r$ centered at $x$, equipped with the restriction of $D$ and the marked point $x$. We say that a sequence of boundedly compact pointed geodesic metric spaces $(X_n,D_n,x_n)$ converge in the \textbf{local Gromov-Hausdorff sense} to $(X,D,x)$ if for each $R>0$, we have
\eqbn
\lim_{n\rightarrow \infty}\BB d^{\op{GH}} \left(  \frk B_R(X_n,D_n,x_n)  , \frk B_R(X,D,x)  \right) =0.
\eqen
\end{defn}

The following lemma is essentially proven in~\cite{epw-real-tree} and will be very useful for us.

\begin{lem}[\cite{{epw-real-tree}}] \label{lem-tree-conv}
Let $\{(T_n,D_n,x_n)\}_{n\geq 1}$ be a sequence of trees, each with a marked point. 
If $(T_n,D_n,x_n) \to (T,D,x)$ in the local Gromov-Hausdorff sense, then $(T,D,x)$ is a tree.
\end{lem}
\begin{proof}
The analogous statement for compact trees, with the Gromov-Hausdorff distance in place of the local Gromov-Hausdorff distance, is~\cite[Lemma 2.1]{epw-real-tree}. The statement for the local Gromov-Hausdorff distance follows easily from the fact that $(T,D,x)$ is a tree if and only if $\frk B_r(T,D,x)  $ is a tree for each $r >0$. 
\end{proof}

Shortly, we will take $T_n$ to be the set consisting of all $D$-geodesics joining points in $\{z_0,\dots,z_n\}\subseteq A$. In order to use the above convergence result, we will first need to establish that the set $T_n$ itself is a tree, and this will be done by using the following easy lemma.
\begin{lem} \label{lem-geo-conc}
Assume that we are in the setting of Proposition~\ref{prop-tree}. 
Let $a,b,c , d \in A$ and let $\Gamma_{a,b}$ (resp.\ $\Gamma_{c,d}$) be the unique $D$-geodesic from $a$ to $b$ (resp.\ $c$ to $d$).
Assume that $\Gamma_{a,b} \cap \Gamma_{c,d} \not=\emptyset$. Let $ \tau$ be the smallest time $t$ for which $\Gamma_{c,d}(t) \in \Gamma_{a,b}   $ and let $\sigma$ be the smallest time such that $\Gamma_{a,b}(\sigma) = \Gamma_{c,d}( \tau)$. Then the $D$-geodesic from $c$ to $a$ (resp.\ $c$ to $b$) is the concatenation of $\Gamma_{c,d}|_{[0, \tau]}$ followed by the time reversal of $\Gamma_{a,b}|_{[0,\sigma]}$ (resp.\ followed by $\Gamma_{a,b}|_{[\sigma,D(a,b)]}$). 
\end{lem}

\begin{figure}[t]
\begin{center}
\includegraphics[width=0.5\textwidth]{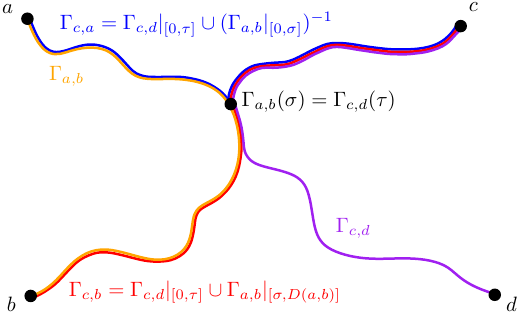}  
\caption{\label{fig:lem-geo-concat}Illustration of the proof of \Cref{lem-geo-conc}. The $D$-geodesic $\Gamma_{c,a}$ from $c$ to $a$, drawn in blue, has $D$-length $\sigma + \tau$; it is the concatenation of $\Gamma_{c,d}|_{[0, \tau]}$ (purple, length $\tau)$ followed by the time reversal of $\Gamma_{a,b}|_{[0,\sigma]}$ (orange, length $\sigma$). Similarly, the $D$-geodesic $\Gamma_{c,b}$ from $c$ to $b$, drawn in red, has $D$-length $\tau + (D(a,b)-\sigma)$; it is the concatenation of $\Gamma_{c,d}|_{[0, \tau]}$ (purple, length $\tau$) followed by $\Gamma_{a,b}|_{[\sigma,D(a,b)]}$ (orange, length $D(a,b)-\sigma$).}
\end{center}
\end{figure}

\begin{proof}
Let $\Gamma_{c,a}$ and $\Gamma_{c,b}$ be the unique $D$-geodesics from $c$ to $a$ and $c$ to $b$ respectively (see \Cref{fig:lem-geo-concat} for an illustration). By our assumption that $A^3\cap \Tsym{X}{D}=\emptyset$, the point $\Gamma_{c,d}(\tau) = \Gamma_{a,b}(\sigma)$ lies in the union $\Gamma_{c,a}\cup \Gamma_{c,b}$. Assume without loss of generality that $\Gamma_{c,d}(\tau) \in \Gamma_{c,a}   $. 
Since $\Gamma_{c,a}$ is a $D$-geodesic and $\Gamma_{c,d}(\tau) = \Gamma_{a,b}(\sigma)$, this implies that
\begin{equation}
  \label{eq:22}
D(a,c) = D(\Gamma_{c,d}(\tau) , a)   + D(\Gamma_{c,d}(\tau) , c) = D(\Gamma_{a,b}(\sigma) , a) + \tau = \sigma+\tau .  
\end{equation}
Therefore, the concatenation of $\Gamma_{c,d}|_{[0, \tau]}$ followed by the time reversal of $\Gamma_{a,b}|_{[0,\sigma]}$ is a $D$-geodesic from $c$ to $a$ since it $D$-length is precisely $\sigma+\tau$. By the uniqueness of $D$-geodesics between points of $A$, this concatenation is equal to $\Gamma_{c,a}$. 
In particular, $\sigma$ is the largest time in $[0,D(a,b)]$ for which $\Gamma_{a,b}(\sigma) \in \Gamma_{c,a} $. 
By our assumption on $A$, it follows that $\Gamma_{a,b}\lvert_{(\sigma,D(a,b)]} \subset \Gamma_{c,b}$. Since (the range of) $\Gamma_{c,b}$ is closed, we have $\Gamma_{a,b}(\sigma) \in \Gamma_{c,b}$, i.e., $\Gamma_{c,d}(\tau) \in \Gamma_{c,b}$. Then, precisely the same argument as in \eqref{eq:22} above shows that $\Gamma_{c,b}$ is the concatenation of $\Gamma_{c,d}|_{[0, \tau]}$ followed by $\Gamma_{a,b}|_{[\sigma,D(a,b)]}$.
\end{proof} 

With the above preparation in hand, we are now ready to complete the proof of Proposition \ref{prop-tree}.

\begin{proof}[Proof of Proposition~\ref{prop-tree}]
  Enumerate the points of $A$ by $A = \{z_n\}_{n\geq 0}$; for each $i,j\geq 0$, we will consider the (unique) geodesic $\Gamma_{z_i,z_j}$. Our goal now is to inductively show that for all $n\geq 0$, the set $T_n=\bigcup_{0\leq i,j\leq n}\Gamma_{z_i,z_j}$ (see \Cref{fig-geo-tree}) satisfies the following properties.
\begin{enumerate}[label=(\roman*)]
\item The set $T_n$ is a tree in the sense that for each $z,w\in T_n$, there is a unique path in $T_n$ from $z$ to $w$. Note that this path must necessarily be simple. \label{item-tree}
\item For each $z,w\in T_n$, the unique simple path in $T_n$ from $z$ to $w$ is, in fact, the unique $D$-geodesic $\Gamma_{z,w}$. \label{item-tree-geo}
\end{enumerate}
To begin, note that, note that $T_1=\Gamma_{z_1,z_0}$ trivially satisfies the above properties. Inductively, assume that $n \geq 2$ and $T_{n-1}$ has been established to be a tree. Let $\tau_n$ be the smallest time $t$ for which $\Gamma_{z_n,z_0}(t) \in T_{n-1}$. Let $H_n := T_{n-1} \cup \Gamma_{z_n,z_0}\lvert_{[0,\tau_n]}$. By the inductive hypothesis, $H_n$ is the union of a tree and a simple path which intersects it at a single point, and is therefore a tree. In fact, since the unique path from $\Gamma_{z_n,z_0}(\tau_n)$ to $z_0$ in $T_n$ is also the unique $D$-geodesic from $\Gamma_{z_n,z_0}(\tau_n)$ to $z_0$, we must have $H_n=T_{n-1}\cup \Gamma_{z_n,z_0}$. We now show that, in fact, $H_n=T_n$. 

Indeed, by applying Lemma~\ref{lem-geo-conc} with $a = c=z_{0} , b = z_{n} , d = z_{i}$, we obtain that for each $0\leq i \leq n-1$, $\Gamma_{z_n,z_i}$ is exactly the concatenation of $\Gamma_{z_n,z_0}\lvert_{[0,\tau_n]}$ with the unique geodesic $\Gamma_{\Gamma_{z_n,z_0}(\tau_n), z_i}\subseteq T_{n-1}$. This implies that $H_n=T_n$, thereby yielding that $T_n$ is a tree and therefore proving item \ref{item-tree}.

The task now is to prove item \ref{item-tree-geo}. For this, we first note that if $z,w\in T_{n-1}\subseteq T_n$, then we already know that $\Gamma_{z,w}$ is unique and lies in $T_{n-1}$. Also, if we have $z,w\in \Gamma_{z_n,z_0}\lvert_{[0,\tau_n]}$, then by the uniqueness of $\Gamma_{z_n,z_0}$, we know that $\Gamma_{z,w}$ is unique and lies in $\Gamma_{z_n,z_0}\lvert_{[0,\tau_n]}$. Thus, without loss of generality, we can assume that $z\in \Gamma_{z_n,z_0}\lvert_{[0,\tau_n]}$ and $w\in T_{n-1}$. Our plan now is to establish that
\begin{equation}
  \label{eq:39}
  T_{n-1}= \bigcup_{0\leq i\leq n-1} \Gamma_{\Gamma_{z_n,z_0}(\tau_n), z_i}.
\end{equation}
Indeed, if this were true, then any $w\in T_{n-1}$ must satisfy $w\in \Gamma_{\Gamma_{z_n,z_0}(\tau_n), z_i}$ for some $0\leq i\leq n-1$. We can then consider the unique geodesic $\Gamma_{z_n,z_i}$, and this uniqueness implies that $\Gamma_{z,w}$ is unique and is contained in $\Gamma_{z_n,z_0}\lvert_{[0,\tau_n]}\cup \Gamma_{\Gamma_{z_n,z_0}(\tau_n), z_i}\subseteq T_{n}$.

In order to show \eqref{eq:39}, we simply establish that for any point $\alpha\in T_{n-1}$, we must have
\begin{equation}
  \label{eq:31}
  T_{n-1}=\bigcup_{0\leq i\leq n-1}\Gamma_{\alpha, z_i}.
\end{equation}
To see this, take a generic $x\in T_{n-1}$; we need to establish that $x\in \bigcup_{0\leq i\leq n-1}\Gamma_{\alpha, z_i}$. We might as well assume that $x\notin \{\alpha ,z_0,\dots,z_{n-1}\}$. Thus, $x$ is not a leaf of the tree $T_{n-1}$, and as a result, the set $T_{n-1}\setminus \{x\}$ must be disconnected. Thus, $x$ must lie on at least one of the paths $\Gamma_{\alpha,z_i}$ for $0\leq i \leq n-1$ as otherwise the set $T_{n-1}\setminus \{x\}$ would necessarily be connected. This shows \eqref{eq:31} and \eqref{eq:39}, thereby concluding the induction. Thus, we have established that the set $T_n$ satisfies \ref{item-tree} and \ref{item-tree-geo} for all $n$. 

\begin{figure}[t]
\begin{center}
\includegraphics[width=0.7\textwidth]{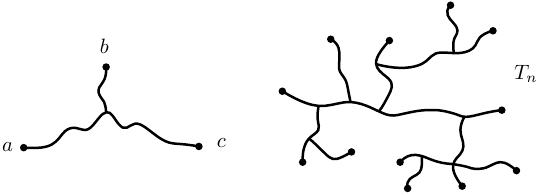}  
\caption{\label{fig-geo-tree}  
\textbf{Left:} The $D$-geodesics from $a$ to $b$, $b$ to $c$, and $c$ to $a$ for three points $a,b,c \in A$. 
\textbf{Right:} The tree $T_n$ formed by the $D$-geodesics between points in a finite subset of $A$, as in the proof of Proposition~\ref{prop-tree}. 
}
\end{center}
\end{figure}

We are now ready to establish that $(X,D)$ must be a tree and for this, we will use Lemma \ref{lem-tree-conv}. Since the pointed space $(T_n, D\lvert_{T_n\times T_n},z_0)$ is a tree, by \Cref{lem-tree-conv} it suffices to establish that it converges to $(X,D,z_0)$ in the local Gromov-Hausdorff sense. That is, for any fixed radius $R>0$, we need to establish that
\begin{equation}
  \label{eq:32}
  \lim_{n\rightarrow \infty}\BB d^{\op{GH}} \left(  \frk B_R(T_n, D\lvert_{T_n\times T_n},z_0)  , \frk B_R(X,D,z_0)  \right) =0
\end{equation}
Now, since $X$ is boundedly compact, the closed metric ball $\overline{\cB_R(z_0)}$ is necessarily compact. As a result, since $A$ is dense in $X$, for any fixed $\epsilon>0$ and all $n$ large enough, $\overline{\cB_R(z_0)}\cap T_n$ contains an $\epsilon$-net of points for the set $\overline{\cB_R(z_0)}$. This establishes \eqref{eq:32}, thereby completing the proof. 
\end{proof}

Having proved \Cref{prop-pos-leb} and \Cref{prop-tree}, we are now finally ready to complete the proof of \Cref{3-star_points_general}.

\begin{proof}[Proof of \Cref{3-star_points_general}]
  We begin by proving the first part of the theorem. Let $f$ be as in the statement of the theorem and let $(a,b,c)\in \Tsym{X}{D}$. Then, by Proposition \ref{prop-pos-leb}, we have $\mathrm{Leb}(f(X))>0$ and therefore $\dim(f(X))=2$. By now applying Theorem \ref{intersection_universal_bounds} with $V=\RR^2$, we have $\dim( \hyperlink{def-non-constancy-set}{S_{f,\RR^2}})\geq 2$.

  We now move on to the second part of the theorem, where we additionally assume that the geodesic metric space $(X,D)$ satisfies assumptions \ref{it:bddcpt}, \ref{it:tree}, and \ref{it:as}. Consider the dense subset $A \subset X$ given by assumption \ref{it:as}. Since $(X,D)$ is boundedly compact (assumption \ref{it:bddcpt}), Proposition \ref{prop-tree} applies; then since there exists a unique geodesic between each pair of points in $A$ (assumption \ref{it:as}) and since $X$ is not a tree (assumption \ref{it:tree}), there must exist a triple $(a,b,c)\in A^3\cap \Tsym{X}{D}$. Further, since $A^2\subseteq \Hsym{X}{D}$ (assumption \ref{it:as}), we have $\{(a,b),(a,c),(b,c)\}\subseteq \Hsym{X}{D}$. We now consider the set $K_{a,b,c}$ of all $(r,s) \in \mathbb{R}^{2}$ for which at least one of the following occurs:
\begin{enumerate}[label=(\Alph*)]
\item \label{it:r1}A geodesic started from $a$ traces $\Fsym{a}{b}{r}{=}\cup \Fsym{a}{c}{s}{=} $ for a nontrivial time interval.
\item \label{it:s1}A geodesic started from $b$ traces $\Fsym{b}{c}{s-r}{=}\cup \Fsym{a}{b}{r}{=}$ for a nontrivial time interval.
\item \label{it:t1} A geodesic started from $c$ traces $\Fsym{b}{c}{s-r}{=}\cup \Fsym{a}{c}{s}{=}$ for a nontrivial time interval.
\end{enumerate}
Then, since $\{(a,b),(a,c),(b,c)\} \subseteq \Hsym{X}{D}$, the two-dimensional Lebesgue measure of $K_{a,b,c}$ must necessarily be equal to zero. Note the special property $\Fsym{a}{b}{r}{=}\cap \Fsym{a}{c}{s}{=}\subseteq \Fsym{b}{c}{s-r}{=}$; this will be useful shortly.

\begin{figure}[t]
\begin{center}
\includegraphics[width=\textwidth]{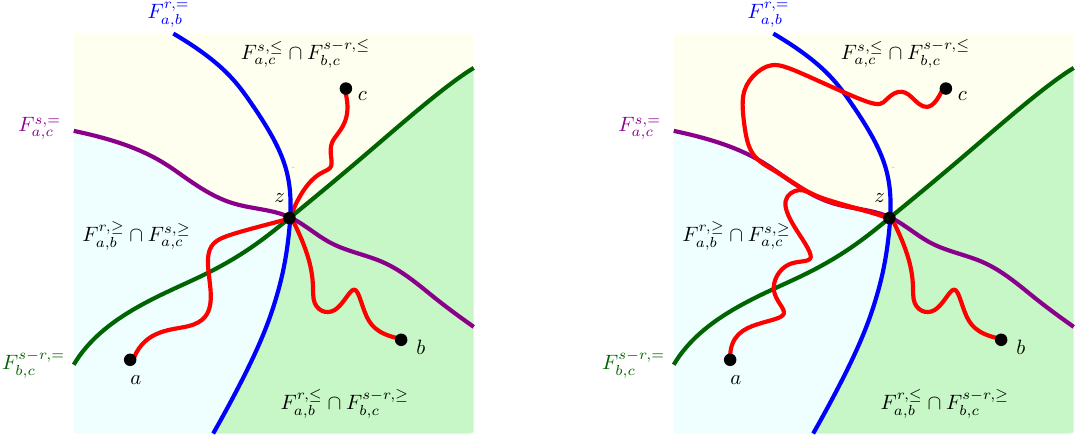}\caption{\textbf{Left:} The point $z \in \Fsym{a}{b}{r}{=}\cap \Fsym{a}{c}{s}{=}$ is a 3-star point for $D$. The $D$-geodesic (red) from $z$ to $a$ does not trace $\Fsym{a}{b}{r}{=}\cup \Fsym{a}{c}{s}{=}$ for a nontrivial interval, and the analogous statements hold for the geodesics from $z$ to $b$ and to $c$. \textbf{Right:} Here the geodesic from $z$ to $a$ traces $\Fsym{a}{b}{r}{=}\cup \Fsym{a}{c}{s}{=}$ and the geodesic from $z$ to $c$ traces $\Fsym{b}{c}{s-r}{=} \cup \Fsym{a}{b}{r}{=}$. Hence, although $z \in \Fsym{a}{b}{r}{=}\cap \Fsym{a}{c}{s}{=}$, we cannot deduce that it is a 3-star point. For simplicity we have depicted $\Fsym{a}{b}{r}{=}\cap \Fsym{a}{c}{s}{=}$ as a single point, but this is not necessarily the case; it may even have nonempty interior (see \Cref{fig-nonempty-interior}).}
\label{fig-voronoi}
\end{center}
\end{figure}

  Observe that, by Lemma \ref{lem-inside}, for $z \in \Fsym{a}{b}{r}{=} \cap \Fsym{a}{c}{s}{=}\subseteq \Fsym{b}{c}{s-r}{=}$, all geodesics $\Gamma_{z,a}, \Gamma_{z,b},\Gamma_{z,c}$ necessarily satisfy
  \begin{equation}
    \label{eq:40}
    \Gamma_{z,a}\subseteq \Fsym{a}{b}{r}{\geq }\cap \Fsym{a}{c}{s}{\geq}, \Gamma_{z,b}\subseteq \Fsym{b}{c}{s-r}{\geq} \cap \Fsym{a}{b}{r}{\leq}, \Gamma_{z,c}\subseteq \Fsym{b}{c}{s-r}{\leq}\cap \Fsym{a}{c}{s}{\leq}.
  \end{equation}
  Also, note that the following sets are all mutually disjoint:
  \begin{equation}
    \label{eq:41}
    \Fsym{a}{b}{r}{>}\cap \Fsym{a}{c}{s}{>}, \Fsym{b}{c}{s-r}{>} \cap \Fsym{a}{b}{r}{<}, \Fsym{b}{c}{s-r}{<}\cap \Fsym{a}{c}{s}{<}.
  \end{equation}
As a result, any $z \in \Fsym{a}{b}{r}{=} \cap \Fsym{a}{c}{s}{=}\subseteq \Fsym{b}{c}{s-r}{=}$ necessarily admits 3 almost disjoint geodesics $\Gamma_{z,a}, \Gamma_{z,b}, \Gamma_{z,c}$ and is therefore a 3-star point for $D$, \textit{unless} one of the following occurs (see \Cref{fig-voronoi}):
\begin{enumerate}
\item \label{it:r}A geodesic from $z$ to $a$ traces $\Fsym{a}{b}{r}{=}\cup \Fsym{a}{c}{s}{=} $ for a nontrivial time interval.
\item \label{it:s}A geodesic from $z$ to $b$ traces $\Fsym{b}{c}{s-r}{=}\cup \Fsym{a}{b}{r}{=}$ for a nontrivial time interval.
\item \label{it:t} A geodesic from $z$ to $c$ traces $\Fsym{b}{c}{s-r}{=}\cup \Fsym{a}{c}{s}{=}$ for a nontrivial time interval.
\end{enumerate}
In other words, for any $(r,s)\in \RR^2\setminus K_{a,b,c}$, every $z\in \Fsym{a}{b}{r}{=}\cap \Fsym{a}{c}{s}{=}$ must satisfy $z\in T_{\textup{3-star}}$.
However, as just discussed,
\begin{equation}
  \label{eq:19}
  \mathrm{Leb}(K_{a,b,c})=0.
\end{equation}

With the above ingredients in hand, we are now ready to complete the proof. Recall that $f(x) = (D(x,a) - D(x,b),D(x,a) - D(x,c))$ and that $\mathrm{Leb}(f(X))>2$. By using \eqref{eq:19}, we obtain that
\begin{equation}
  \label{eq:20}
  \mathrm{Leb}(f(X)\cap (\RR^2\setminus K_{a,b,c}))>0.
\end{equation}
We now apply \Cref{intersection_universal_bounds} with $n=2$ and $V=\RR^2\setminus K_{a,b,c}$ to conclude that we have $\dim(f^{-1}(\RR^2\setminus K_{a,b,c}))\geq 2$.
Now, as discussed just before \eqref{eq:19}, any $z\in f^{-1}(\RR^2\setminus K_{a,b,c})$ necessarily satisfies $z\in T_{\textup{3-star}}$. This completes the proof.
\end{proof}

\subsection{Lower bound for 2-stars on a connected set}
\label{sec:lower-bound-2}

\begin{proof}[Proof of Theorem \ref{boundary_thm}]
As in the statement of the result, let $a,b\in A$ be distinct points and let $f\colon A\rightarrow \RR$ be defined by $f(x)=D(x,a)-D(x,b)$. Note that $f(a)\neq f(b)$ since $f(a)=-D(a,b)<0$ while $f(b)=D(a,b)>0$. Since $f$, as defined above, is continuous and $A$ is connected, $f(A)$ must be connected as well, and as a result, we must have
\begin{equation}
  \label{eq:17}
  \mathrm{Leb}(f(A))>0,
\end{equation}
which in particular implies that $\dim f(A)=1$. Now, since $X$ is $\sigma$-compact and $A$ is closed, $(A, D\lvert_{A\times A})$ is also a $\sigma$-compact metric space. Thus, we can now apply Theorem \ref{intersection_universal_bounds} to the metric space given by $(A, D\lvert_{A\times A})$ and the function $f$ as defined above. As a consequence, we immediately obtain that $\dim(S_{f,\RR})\geq 1$.
  
We now come to the second part of the theorem and additionally assume that $\Hsym{X}{D}\cap A^2\neq \emptyset$. Fix $(a,b)\in \Hsym{X}{D}\cap A^2$ and note that we must have $a\neq b$ by the definition of $\Hsym{X}{D}$. Now, as a consequence of Lemma \ref{lem-inside}, note that for any $r\in \RR$ and any $z\in \Fsym{a}{b}{r}{=}$, we are guaranteed that $z$ admits almost disjoint geodesics $\Gamma_{z,a},\Gamma_{z,b}$ unless the following holds:
\begin{enumerate}
\item \label{it:timeint} A geodesic from $z$ to $a$ traces $\Fsym{a}{b}{r}{=}$ for a nontrivial time interval.
\end{enumerate}
In particular, note that if there exist almost disjoint geodesics $\Gamma_{z,a},\Gamma_{z,b}$, then we must necessarily have $z\in T_{\textup{2-star}}$.

Since $(a,b)\in \Hsym{X}{D}$, the set $K_{a,b}$ of $r\in \RR$ for which condition \ref{it:timeint} above holds satisfies $\mathrm{Leb}(K_{a,b})=0$. By combining this with \eqref{eq:17}, we have
\begin{equation}
  \label{eq:18}
  \mathrm{Leb}(f(A)\cap (\RR \setminus K_{a,b}))=\mathrm{Leb}(f(A))>0.
\end{equation}
We now apply Theorem \ref{intersection_universal_bounds} with $n=1$ and $V=(\RR \setminus K_{a,b})$ to obtain that $\dim(f^{-1}(\RR\setminus K_{a,b}))\geq 1$. Finally, as we discussed, any $z\in f^{-1}(\RR\setminus K_{a,b})\subseteq A$ must satisfy $z\in T_{\textup{2-star}}$ as well. This completes the proof.
\end{proof}

\begin{remark} It is natural to ask whether the techniques of \Cref{3-star-section} and \Cref{sec:lower-bound-2} can be extended to prove lower bounds on the set of $k$-stars for $k\geq 4$. To do so, one would need to establish an analogue of the ``positive Lebesgue measure condition'' \Cref{prop-pos-leb} for a Lipschitz function $f$ whose non-constancy set corresponds to potential $k$-star points. A central obstacle is that the collection of Voronoi-type cells $V_{1}, ... V_{k}$ defined as
\begin{equation*}
V_j=\bigcap_{i\neq j}\Fsym{u_{j}}{u_{i}}{r_{i}-r_{j}}{\geq}
\end{equation*}
(where $r_{2},...,r_{k} \in \mathbb{R}$ are ``offset parameters'' and $r_{1} = 0$) for $k > 3$ points may exhibit increasingly complicated geometric configurations, in contrast to the relatively tractable cases for $k \leq 3$ which we have analyzed so far. In fact, it is unclear in general whether there will exist any point which is incident to four distinct Voronoi cells, let alone a positive measure set of such points. More concretely, it is unclear what conditions on the underlying geodesic space would guarantee the existence of parameters $r_{2}, r_{3}, ..., r_{k}$ such that
\begin{equation}\label{eq:voronoi-condition-nonempty}
\bigcap_{i=1}^{k}V_{i} = \bigcap_{i=2}^{k}\Fsym{u_{1}}{u_{i}}{r_{i}}{=} \neq \emptyset,
\end{equation}
which is the type of condition one would hope to use to deduce the existence of a $k$-star. Because of this difficulty, it seems likely that new techniques will be required to obtain analogous bounds for larger $k$.
\end{remark}

\section{Dimension lower bounds specific to LQG}\label{section-lqg-bounds}

\subsection{Frostman-type estimates}

In this section we will lay the groundwork for the proof of \Cref{lqg_lower_bounds_thm}, which establishes lower bounds for the LQG and Euclidean Hausdorff dimensions of the non-constancy set of a locally Lipschitz function for the LQG metric. This will be done by obtaining Frostman-type estimates, as is usual for dimension lower bounds. Interestingly, we will first obtain a Frostman result (see \Cref{radii-product-lower-bound-lemma}) wherein Euclidean and LQG diameters are juxtaposed into a single estimate. Thereafter, we will leverage this estimate along with an analysis of thick points of the GFF to obtain separate Frostman estimates and thus dimension lower bounds corresponding to each of the Euclidean (Proposition \ref{euclidean-dim-lemma}) and LQG (Proposition \ref{lqg-dim-lemma}) metrics.

Throughout, we will consider a function $f$ which is locally Lipschitz with respect to the LQG metric, along with its associated non-constancy set $\Ssym{f}{V}$ for $V \subseteq f(\mathbb{C})$. So long as $f$ is not just a constant function, we can choose random constants $a < b$ such that each of the sets $\{x : f(x) < a\}$ and $\{x : f(x) > b\}$ contains a Euclidean ball, denoted by $A$ and $B$ respectively (here we are using that almost surely the subcritical LQG metric induces the Euclidean topology), which we will choose so that both have the same radius. We now introduce a definition which will be useful shortly.

\begin{defn}[The collection $\{L_{\theta}\}$]\label{def-l-theta} Given two disjoint Euclidean balls $A$ and $B$ with the same radius, we define the collection of line segments $\{L_{\theta}\}$ as follows. Parametrize the boundaries of $\partial A$ and $\partial B$ by $\theta \in (-\pi, \pi]$ at constant Euclidean speed, choosing the point $\theta = 0$ on each of $\partial A$ and $\partial B$ to be the unique closest point to the other ball. Then we define $L_{\theta}$ to be the (Euclidean) line segment connecting the points $\theta$ on each of $\partial A$ and $\partial B$, and we take $\{L_{\theta}\}$ to be the collection of all such line segments for $\theta \in \left(-\frac{\pi}{4},\frac{\pi}{4}\right)$. (See \Cref{fig-demo} for an illustration of this definition.)    
\end{defn}

\begin{figure}[t]
\begin{center}
\includegraphics[width=0.7\textwidth]{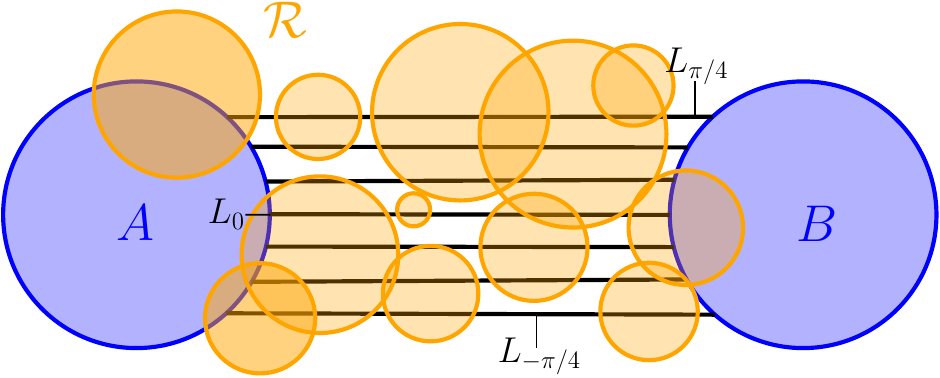}  
\caption{\label{fig-demo}  
Illustration of the setup for \Cref{lqg_lower_bounds_thm}, in the case where we lower-bound the Euclidean dimension of $\Ssym{f}{V}$. $A$ and $B$ are Euclidean balls of equal radius chosen so that $f(x) < a$ on $A$ and $f(x) > b$ on $B$, for random constants $a$ and $b$. $\{L_{\theta}\}$ is the family of line segments defined in \Cref{def-l-theta}. $\mathcal{R}$ is an arbitrary cover of the non-constancy set $\Ssym{f}{V}$ by Euclidean balls. The case where we are lower-bounding the LQG dimension is identical, except that $\mathcal{R}$ will be an arbitrary cover of $\Ssym{f}{V}$ by LQG balls rather than Euclidean ones.}
\end{center}
\end{figure}

The following lemma, whose proof will utilize the above definition, is the crucial Frostman-style estimate which will eventually be used to obtain the LQG and Euclidean dimension lower bounds stated in Theorem \ref{lqg_lower_bounds_thm}.

\begin{lem}\label{radii-product-lower-bound-lemma}
  Let $h$ be a whole-plane GFF and let $\gamma \in (0,2)$. For a convex open set $X\subseteq \CC$, let $f: X \rightarrow \mathbb{R}$ be a random non-constant locally Lipschitz function with respect to $D_{h}$. Then almost surely, for every set $V\subseteq f(X)$, and every $\beta<\dim(V)$, there exists $C_{V,\beta}>0$ such that for every countable cover $\cR$ of $\Ssym{f}{V}$ by subsets of $\mathbb{C}$, we have
\begin{equation}\label{radii-product-lower-bound}
    C_{V,\beta} \leq \sum_{R \in \cR}\operatorname{diam}_{D_{h}}(R)^\beta\cdot\operatorname{diam}_{\mathrm{Euc}}(R),
\end{equation}
\end{lem}

\begin{proof}
  Let $a < b$ be (random) constants chosen such that
  \begin{enumerate}
  \item $\{x : f(x) < a\}$ contains an Euclidean ball $A$ and $\{x : f(x) > b\}$ contains an Euclidean ball $B$, both with the same radius.
  \item \label{it:ab} $\dim(V\cap [a,b])>\beta$.
  \end{enumerate}
That such constants $a,b$ indeed exist almost surely can be checked by using the countable stability of Hausdorff dimension and the continuity of $f$, along with the fact that almost surely the LQG metric induces the Euclidean topology. For this choice of $A$ and $B$, consider the collection of line segments $\{L_{\theta}\}$ (see \Cref{def-l-theta}), which we note lies in $X$ since $X$ is convex. See \Cref{fig-demo} for an illustration of the setup. 

We now fix a big Euclidean ball $B_K(0)$ such that $B_K(0)\supset A \cup B \cup \{L_{\theta}\}$. For our purposes, it will suffice to show that the bound in \eqref{radii-product-lower-bound} holds for all countable covers $\cR$ of $\Ssym{f}{V}\cap \{L_{\theta}\}$, since any cover of $\Ssym{f}{V}$ is also a cover of $\Ssym{f}{V} \cap \{L_{\theta}\}$. In fact, we can additionally restrict to countable covers $\cR$ for which every $R\in \cR$ satisfies
  \begin{equation}
    \label{eq:53}
    R\cap B_{K}(0)\neq \emptyset\textrm{ and } R\subseteq B_{2K}(0).
  \end{equation}
This restriction is justified as follows. Since $\Ssym{f}{V} \cap \{L_{\theta}\} \subset B_{K}(0)$, we may discard any $R \in \cR$ with $R\cap B_{K}(0) = \emptyset$. Additionally, if both $R\cap B_{K}(0)$ and $R\cap B_{2K}(0)^c$ are non-empty for some $R\in \cR$, then we would already have
  \begin{equation}
    \label{eq:52}
    \operatorname{diam}_{D_{h}}(R)^\beta\cdot\operatorname{diam}_{\mathrm{Euc}}(R)\geq D_h(\partial B_{K}(0), \partial B_{2K}(0))^\beta \cdot \mathrm{dist}(\partial B_{K}(0), \partial B_{2K}(0))>0.
  \end{equation}

  By the continuity of $f$, for every $\theta\in (-\pi/4,\pi/4)$, we necessarily have $[a,b]\subseteq f(L_\theta)$, and thereby $V\cap [a,b]\subseteq f(L_\theta)$. Further, we note the basic inequality
  \begin{equation}
    \label{eq:48}
    \Ssym{f\lvert_{L_\theta}}{V}\subseteq \Ssym{f}{V}.
  \end{equation}
  Therefore, since $\cR$ covers $\Ssym{f}{V}$, it also covers $\Ssym{f\lvert_{L_\theta}}{V}$ for every $\theta\in (-\pi/4,\pi/4)$. Thus, by Lemma \ref{lem:conn}, we have
    \begin{equation}
    \label{eq:49}
    f(\Ssym{f\lvert_{L_\theta}}{V})= f( L_{\theta})\cap V\supseteq V\cap [a,b].
  \end{equation}
  Thus, since $\cR$ covers $\Ssym{f\lvert_{L_\theta}}{V}$ for every $\theta$, the family $\{f(R)\}_{R\in \cR,R \cap L_{\theta} \neq \emptyset}$ covers $V \cap [a,b]$ for every $\theta$. Since $f\lvert_{B_{2K}(0)}$ is $C_K$-Lipschitz with respect to $D_{h}$ for a constant depending on $K$, we must have
  \begin{equation}\label{crucial_inequality}
\sum\limits_{R \in \cR : R \cap L_{\theta} \neq \emptyset}\text{diam}_{D_{h}}(R )^{\beta} \geq C_K^{-\beta}\sum\limits_{R \in \cR : R \cap L_{\theta} \neq \emptyset}\text{diam}_{\mathrm{Euc}}(f(R) )^{\beta} \geq C_K^{-\beta} C^*_{V,\beta},
\end{equation}
where $C^*_{V,\beta}$ can be defined by
\begin{equation}
  \label{eq:51}
  C^*_{V,\beta}\coloneqq \inf_{\textrm{covers } \cR' \textrm{ of } V\cap [a,b]} \sum_{R\in \cR'} \mathrm{diam}_{\mathrm{Euc}}(R)^\beta,
\end{equation}
which we note is strictly positive since by \eqref{it:ab}, we have $\beta<\mathrm{dim}(V\cap [a,b])$.

The bound \eqref{crucial_inequality} holds for each choice of $\theta \in \left(-\frac{\pi}{4},\frac{\pi}{4}\right)$, so integrating over $\theta \in (-\frac{\pi}{4}, \frac{\pi}{4})$ (i.e., along the ``innermost quadrant" of the boundaries of each of $A$ and $B$), we have: 

\begin{align}
  \label{eq:fubini}
\frac{\pi}{2}C_K^{-\beta} C^*_{V,\beta} &\leq \int_{-\frac{\pi}{4}}^{\frac{\pi}{4}}\sum\limits_{R \in \cR : R \cap L_{\theta} \neq \emptyset}\text{diam}_{D_{h}}(R)^\beta d\theta \nonumber \\
&= \sum\limits_{R \in \cR}\int_{\theta \in \left(-\frac{\pi}{4}, \frac{\pi}{4}\right) : R \cap L_{\theta} \neq \emptyset} \text{diam}_{D_{h}}(R)^\beta d\theta \quad \text{ (Fubini) } \nonumber \\
&= \sum\limits_{R \in \cR}\text{diam}_{D_{h}}(R)^\beta \cdot \text{Leb}\{\theta \in \left(-\frac{\pi}{4}, \frac{\pi}{4}\right): R \cap L_{\theta} \neq \emptyset\} \nonumber \\
&\lesssim \sum\limits_{R \in \cR} \text{diam}_{D_{h}}(R)^\beta \cdot \text{diam}_{\mathrm{Euc}}(R).
\end{align}
To obtain the last line above, we have simply used that there exists a universal constant $c > 0$ such that for any two $\theta_1,\theta_2\in (-\pi/4,\pi/4)$ and points $x_1\in L_{\theta_1}, x_2\in L_{\theta_2}$, we have $|x_1-x_2|\geq c|\theta_1-\theta_2|$. For this step, it is necessary to work with $\theta$ bounded away from $-\pi/2,\pi/2$, and for this reason we considered $\theta\in (-\pi/4,\pi/4)$ as opposed to $\theta\in (-\pi/2,\pi/2)$. The proof is now completed by defining $C_{V,\beta}=c\frac{\pi}{2}C_K^{-\beta} C^*_{V,\beta}$ and invoking \eqref{eq:fubini}.
\end{proof}

To lower-bound the LQG Hausdorff dimension, we will want to convert the bound \eqref{radii-product-lower-bound} given by \Cref{radii-product-lower-bound-lemma} to a bound of the form
\begin{equation}\label{euclidean-ineq}
C \leq \sum\limits_{R \in \mathcal{R}} \text{diam}_{D_{h}}(R)^{\beta+p}
\end{equation}
which holds for the same constant $C$ across all arbitrary covers, yielding a lower bound of $\beta + p$ for the LQG dimension of $\Ssym{f}{V}$. Analogously, to lower-bound the Euclidean Hausdorff dimension, we will want to convert the bound \eqref{radii-product-lower-bound} to a bound of the form
\begin{equation}\label{lqg-ineq}
C' \leq \sum\limits_{R \in \mathcal{R}} \text{diam}_{\mathrm{Euc}}(R)^{\beta p'+1},
\end{equation}
yielding a lower bound of $\beta p' + 1$ for the Euclidean dimension of $\Ssym{f}{V}$. Since \Cref{radii-product-lower-bound-lemma} applies for every $\beta < \dim(V)$, this will then give lower bounds of $\dim(V) + p$ for the LQG dimension of $\Ssym{f}{V}$ and $\dim(V) \cdot p' + 1$ for its Euclidean dimension.

Although we could directly substitute in the H\"older exponents for which the LQG metric is known to be almost surely locally H\"older continuous with respect to Euclidean distance and vice versa (\Cref{lem:estholder}) in order to obtain bounds of the form \eqref{euclidean-ineq} and \eqref{lqg-ineq}, it turns out to be the case that we can get much better bounds by partitioning the elements of our cover according to the approximate local ``thickness'' of the underlying Gaussian free field, and applying different estimates to relate Euclidean and LQG diameters depending on this thickness for a given element.

Let us now introduce a definition which will be very useful to us throughout the proofs to follow.

\begin{defn}(Sufficient covering family)\label{def-sufficient-collection}
Let $(X,D)$ be a metric space and consider a subset $K \subseteq X$. Let $\mathcal{A}$ be a collection of sets which satisfies the following property for some $\delta > 0$, $M > 0$, and $C > 0$: for every set $U \subset K$ with $0 < \operatorname{diam}(U) \leq \delta$, there exists a collection $\{A_{i}\} \subset \mathcal{A}$ with cardinality at most $M$ such that $\operatorname{diam}(A_{i}) \leq C\operatorname{diam}(U)$ for each $i$, and such that $U \subset \cup_{i=1}^{M}A_{i}$. Then we call $\mathcal{A}$ a \textbf{sufficient covering family} for $K$.
\end{defn}

We observe that if $A$ is a sufficient covering family for $K$, then to lower-bound the Hausdorff dimension of an arbitrary subset of $K$, it suffices to consider only covers consisting of elements of $A$, rather than all arbitrary covers. This follows directly from the definition of Hausdorff dimension. So, as we go on to compute our dimension lower bounds, we will make use of convenient choices of sufficient covering families to simplify our arguments.

We will now proceed to obtain inequalities of the form \eqref{euclidean-ineq} and \eqref{lqg-ineq}, which will then lead us in \Cref{lqg_lower_bounds_thm} to lower bounds on the Euclidean and LQG Hausdorff dimensions of $\Ssym{f}{V}$. 

\subsection{Euclidean lower bound}
\label{subsec:Euc}

The goal of this section is to obtain the following lower bound on the Euclidean dimension of sets which satisfy a particular ``juxtaposed'' Frostman-style estimate \eqref{eq:78}.

\begin{prop}\label{euclidean-dim-lemma}
  Let $h$ be a whole-plane GFF and let $\gamma \in (0,2)$. Then almost surely, the following holds for all bounded Borel sets $H\subseteq \CC$ and $\beta>0$. If there exists a constant $C_{H,\beta}>0$ such that for all countable covers $\cR$ of $H$, we have
  \begin{equation}
    \label{eq:78}
    \sum_{R\in \cR}\mathrm{diam}_{D_h}(R)^\beta \cdot \mathrm{diam}_{\mathrm{Euc}}(R)>C_{H,
    \beta},
  \end{equation}
then we necessarily have $\dim_{\mathrm{Euc}}(H)\geq 1+\DEuc{\gamma}{\beta}$.
\end{prop}
  
In order to prove the above result, we will first prove a series of small lemmas, and will thereafter combine them to complete the proof. The broad strategy is as follows: we will group up the elements $R$ of an arbitrary countable cover $\cR$ of $H$ according to the exponent $\alpha$ such that $\operatorname{diam}_{D_{h}}(R)\sim\operatorname{diam}_{\mathrm{Euc}}(R)^\alpha$, where intuitively, the above exponent $\alpha$ is small if the GFF $h$ is ``thick'' around $R$. Standard estimates for the GFF will allow us to estimate the number of $R\in \cR$ which are atypically thick, in the sense that the corresponding $\alpha$ is small. Thus, after choosing an appropriate ``thickness threshold'' $\alpha_*$, we can decompose $\cR=\cR_{\downarrow}\cup \cR_{\uparrow}$, where the first set $\cR_{\downarrow}$ consists of atypically thick $R$ (in the sense that $\alpha=\alpha(R)<\alpha_*$) and has very small cardinality, thereby ensuring that the sum $\sum_{R \in \cR_\downarrow}\operatorname{diam}_{D_{h}}(R)^\beta\cdot\operatorname{diam}_{\mathrm{Euc}}(R)$ is smaller than $C_{H,\beta}/2$. Given \eqref{eq:78}, this will allow us to conclude that
\begin{displaymath}
  C_{H,\beta}/2\leq \sum_{R \in \cR_\uparrow}\operatorname{diam}_{D_{h}}(R)^\beta\cdot\operatorname{diam}_{\mathrm{Euc}}(R)\lesssim \sum_{R \in \cR_\uparrow} \operatorname{diam}_{\mathrm{Euc}}(R)^{1+\beta \alpha_*} \leq \sum_{R \in \cR} \operatorname{diam}_{\mathrm{Euc}}(R)^{1+\beta \alpha_*}
\end{displaymath}
for all countable covers $\mathcal{R}$ of $H$, thereby yielding a lower bound of $1+\beta \alpha_*$ for its Euclidean dimension. As we will soon see, the optimal value of $\alpha_*$ will arise as a root of the quadratic equation
\begin{equation}
  \label{eq:quad}
g(\alpha) = -1+\beta\alpha+(Q-\alpha/\xi)^{2}/2.
\end{equation}

We now set the stage to formally execute the above strategy. First, note that because $\mathbb{C} = \bigcup_{n \in \mathbb{N}} \overline{B_{n}(0)}$ and because our sets of interest $H$ are bounded, it will suffice to fix $N \in \mathbb{N}$ and prove the desired bound under the additional assumption that $H \subset \overline{B_{N}(0)}$. We will work with a fixed such choice of $\overline{B_{N}(0)}$ throughout this section. 

We will also work throughout this section with the following collections of sets $\Gsymlink{k}$. Fix some positive integer $k_{0}$ (we will choose this later on to be very large). For each integer $k > k_{0}$, we define the collection $\Gsymlink{k}$ to consist of the collection of Euclidean balls with radius $2^{-k}$ which are centered at the points $\{z_{j}\}$ of the lattice $2^{-k}\mathbb{Z}^{2} \cap \overline{B_{N+1}(0)}$. That is, we define
\begin{equation}\label{def:G-k}
    \cG_k
    :=
    \left\{
    B_{2^{-k}}(z_j) \colon z_j \in 2^{-k}\mathbb Z^2 \cap \overline{B_{N+1}(0)}
    \right\}.
\end{equation}

When proving the desired lower bound on $\dim_{\mathrm{Euc}}(H)$, rather than working with all possible covers $\cR$ of $H$, we will instead only consider covers whose elements belong to the family
\begin{equation}\label{def:G-geqk0}
    \cG_{>k_0}:=\bigcup_{k>k_0}\Gsymlink{k}.
\end{equation}
This restriction is justified by the following lemma:

\begin{lem}
  \label{lem:it-is-suf}
  For $\overline{B_N(0)}$ equipped with the Euclidean metric, the family $\Ggeqsymlink{k_{0}}$ is a sufficient covering family (in the sense of Definition \ref{def-sufficient-collection}) for any $k_0\in \mathbbm{N}$.
\end{lem}

\begin{proof}
    Take e.g. $\delta = 2^{-k_{0}}$, $M = 4$, and $C=4$ in \Cref{def-sufficient-collection}.
\end{proof}

We now make a few more useful definitions.
\begin{enumerate}
\item With $\xi, Q$ defined as in \eqref{def:xi-and-Q}, we set
\begin{equation}\label{def:holder-exps}
    \underline{\alpha} := \xi\left(Q-2\right) \text{ and } \overline{\alpha} := \xi\left(Q+2\right).
\end{equation} Note that these are the almost sure local H\"older exponents relating the Euclidean and LQG metrics (\Cref{lem:estholder}).
\item\label{scaling-exp-theta} We also define a ``mesh size" $\delta > 0$ and discretize the continuous interval $[\underline{\alpha},\overline{\alpha}]$ into
\begin{equation}\label{def:A-delta}
    A_{\delta} := [\underline{\alpha}-\delta,\overline{\alpha}+\delta]\cap \delta\ZZ.
\end{equation}
The elements of $\Adeltasymlink{\delta}$ will represent approximate scaling exponents $\theta$ relating a set's LQG and Euclidean diameters.
 \item For $\alpha\in \Adeltasymlink{\delta}$, define 
 \begin{equation}\label{def:Gkalpha}
     \cG_{k}^{\alpha} := \{U \in \Gsymlink{k} : \text{diam}_{D_{h}}(U) \in (2^{-\alpha k},2^{-(\alpha-\delta) k}] \}.
 \end{equation}
Intuitively, the family $\Gkasymlink{k}{\alpha}$ consists of Euclidean balls of radius $2^{-k}$ for which the scaling exponent $\theta$ (see \Cref{scaling-exp-theta}) is approximately $\alpha$.
\end{enumerate}

\noindent The following easy decomposition of covers is at the heart of the argument.

\begin{lem}
  \label{lem:split}
Almost surely, for all $k_0$ large enough, and any cover $\cR$ of a set $H\subseteq \overline{B_N(0)}$ with $\cR$ consisting only of elements of $\Ggeqsymlink{k_0}=\bigcup_{k>k_0}\Gsymlink{k}$, and any $\alpha_*\in \Adeltasymlink{\delta}$, we have
  \begin{equation}
    \label{eq:43}
    \sum_{R\in \cR}\operatorname{diam}_{D_{h}}(R)^\beta\cdot\operatorname{diam}_{\mathrm{Euc}}(R)\leq \sum_{R\in \cR}\operatorname{diam}_{\mathrm{Euc}}(R)^{1+\beta (\alpha_*-\delta)} + \sum_{\alpha\in \Adeltasymlink{\delta}, \alpha<\alpha_*} \sum_{k>k_0}\# \Gkasymlink{k}{\alpha} (2^{-k})^{1+\beta (\alpha-\delta)}.
  \end{equation}
\end{lem}
\begin{proof}
  To begin, we note that by H\"older continuity (\Cref{lem:estholder}), if $k_0$ is chosen to be large enough, then we can ensure that each $R\in \cR$ belongs to $\Gkasymlink{k}{\alpha(R)}$ for precisely one $k > k_{0}$ and one $\alpha(R)\in \Adeltasymlink{\delta}$. We note that this $\alpha(R)$ must satisfy precisely one of  $\alpha(R)\geq \alpha_*$ and $\alpha(R)< \alpha_*$. Thus, we can write
  \begin{align}
    \label{eq:44}
    &\sum_{R\in \cR}\operatorname{diam}_{D_{h}}(R)^\beta\cdot\operatorname{diam}_{\mathrm{Euc}}(R)\nonumber\\
    &= \sum_{R\in \cR, \alpha(R)\geq   \alpha_* }\operatorname{diam}_{D_{h}}(R)^\beta\cdot\operatorname{diam}_{\mathrm{Euc}}(R)+ \sum_{R\in \cR, \alpha(R)<  \alpha_*}\operatorname{diam}_{D_{h}}(R)^\beta\cdot\operatorname{diam}_{\mathrm{Euc}}(R)\nonumber\\
    &\leq \sum_{R\in \cR, \alpha(R)\geq  \alpha_* }\operatorname{diam}_{\mathrm{Euc}}(R)^{1+\beta(\alpha(R)-\delta)} + \sum_{R\in \cR, \alpha(R)< \alpha_* }\operatorname{diam}_{\mathrm{Euc}}(R)^{1+\beta (\alpha(R)-\delta)}\nonumber\\
    &\leq \sum_{R\in \cR}\operatorname{diam}_{\mathrm{Euc}}(R)^{1+\beta (\alpha_* -\delta)} + \sum_{\alpha\in \Adeltasymlink{\delta}, \alpha<\alpha_*} \sum_{k>k_0}\# \Gkasymlink{k}{\alpha} (2^{-k})^{1+\beta(\alpha -\delta)},
  \end{align}
 where to obtain the first term in the last line above, we have used that every $R\in \cR\subseteq \Ggeqsymlink{k_0}$ has $\operatorname{diam}_{\mathrm{Euc}}(R)\leq 1$ since $k_0$ is taken to be large.
\end{proof}

We will adjust the constant $\alpha_{*}$ in the above lemma such the latter term on the right hand side in \eqref{eq:43} is small. For this, we will need the following bound on the cardinalities $\#\Gkasymlink{k}{\alpha}$.

\begin{lem}\label{tk-alpha-bound-lemma} Almost surely, there exists a constant $C > 0$ such that for all $\alpha\in \Adeltasymlink{\delta}$ satisfying $\alpha< \xi Q$ and all $k$ large enough, we have
\begin{equation}\label{tkalpha_bound}
    \#\Gkasymlink{k}{\alpha} < C\cdot 2^{2k}\cdot 2^{-k\left((Q-\alpha/\xi)^{2}/2-\delta\right)}.
\end{equation}
\end{lem}

\begin{proof}
 Since $\alpha< \xi Q$, by \Cref{lem:est1}, we obtain that for all $z \in \overline{B_{N+1}(0)}$, and all sufficiently small $\epsilon > 0$, we have
\begin{align*}
    \mathbb{P}\left[\mathrm{diam}_{D_{h}}(B_{\epsilon}(z)) > \epsilon^{\alpha} \right] \leq \epsilon^{(Q-\alpha/\xi)^{2}/2}, \
\end{align*}
so that for all sufficiently large $k$, for every element $R \in \Gsymlink{k}$, we have
\begin{align*}
    \mathbb{P}\left[R \in \Gkasymlink{k}{\alpha}\right] \leq 2^{-k\left(Q-\alpha/\xi\right)^{2}/2}.
\end{align*}
Note that the set $\Gsymlink{k}$ contains at most $O(1)2^{2k}$ elements. Thus, for all $k$ sufficiently large, there exists a constant $C > 0$ such that
\begin{equation}\label{tkalpha_exp_bound}
    \mathbb{E}[\#\Gkasymlink{k}{\alpha}] \leq C\cdot 2^{2k}\cdot 2^{-k(Q-\alpha/\xi)^{2}/2}.
\end{equation}
Thus for all $k$ sufficiently large, defining the event $X_{k,\alpha} = \{\#\Gkasymlink{k}{\alpha} \geq C\cdot 2^{2k}\cdot 2^{-k\left((Q-\alpha/\xi)^{2}/2-\delta\right)}\}$, by Markov's inequality and \eqref{tkalpha_exp_bound} we have
\begin{align*}
    \mathbb{P}(X_{k,\alpha}) \leq \frac{\mathbb{E}[\#\Gkasymlink{k}{\alpha}]}{C\cdot 2^{2k}\cdot 2^{-k\left((Q-\alpha/\xi)^{2}/2-\delta\right)}} \leq 2^{-k\delta},
\end{align*}
and by the union bound,
\begin{align*}
\mathbb{P}(X_{k,\alpha} \text{ occurs for some } \alpha \in \Adeltasymlink{\delta}) \leq \#\Adeltasymlink{\delta}\cdot 2^{-k\delta}.
\end{align*}
Thus, by the Borel-Cantelli lemma, the $\limsup$ of the events $\{X_{k,\alpha} \text{ occurs for some } \alpha \in \Adeltasymlink{\delta}\}$ must have zero probability. That is, almost surely, for all $k$ large enough and all $\alpha\in \Adeltasymlink{\delta}$, we have
\begin{equation}\label{tkalpha_bound2}
    \#\Gkasymlink{k}{\alpha} < C\cdot 2^{2k}\cdot 2^{-k\left((Q-\alpha/\xi)^{2}/2-\delta\right)},
  \end{equation}
  and this completes the proof.
\end{proof}

We now use the above lemma to show that for an appropriate choice of $\alpha_*$, the second term on the right hand side in \eqref{eq:43} is necessarily small.

\begin{lem}\label{euclidean-sum-control-lemma}
  Fix an $\alpha_*\in \Adeltasymlink{\delta}\cap (0,\xi Q)$ such that $\min g( (-\infty,\alpha_*]\cap \Adeltasymlink{\delta})>(2+\beta)\delta$, where $g$ is the quadratic defined in \eqref{eq:quad}. Then, almost surely, for any $\epsilon>0$, we can choose $k_0$ to be large enough such that
  \begin{equation}\label{negative-set-lower-bound}
    \sum_{\alpha\in \Adeltasymlink{\delta}, \alpha<\alpha_*} \sum_{k>k_0}\# \Gkasymlink{k}{\alpha} (2^{-k})^{1+\beta (\alpha-\delta)}<\epsilon.
    \end{equation}
\end{lem}

\begin{proof}
  By Lemma \ref{tk-alpha-bound-lemma}, for all large $k_0$, we have
  \begin{align}
    \label{eq:45}
    \sum_{\alpha\in \Adeltasymlink{\delta}, \alpha<\alpha_*} \sum_{k>k_0}\# \Gkasymlink{k}{\alpha} (2^{-k})^{1+\beta (\alpha-\delta)} & \leq \sum_{\alpha\in \Adeltasymlink{\delta}, \alpha<\alpha_*} \sum_{k>k_0}C'\cdot 2^{2k}\cdot 2^{-k\left((Q-\alpha/\xi)^{2}/2-\delta\right)} (2^{-k})^{1+\beta(\alpha-\delta)}\nonumber\\
    &=\sum_{\alpha\in \Adeltasymlink{\delta}, \alpha<\alpha_*} \sum_{k>k_0}C' (2^{-k})^{g(\alpha)-(1+\beta)\delta}\leq C'\#\Adeltasymlink{\delta} \sum_{k > k_0}2^{-\delta k},\\
  \end{align}
where to obtain the final inequality we have used the assumption $\min g( (-\infty,\alpha_*]\cap \Adeltasymlink{\delta})>(2+\beta)\delta$. It is now easy to see that the expression above must be smaller than $\epsilon$ for all large $k_0$.
\end{proof}

With these lemmas in hand, we are now ready to complete the proof of Proposition \ref{euclidean-dim-lemma}.

\begin{figure}[t]
\begin{center}  
\begin{tikzpicture}[scale=1.3, >=stealth]
  \def\r{2.0}
  \def\aleft{4}
  \def\aright{4}
  \def\xmax{4.0}
  \def\xleft{5.0}
  \def\xaxisright{5.0}
  \def\ymax{3.2}
  \def\xiq{1}
  \def\xad{2.46}
  
  \pgfmathsetmacro{\k}{\ymax/((\xmax^2 - \r^2))}

  \tikzmath{
    function g(\x) {return \k*(\x*\x - \r*\r);};
  }

  \pgfmathsetmacro{\ydelta}{g(-\xad+0.2)}

  \draw[->] (-\xleft,0) -- (\xaxisright,0) node[below right] {$\alpha$};

  \draw[red, line width=2pt, dotted] (-\aleft,0) -- (-\xad,0);
  \draw[green!60!black, line width=2pt, dotted] (-\xad,0) -- (\aright,0);

  \fill (-\r,0) circle (1.4pt);
  \fill ( \r,0) circle (1.4pt);
  \node[below=4pt] at (-\r,0) {$\underline{r}$};
  \node[below=4pt] at ( \r,0) {$\overline{r}$};

  \fill (-\aleft,0) circle (1.4pt);
  \fill ( \aright,0) circle (1.4pt);
  \node[below=4pt] at (-\aleft,0) {$\underline{\alpha}$};
  \node[below=4pt] at ( \aright,0) {$\overline{\alpha}$};

  \fill (-\xiq,0) circle (1.4pt);
  \node[below=4pt] at (-\xiq,0) {$\xi Q$};

  \fill[red] (-\xad,0) circle (1.8pt);
  \node[below=4pt] at (-\xad,0) {$\alpha_{\delta,*}$};

  \draw[dashed, black] (-\xleft,\ydelta) -- (\xaxisright,\ydelta)
  node[pos=0.9, above] {$y=(2+\beta)\delta$};

  \draw[line width=0.8pt, black, domain=-\xmax:\xmax, samples=180]
    plot (\x,{g(\x)});

  \node[black] at (2.95,2) {$g(\alpha)$};
\end{tikzpicture}
\caption{Illustration of the proof setup for \Cref{euclidean-dim-lemma}. For small $\delta > 0$, we define $\alpha_{\delta,*}$ to be the largest value in $\Adeltasymlink{\delta} \cap [\underline{\alpha},\underline{r}]$ such that $g(\alpha_{\delta,*})>(2+\beta)\delta$. The values $\Adeltasymlink{\delta} \cap [\underline{\alpha},\alpha_{\delta,*}]$, shown in red, are approximate scaling exponents relating LQG and Euclidean diameters for ``atypically thick'' elements of $\mathcal{R}$, and they index the latter sum in \eqref{eq:46}. We show that this latter sum is small (in particular, at most $C_{H,\beta}/2$) for sufficiently small $\delta$. The values $\Adeltasymlink{\delta} \cap [\alpha_{\delta,*},\overline{\alpha}]$, shown in green, are approximate scaling exponents for the remaining elements of $\mathcal{R}$, which are not atypically thick.}
\label{fig:g-alpha-graph}
\end{center}
\end{figure}
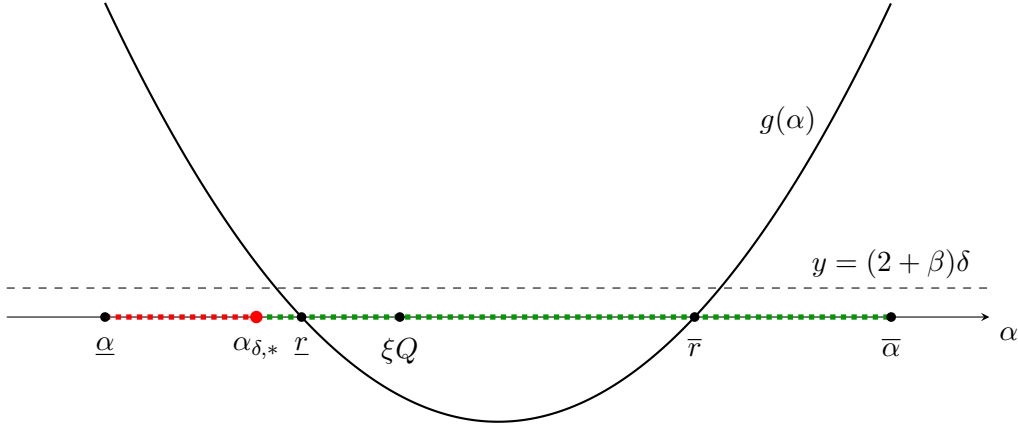

\begin{proof}[Proof of Proposition \ref{euclidean-dim-lemma}]
 Recall that it suffices to fix $N \in \mathbb{N}$ and assume that $H\subseteq \overline{B_N(0)}$. Fix $\delta>0$ and use this to define the set $\Adeltasymlink{\delta}$; we will send $\delta$ to $0$ at the end of the proof. By assumption there exists a constant $C_{H,\beta}>0$ such that for all countable covers $\cR$ of $H$ we have $\sum_{R\in \cR}\mathrm{diam}_{D_h}(R)^\beta \cdot \mathrm{diam}_{\mathrm{Euc}}(R)>C_{H,\beta}$. Combining this with \Cref{lem:split}, we deduce that for all $k_0$ large enough and any cover $\cR$ of $H$ consisting only of elements of $\Ggeqsymlink{k_{0}}=\bigcup_{k>k_0}\Gsymlink{k}$, for any $\alpha\in \Adeltasymlink{\delta}$, we have
  \begin{equation}
    \label{eq:46}
        C_{H,\beta} \leq \sum_{R\in \cR}\operatorname{diam}_{\mathrm{Euc}}(R)^{1+\beta (\alpha_* -\delta)} + \sum_{\alpha\in \Adeltasymlink{\delta}, \alpha< \alpha_*} \sum_{k>k_0}\# \Gkasymlink{k}{\alpha} (2^{-k})^{1+\beta (\alpha-\delta)}.
      \end{equation}
We now consider the quadratic equation $\alpha\mapsto g(\alpha)$ defined in \eqref{eq:quad}. Its roots are given by
\begin{align*}
    \underline{r}, \overline{r} 
    &= \frac{d_{\gamma}(\gamma^{2}+4) - 2\beta\gamma^{2} \pm 2\gamma\sqrt{2d_{\gamma}^{2}-(4\beta+\beta\gamma^{2}) d_{\gamma}+\beta^{2}\gamma^{2}}}{2d_{\gamma}^{2}},   
\end{align*}
and we note that $\beta\underline{r}=\DEuc{\gamma}{\beta}$. By using routine estimates for LQG dimension derived in \cite{gp-lfpp-bounds}, it can be checked that the above roots are real and that $\underline{r}\in (\underline{\alpha},\overline{\alpha})$. For completeness, we include a proof of this in \Cref{euclidean-roots-appendix-lemma} of the Appendix. Further, we note that
\begin{equation}
  \label{eq:59}
  \xi Q>\underline{r}
\end{equation}
(since $\textrm{argmin } g = \xi Q-\beta\xi^2\in [\underline{r},\overline{r}]$, so $\xi Q \geq \underline{r} + \beta\xi^{2} > \underline{r}$). We also note that, since $g(\alpha)$ is quadratic with the coefficient of $\alpha^2$ being positive, $g|_{[\underline{\alpha},\underline{r}]}$ is a strictly decreasing and non-negative continuous function. We may thus define
\begin{equation}
    \alpha_{\delta,*} = \max \left\{\alpha \in \Adeltasymlink{\delta} \cap [\underline{\alpha},\underline{r}] : g(\alpha) > (2+\beta)\delta \right\},
\end{equation}
and this is well-defined for all sufficiently small $\delta > 0$. (See \Cref{fig:g-alpha-graph} for an illustration of the setup.)

Now, as a consequence of \eqref{eq:59}, we know that $\alpha_{\delta,*}< \xi Q$.  Thus, we may invoke Lemma \ref{euclidean-sum-control-lemma} with $\epsilon=C_{H,\beta}/2$ and $\alpha_*=\alpha_{\delta,*}$ to deduce that for all sufficiently large $k_0$, the latter term in \eqref{eq:46} is at most $C_{H,\beta}/2$. Thus, for all sufficiently large $k_0$ and any cover $\cR$ of $H$ consisting only of elements of $\Ggeqsymlink{k_{0}}=\bigcup_{k>k_0}\Gsymlink{k}$, we must have
\begin{equation}
    \label{eq:47a}
    C_{H,\beta}/2 \leq \sum_{R\in \cR}\operatorname{diam}_{\mathrm{Euc}}(R)^{1+\beta (\alpha_{\delta,*}-\delta)}.       
\end{equation}
Since $\Ggeqsymlink{k_0}$ is a sufficient covering family in the sense of Definition \ref{def-sufficient-collection} for any sufficiently large $k_0$ (see Lemma \ref{lem:it-is-suf}), the above implies that $\operatorname{dim}_{\mathrm{Euc}}(H)\geq 1+\beta (\alpha_{\delta,*}-\delta)$. Finally, since $g|_{[\underline{\alpha},\underline{r}]}$ is a strictly decreasing continuous function, we also have $\lim_{\delta\rightarrow 0} \alpha_{\delta,*}=\underline{r}$, thereby yielding that $\operatorname{dim}_{\mathrm{Euc}}(H)\geq 1+\beta \underline{r}=1+\DEuc{\gamma}{\beta}$, which completes the proof. 
\end{proof}

\subsection{LQG lower bound}
\label{subsc:LQG}

Our goal in this section will be to obtain the following lower bound on the LQG dimension of sets which satisfy the same ``juxtaposed'' Frostman-style estimate \eqref{eq:79} as in the previous section.

\begin{prop}\label{lqg-dim-lemma}
Let $h$ be a whole-plane GFF and let $\gamma \in (0,2)$. Then almost surely, the following holds for all bounded Borel sets $H\subseteq \CC$ and $\beta>0$. If there exists a constant $C_{H,\beta}>0$ such that for all countable covers $\cR$ of $H$, we have
  \begin{equation}
    \label{eq:79}
    \sum_{R\in \cR}\mathrm{diam}_{D_h}(R)^\beta \cdot \mathrm{diam}_{\mathrm{Euc}}(R)>C_{H,\beta},
  \end{equation}
  then we necessarily have $\dim_{D_h}(H)\geq \beta+\DLQG{\gamma}{\beta}$.
  \end{prop}

Just as in the previous section, because $\mathbb{C} = \bigcup_{n \in \mathbb{N}} \overline{B_{n}(0)}$ and because our sets of interest $H$ are bounded, it will suffice to fix $N \in \mathbb{N}$ and prove the desired bound under the additional assumption that $H \subset \overline{B_{N}(0)}$; we will work with a fixed such $\overline{B_{N}(0)}$ throughout this section.   
The intuition behind our proof strategy is the same as in the previous section, except that this time, the roles of the Euclidean and LQG metrics will be reversed. Recall that when deriving the bound on $\dim_{\mathrm{Euc}}(H)$ in the previous section, we restricted the elements of our covers to the sufficient covering family $\Ggeqsymlink{k_{0}}$, which were Euclidean balls centered at points of a Euclidean lattice. This time, when deriving the desired bound on $\dim_{D_{h}}(H)$, we will restrict instead to covers consisting of LQG balls associated to an ``LQG lattice.'' More specifically, we will limit ourselves to considering covers which consist of elements of a collection $\Wgeqsymlink{k_0}$ of LQG metric balls centered at Poissonian sprinkled points, which we now define.

To begin, we fix a ``mesh size'' $\delta > 0$ which we will use throughout this section (just as in the Euclidean case before, we will eventually send $\delta$ to zero at the end of the proof). With $\mu_h$ denoting the LQG measure on $\CC$ associated to the GFF $h$, we define 
\begin{equation}\label{def:ppp}
    \Pi_k \sim \mathrm{PPP}\big((2^{-k})^{-d_\gamma-2\delta}\mu_h\big),
\end{equation}
where $\mathrm{PPP}(\lambda)$ denotes a Poisson point process with intensity measure $\lambda$. We now define 
\begin{equation}\label{def:wk}
\cW_{k}:=\{\cB_{2^{-k}}(u): u\in \PPPlink{k+1}\},   
\end{equation}
where $\cB_r(u)$ denotes the LQG metric ball of radius $r$ centered at $u$. We then define the family
\begin{equation}\label{def:W-geqk0}
    \cW_{>k_0}:=\bigcup_{k>k_0}\Wsymlink{k}.
\end{equation}
The following easy lemma guarantees that the collection $\Wsymlink{k}$ can in fact be used to cover $\overline{B_{N}(0)}$.

\begin{lem}\label{lqg-sufficient-collection-covering-lemma}
Almost surely, for all $k$ large enough, the collection $\Wsymlink{k}$ covers $\overline{B_{N}(0)}$.
\end{lem}

\begin{proof}
  In fact, we will prove the following slightly stronger statement which will be useful in the proof of the next lemma: that the collection $\widetilde{\cW_{k}} := \{\cB_{2^{-k-1}}(u): u\in \Pi_{k+1}\}$ covers $\overline{B_{N}(0)}$ for all $k$ large enough. By \Cref{lem:est3}, the Minkowski dimension of $\gamma$-LQG is a.s.\ equal to $d_\gamma$, and therefore almost surely for all sufficiently large $k$ there exists a collection $Z_{k}$ of $\#Z_k \leq C(2^{-k-2})^{-d_{\gamma}-\delta}$ points such that $\cup_{z \in Z_{k}}\cB_{2^{-k-2}}(z)$ covers $\overline{B_{N}(0)}$. Also by \Cref{lem:est3}, we have $\mu_{h}(\cB_{2^{-k-2}}(z)) \geq C'(2^{-k-2})^{d_{\gamma}+\delta}$ for each $z \in Z_{k}$. (Here $C, C' > 0$ are random constants which depend on $N$ and $\delta$ but not on $k$ or $z$.) So for each $z \in Z_{k}$, we have
\begin{align*}
    \mathbb{P}\left(\Pi_{k+1}\cap \cB_{2^{-k-2}}(z) = \emptyset \lvert h\right) &\leq \exp\left(-(2^{-k-1})^{-d_{\gamma}-2\delta}\cdot C'(2^{-k-2})^{d_{\gamma}+\delta}\right) = e^{-c2^{k\delta}},
\end{align*}
for some $c > 0$ independent of $z$ and $k$, and thus by a union bound,
\begin{align}\label{missing-prob-ineq}
  \mathbb{P}\left(\bigcap_{z\in Z_{k}}\{\Pi_{k+1} \cap \cB_{2^{-k-2}}(z)\neq \emptyset\}\lvert h\right) &\geq 1- \#Z_ke^{-c2^{k\delta}}\nonumber\\
  &\geq 1 - C(2^{-k-2})^{-d_{\gamma}-\delta}e^{-c2^{k\delta}}.
\end{align}
Observe that if the event $\{\Pi_{k+1} \cap \cB_{2^{-k-2}}(z)\neq \emptyset\}$ occurs for every $z \in Z_{k}$, then by the triangle inequality $\widetilde{\cW_{k}}$ will cover $\overline{B_{N}(0)}$. This is because $\widetilde{\cW_{k}}$ consists of balls of radius $2^{-k-1}$ centered at the points of $\Pi_{k+1}$, and we already know by the choice of $Z_{k}$ that $\bigcup_{z\in Z_{k}}\cB_{2^{-k-2}}(z)$ covers $\overline{B_{N}(0)}$. Now applying \eqref{missing-prob-ineq} along with the Borel-Cantelli lemma, we obtain the desired result.
\end{proof}

In fact, the proof of the above lemma implies the following stronger result.
\begin{lem}\label{lem:lqg-covering-lemma}
    For $\overline{B_{N}(0)}$ equipped with the restriction of the LQG metric, almost surely, the family $\Wgeqsymlink{k_0}=\bigcup_{k>k_0}\Wsymlink{k}$ is a sufficient covering family for $\overline{B_{N}(0)}$ (in the sense of Definition \ref{def-sufficient-collection}) for all $k_0$ large enough. 
\end{lem}

\begin{proof}
  By the proof of \Cref{lqg-sufficient-collection-covering-lemma}, almost surely there exists $k_{0} \in \mathbb{N}$ such that for all $k \geq k_{0}$, the collection $\widetilde{\cW_{k}} := \{\cB_{2^{-k-1}}(u): u\in \Pi_{k+1}\}$ covers $\overline{B_{N}(0)}$. Consider $U \subseteq \overline{B_{N}(0)}$ satisfying $r := \textrm{diam}_{D_{h}}(U) \leq 2^{-k_{0}-1}$, and set $k := \lfloor \log_2 r^{-1}\rfloor - 1$, so that $k \geq k_{0}$ and $2^{-k-2} < r \leq 2^{-k-1}$. Now take a point $u \in U$. Since $\widetilde{\cW_{k}}$ covers $\overline{B_{N}(0)}$, there exists $z \in \Pi_{k+1}$ such that $u \in \cB_{2^{-k-1}}(z)$. Since $\textrm{diam}_{D_{h}}(U)\leq 2^{-k-1}$, we must also have $U\subseteq \cB_{2^{-k-1}}(u)$, and therefore by the triangle inequality $U\subseteq \cB_{2^{-k}}(z)\in \Wsymlink{k}\subseteq \Wgeqsymlink{k_0}$. Thus, for all $k_0$ large enough, $\Wgeqsymlink{k_0}$ satisfies the definition of a sufficient covering family with $C=4,M=1$.
\end{proof}

  In view of \Cref{lem:lqg-covering-lemma}, to lower bound $\dim_{D_{h}}(H)$, it will suffice to consider only covers of $H$ which consist of elements of $\Wgeqsymlink{k_0}$ for large $k_0$. Fix $\delta > 0$ and recall the constants $\underline{\alpha}, \overline{\alpha}$ (\ref{def:holder-exps}) and the set $\Adeltasymlink{\delta}$ defined in the previous section. We now introduce a few additional definitions. For $\alpha\in \Adeltasymlink{\delta}$, we define
  \begin{equation}
    \label{eq:55}
    \cW_{k}^\alpha=\{W\in \Wsymlink{k} : \mathrm{diam}_{\mathrm{Euc}}(W)\in ( 2^{-k/\alpha}, 2^{-k/(\alpha+\delta)}]\}.
  \end{equation} 
  For later use, we also define the functions
  \begin{align}
    \label{eq:541}
    H(\alpha) &= \frac{1}{2\alpha}\left( \frac{2}{\gamma} - \frac{\gamma}{2} - \frac{\alpha}{\xi}\right)^{2}\\
    \label{eq:54}
    \ell(\alpha)&= \beta + 1/\alpha + H(\alpha) -d_\gamma.
  \end{align}
 
We will now proceed to prove a few preparatory lemmas, which will help us to prove Proposition \ref{lqg-dim-lemma}. First, we have the following analogue of Lemma \ref{lem:split}.

\begin{lem}
  \label{lem:split1}
  Almost surely, for all $k_0$ large enough, and any cover $\cR$ of $H$ which consists only of elements of $\Wgeqsymlink{k_0}$, and any $\alpha_*\in \Adeltasymlink{\delta}$, we have
  \begin{equation}
    \sum_{R\in \cR}\operatorname{diam}_{D_{h}}(R)^\beta\cdot\operatorname{diam}_{\mathrm{Euc}}(R)\leq \sum_{R\in \cR}\operatorname{diam}_{D_h}(R)^{\beta+1/(\alpha_* +\delta)} + \sum_{\alpha\in \Adeltasymlink{\delta}, \alpha> \alpha_*} \sum_{k>k_0}\#\Wkasymlink{k}{\alpha} (2^{-k})^{\beta+ 1/(\alpha+\delta)}.
  \end{equation}
\end{lem}
\begin{proof}
  The proof follows precisely the same steps as the proof of Lemma \ref{lem:split}, and we thus omit it.
\end{proof}

We will also need the following analogue of Lemma \ref{tk-alpha-bound-lemma}; recall the definition of the function $H$ from \eqref{eq:541}.
\begin{lem}\label{wk-alpha-bound-lemma}
Fix $\zeta>0$. For any $\alpha\in \Adeltasymlink{\delta}$ additionally satisfying $\alpha > \xi (Q-\gamma)$, and all $k$ sufficiently large, we have
\begin{align*}
    \#\Wkasymlink{k}{\alpha} \leq 2^{k (d_\gamma+2\delta)}2^{-k\left(H(\alpha)-\zeta\right)}.
\end{align*}
\end{lem}

In order to prove the above, we will require the following standard description of the GFF around a point uniformly sampled from the LQG measure.
\begin{lem}[{{\cite[Lemma A.10]{DMS14}}}]
  \label{lem:biased}
Consider a field $\mathtt{h}$ sampled from the law of $h$ weighted by $\mu_h(\overline{B_{N}(0)})/ \EE \mu_h(\overline{B_{N}(0)})$. Let $\mathfrak{z}$ be uniformly sampled from $\mu_{\mathtt{h}}\lvert_{\overline{B_{N}(0)}}$. Then, conditional on $\mathfrak{z}$, $\mathtt{h}(\cdot)$ is distributed as $h(\cdot)+\gamma G_{\CC}(\cdot,\mathfrak{z})$, where $G_{\CC}(v,w)$ denotes the Green's function \eqref{eq:21}.
\end{lem}

We are now ready to provide the proof of Lemma \ref{wk-alpha-bound-lemma}.

\begin{proof}[Proof of Lemma \ref{wk-alpha-bound-lemma}] 
As in Lemma \ref{lem:biased}, consider a field $\mathtt{h}$ sampled from the law of $h$ weighted by $\mu_h(\overline{B_{N}(0)})/\EE\mu_h(\overline{B_{N}(0)})$. Note that the laws of $h$ and $\mathtt{h}$ are mutually absolutely continuous. As a result, instead of proving the desired bound on the cardinality of $\Wkasymlink{k}{\alpha}$ for all large enough $k$, we can replace $h$ by $\mathtt{h}$ when defining the Poisson point process $\PPPlink{k}$ and the sets $\Wsymlink{k}$, $\Wkasymlink{k}{\alpha}$, and prove the desired bound instead for the analogously defined $\Pi_{k+1,\mathtt{h}}, \cW_{k,\mathtt{h}}, \cW_{k,\mathtt{h}}^\alpha$.

Thus, we wish to estimate the expectation $\EE[\#\cW_{k,\mathtt{h}}^\alpha]$. Recall the definition of $\cW_{k,\mathtt{h}}$ as the collection of LQG metric balls of radius $2^{-k}$ around points of a Poisson process $\Pi_{k+1,\mathtt{h}}$. Using $\mathfrak{z}$ to denote a point uniformly sampled from the measure $\mu_{\mathtt{h}}\lvert_{\overline{B_{N}(0)}}$, we can use the above definition to write

\begin{equation}
  \label{eq:57}
  \EE[\#\cW_{k,\mathtt{h}}^\alpha]\leq \mathbb{E}[\# \cW_{k,\mathtt{h}}] \PP(\mathrm{diam}_{\mathrm{Euc}}(\cB_{2^{-k},\mathtt{h}}(\mathfrak{z}))> 2^{-k/\alpha}).
\end{equation}
Note that if $\mathrm{diam}_{\mathrm{Euc}}(\cB_{2^{-k},\mathtt{h}}(\mathfrak{z}))> 2^{-k/\alpha}$, then at least one point of $\cB_{2^{-k},\mathtt{h}}(\mathfrak{z})$ must lie outside $B_{2^{-k/\alpha-1}}(\mathfrak{z})$. As a result, we have 
\begin{equation}
  \label{eq:58}
\PP(\mathrm{diam}_{\mathrm{Euc}}(\cB_{2^{-k},\mathtt{h}}(\mathfrak{z}))> 2^{-k/\alpha})\leq \PP(D_{\mathtt{h}}(\partial B_{2^{-k/\alpha-2}}(\mathfrak{z}),\partial B_{2^{-k/\alpha-1}}(\mathfrak{z}))\leq 2^{-k}),
\end{equation}
and thus
\begin{equation}\label{eq:combine-57-58}
\EE[\#\cW_{k,\mathtt{h}}^\alpha]\leq \mathbb{E}[\# \cW_{k,\mathtt{h}}] \PP(D_{\mathtt{h}}(\partial B_{2^{-k/\alpha-2}}(\mathfrak{z}),\partial B_{2^{-k/\alpha-1}}(\mathfrak{z}))\leq 2^{-k}).
\end{equation}
Now we can use Lemma \ref{lem:biased} to describe the law of the field $\mathtt{h}$ around $\mathfrak{z}$; using $h$ to denote a whole-plane GFF independent of $\mathfrak{z}$, we can write
\begin{equation}
  \label{eq:60}
  \PP(D_{\mathtt{h}}(\partial B_{2^{-k/\alpha-2}}(\mathfrak{z}),\partial B_{2^{-k/\alpha-1}}(\mathfrak{z}))\leq 2^{-k})= \mathbb{P}\big( D_{h(\cdot)+ \gamma G_{\CC}(\cdot,\mathfrak{z})}(\partial B_{2^{-k/\alpha-2}}(\mathfrak{z}),\partial B_{2^{-k/\alpha-1}}(\mathfrak{z}))\leq 2^{-k}\big).
\end{equation}
For all $\epsilon$ small enough, note that for $x\in B_{\epsilon/2}(\mathfrak{z}))\setminus B_{\epsilon/4}(\mathfrak{z}))$, we have $\gamma G_{\CC}(x,\mathfrak{z})\sim \gamma \log\epsilon^{-1}$. As a consequence, by using Proposition \ref{lem:est2} along with the Weyl scaling of the LQG metric (see (1.8) in \cite{gm-uniqueness}), we obtain that the event $\cE_{\epsilon,A}$ defined by
  \begin{equation}
    \label{eq:56}
   \cE_{\epsilon,A}=\{ A^{-1} \epsilon^{-\gamma \xi } \epsilon^{\xi Q} e^{\xi h_{\epsilon}(\mathfrak{z})}\leq D_{h}\left(\partial B_{\epsilon/4}(\mathfrak{z}),\partial B_{\epsilon/2}(\mathfrak{z})\right)\leq A \epsilon^{-\gamma \xi } \epsilon^{\xi Q} e^{\xi h_{\epsilon}(\mathfrak{z})}\}
 \end{equation}
 has the property that $\PP(\cE_{\epsilon,A}^{c})\rightarrow 0$ superpolynomially as $A\rightarrow \infty$ uniformly in $\epsilon>0$. In view of the above, for any $\eta>0$, we can write
\begin{equation}\label{lqg_radius_est1}
    \mathbb{P}\big(D_{h}\big(\partial B_{\epsilon/4}(\mathfrak{z}),\partial B_{\epsilon/2}(\mathfrak{z})\big) \leq  \epsilon^\alpha \big) \leq \mathbb{P}(\cE_{\epsilon,\epsilon^{-\eta}}^{c})+ \mathbb{P}\big(\epsilon^{\xi (Q-\gamma)+\eta}e^{\xi h_{\epsilon}(\mathfrak{z})}< \epsilon^\alpha \big). \\
\end{equation}
As discussed, for a fixed $\eta>0$, the first term above decays superpolynomially as $\epsilon\rightarrow 0$ and thus we focus on the second term. Since $h$ is independent of $\mathfrak{z}$, we know that $B_t= h_{e^{-t}}(\mathfrak{z})-h_1(\mathfrak{z})$ is a standard Brownian motion (see e.g. \cite[Theorem 1.59]{berestycki-lqg-notes}). As a result, denoting $s_\epsilon = \log\epsilon^{-1}$, we have
\begin{align}
  \label{eq:61}
  \PP\big(\epsilon^{\xi (Q-\gamma)+\eta}e^{\xi h_{\epsilon}(\mathfrak{z})}< \epsilon^\alpha \big)
  &= \PP\big(h_1(\mathfrak{z})+B_{s_\epsilon}< -\xi^{-1}\big((\alpha- \xi(Q-\gamma)-\eta) s_\epsilon\big)\big)\nonumber\\
  &\leq\PP\big(B_{s_\epsilon} <  -\xi^{-1}\big( (\alpha -\xi (Q-\gamma)-2\eta) s_\epsilon\big)\big) + \PP\big(h_1(\mathfrak{z})<- \xi^{-1}\eta s_\epsilon\big).
\end{align}
By our assumption $\alpha>\xi(Q-\gamma)$ in the statement of the lemma, we know that $\alpha -\xi (Q-\gamma)>0$; for the remainder of the proof, we will assume that $\eta$ is taken to be small enough such that $\alpha -\xi(Q-\gamma)-2\eta>0$. We first consider the latter term in \eqref{eq:61} -- note that for any fixed $z\in \overline{B_{N}(0)}$, $h_1(z)$ is a centered Gaussian with variance given by $\int_{x,y\in \partial B_{1}(z)}G_{\mathbb{C}}(x,y)dxdy$. As a result, for all $z\in \overline{B_{N}(0)}$, $\mathrm{Var}(h_1(z))$ is bounded by a positive constant depending only on $N$. Thus, since $\mathfrak{z}$ is a point in $\overline{B_{N}(0)}$ chosen independently of $h$, conditional on $\mathfrak{z}$, $h_1(\mathfrak{z})$ is a centered Gaussian with uniformly bounded variance. As a result, by a standard Gaussian estimate, the term
\begin{displaymath}
  \PP\big(h_1(\mathfrak{z})<- \xi^{-1} \eta s_\epsilon\big)
\end{displaymath}
decays superpolynomially in $\epsilon$ (recall that $s_\epsilon=\log \epsilon^{-1}$).

Thus, since $B_t$ is simply a standard Brownian motion, using $\phi(\epsilon)$ to denote a term decaying superpolynomially as $\epsilon\rightarrow 0$, we have
\begin{align}
  \label{eq:62}
  \mathbb{P}\big(\epsilon^{\xi (Q-\gamma)+\eta}e^{\xi h_{\epsilon}(\mathfrak{z})}< \epsilon^\alpha \big)&\leq  \exp (- \xi^{-2} (\alpha -\xi(Q-\gamma) -2\eta)^2s_\epsilon/2)+\phi(\epsilon)\nonumber\\
  &\leq  2\epsilon^{ \xi^{-2} (\alpha -\xi (Q-\gamma) -2\eta)^2/2},
\end{align}
and combining this with \eqref{lqg_radius_est1}, for all $\epsilon$ small enough we have
\begin{equation}
    \mathbb{P}\big(D_{h}\big(\partial B_{\epsilon/4}(\mathfrak{z}),\partial B_{\epsilon/2}(\mathfrak{z})\big) \leq  \epsilon^\alpha \big) \leq 2\epsilon^{ \xi^{-2} (\alpha -\xi (Q-\gamma) -2\eta)^2/2}.
\end{equation}
Now applying the above bound to \eqref{eq:combine-57-58} with $\epsilon = 2^{-k/\alpha}$ and with sufficiently small $\eta$ (depending on the fixed choice of $\zeta > 0$), we obtain that for all sufficiently large $k$, for some positive constant $C_N$, we have
\begin{align}
  \label{eq:63}
  \EE[\#\cW^\alpha_{k,\mathtt{h}}]&\leq \mathbb{E}[\# \cW_{k,\mathtt{h}}](2^{-k})^{(\xi^{-2} (\alpha-\xi (Q-\gamma))^2/(2\alpha)- \zeta/3) }\nonumber\\
  &= \mathbb{E}[\# \cW_{k,\mathtt{h}}](2^{-k})^{H(\alpha)-\zeta/3}\nonumber\\
  &\leq C_N (2^{-k})^{H(\alpha)-d_\gamma-2\delta -\zeta/3}.                                
\end{align}
(To obtain the last inequality above, we have used that $\EE[\#\cW_{k,\mathtt{h}}]\sim C_N 2^{(d_\gamma+2\delta)k}$, as follows from the definition of it as balls centered around the set $\Pi_{k+1,\mathtt{h}}\cap \overline{B_{N}(0)}$.) By using \eqref{eq:63} along with Markov's inequality, we have
\begin{equation}
  \label{eq:64}
  \PP\big(\#\cW^\alpha_{k,\mathtt{h}}\geq  C_N(2^{-k})^{H(\alpha)-d_\gamma- 2\delta -2\zeta/3}\big)\leq 2^{-\zeta k/3},
\end{equation}
and an application of the Borel-Cantelli lemma now yields that almost surely, we have
\begin{align}
  \label{eq:71}
  \#\cW^\alpha_{k,\mathtt{h}}&<  C_{N}2^{k (d_\gamma+2\delta)}2^{-k\left(H(\alpha)-2\zeta/3\right)}\nonumber\\
  &\leq 2^{k (d_\gamma+2\delta)}2^{-k\left(H(\alpha)-\zeta\right)},
\end{align}
for all sufficiently large $k$ (depending on $N$).
By the mutual absolute continuity of $h$ and $\mathtt{h}$ as discussed at the beginning of the proof, we have thus shown that a.s. for all $k$ large enough, we have
\begin{equation}
  \label{eq:72}
  \#\Wkasymlink{k}{\alpha}< 2^{k (d_\gamma+2\delta)}2^{-k\left(H(\alpha)-\zeta\right)},
\end{equation}
and this completes the proof.
\end{proof}

\begin{lem}\label{lqg-sum-control-lemma-bas}
  Fix an $\alpha_*\in \Adeltasymlink{\delta}\cap (\xi(Q-\gamma),\infty)$ such that $\min \ell ( [\alpha_*,\infty)\cap \Adeltasymlink{\delta})>C'\delta$, where $\ell$ is the function defined in \eqref{eq:54} and $C' > 0$ is a large deterministic constant whose precise value depends only on $\gamma$ and $\beta$. Then, almost surely, for any $\epsilon>0$, we can choose $k_0$ to be large enough such that
  \begin{equation}\label{positive-set-lower-bound}
    \sum_{\alpha\in \Adeltasymlink{\delta}, \alpha> \alpha_*} \sum_{k>k_0}\# \Wkasymlink{k}{\alpha} (2^{-k})^{\beta+ 1/(\alpha+\delta)}<\epsilon.
    \end{equation}
  \end{lem}
  \begin{proof}
    By applying Lemma \ref{wk-alpha-bound-lemma} with $\zeta=\delta$, almost surely, for all sufficiently large $k_0$, we have
  \begin{align}
    \label{eq:45a}
    \sum_{\alpha\in \Adeltasymlink{\delta}, \alpha> \alpha_*} \sum_{k>k_0}\# \Wkasymlink{k}{\alpha} (2^{-k})^{\beta+ 1/(\alpha+\delta)} &\leq \sum_{\alpha\in \Adeltasymlink{\delta}, \alpha> \alpha_*} \sum_{k>k_0} 2^{k (d_\gamma+2\delta)}2^{-k\left(H(\alpha)-\delta\right)} (2^{-k})^{\beta+ 1/(\alpha+\delta)}\nonumber\\
    &=\sum_{\alpha\in \Adeltasymlink{\delta}, \alpha>\alpha_*} \sum_{k>k_0}(2^{-k})^{\ell(\alpha)-3\delta-\delta/[\alpha(\alpha+\delta)]}\leq \#\Adeltasymlink{\delta} \sum_{k>k_0}2^{-\delta k},
  \end{align}
where the final inequality follows because $\alpha > 
\xi(Q-\gamma)$ (which depends only on $\gamma$; see \eqref{def:xi-and-Q}), so the term $\delta/[\alpha(\alpha+\delta)]$ in the exponent is bounded above by $C''\delta$ for some $C'' > 0$ depending only on $\gamma$, and because $C'$ has been chosen large enough. Now the proof is completed by noting that the final expression can be made arbitrarily small by taking $k_0$ large.
  \end{proof}

  We now combine the above preparatory lemmas to give the proof of \Cref{lqg-dim-lemma}.

  \begin{proof}[Proof of Proposition \ref{lqg-dim-lemma}]
   Recall from our earlier discussion that it suffices to fix a closed unit ball $\overline{B_{N}(0)}\subseteq \CC$ and assume that $H\subseteq \overline{B_{N}(0)}$. Fix $\delta>0$ and use this to define the set $\Adeltasymlink{\delta}$; we will send $\delta$ to $0$ at the end of the proof. By the assumptions of the proposition combined with Lemma \ref{lem:split1}, we know that there exists a random $C_{H,\beta}>0$ (independent of $\delta$) such that for all $k_0$ large enough, and any cover $\cR$ of $H$ consisting only of elements of $\Wgeqsymlink{k_0}$, we must have
  \begin{equation}
    \label{eq:lqg-2sum-decomp}
        C_{H,\beta} \leq \sum_{R\in \cR}\operatorname{diam}_{D_{h}}(R)^{\beta+ 1/(\alpha_*+\delta)} + \sum_{\alpha\in \Adeltasymlink{\delta}, \alpha> \alpha_*} \sum_{k>k_0}\# \Wkasymlink{k}{\alpha} (2^{-k})^{\beta + 1/(\alpha+\delta)}.
      \end{equation}
We now consider the equation $\ell(\alpha) = 0$; recall the definition of $\ell$ from \eqref{eq:54}. The solutions to $\ell(\alpha)=0$ are the quantities $0<\underline{r}\leq\overline{r}$ whose reciprocals $0<\overline{r}^{-1}\leq \underline{r}^{-1}$ are given by
\begin{align*}
    \overline{r}^{-1},\underline{r}^{-1}&=
    \frac{2d_{\gamma}(\gamma^{2}+4) - 4\beta\gamma^{2} \pm 4\gamma\sqrt{2d_{\gamma}^{2}-(4\beta+\beta\gamma^{2}) d_{\gamma}+\beta^{2}\gamma^{2}}}{16+\gamma^{4}},
\end{align*}
and we note that $\DLQG{\gamma}{\beta}= \overline{r}^{-1}$. By using routine estimates for LQG dimension derived in \cite{gp-lfpp-bounds}, it can be checked that the above roots are real and that
\begin{equation}
  \label{eq:65}
  \underline{\alpha}< \xi(Q-\gamma)< \overline{r}< \overline{\alpha}.
\end{equation}
It can also be checked that $\ell\lvert_{[\overline{r}, \overline{\alpha}]}$ is a strictly increasing and non-negative continuous function. For completeness, we include proofs of these facts in \Cref{lqg-roots-appendix-lemma} of the Appendix. Thus, we define $\alpha_{\delta,*}$ to be the smallest value in $[\overline{r}, \overline{\alpha}]\cap \Adeltasymlink{\delta}$ such that $\ell (\alpha_{\delta,*})>C'\delta$, where $C' > 0$ is the deterministic constant from the statement of \Cref{lqg-sum-control-lemma-bas}, and we note that such a value does indeed exist as long as $\delta$ is chosen to be small enough.
      
Now, as a consequence of \eqref{eq:65}, note that $\alpha_{\delta,*}> \xi(Q-\gamma)$. Thus, by invoking Lemma \ref{lqg-sum-control-lemma-bas} with $\epsilon=C_{H,\beta}/2$ and with $\alpha_*=\alpha_{\delta,*}$, we know that for all sufficiently large $k_0$, the latter term in \eqref{eq:lqg-2sum-decomp} is at most $C_{H,\beta}/2$. Thus, for all $k_0$ large enough and any cover $\cR$ of $H$ consisting only of elements of $\Wgeqsymlink{k_0}$, we must have
      \begin{equation}
        \label{eq:47}
        C_{H,\beta}/2 \leq \sum_{R\in \cR}\operatorname{diam}_{D_{h}}(R)^{\beta+  (\alpha_{\delta,*}+\delta)^{-1}}.        
      \end{equation}
By \Cref{lem:lqg-covering-lemma}, $\Wgeqsymlink{k_0}$ is a sufficient covering family in the sense of Definition \ref{def-sufficient-collection} for any $k_0$, and so \eqref{eq:47} implies that $\operatorname{dim}_{\mathrm{LQG}}(H)\geq \beta+ (\alpha_{\delta,*}+\delta)^{-1}$. Finally, since $\ell\lvert_{[\overline{r}, \overline{\alpha}]}$ is a strictly increasing continuous function, we also have $\lim_{\delta\rightarrow 0} \alpha_{\delta,*}=\overline{r}$, thereby yielding that $\operatorname{dim}_{\mathrm{LQG}}(H)\geq \beta+ \overline{r}^{-1}=\beta+\DLQG{\gamma}{\beta}$ and completing the proof.
\end{proof}

\subsection{Non-constancy sets in the bulk}

By combining Lemma \ref{radii-product-lower-bound-lemma} with Propositions \ref{euclidean-dim-lemma} and \ref{lqg-dim-lemma}, we can now complete the proof of Theorem \ref{lqg_lower_bounds_thm}.

\begin{proof}[Proof of \Cref{lqg_lower_bounds_thm}]

  Fix $V \subseteq f(\mathbb{C})$ and $\beta < \dim(V)$. By the countable stability of Hausdorff dimension, we may choose an integer $N$ such that, with $K = B_{N}(0)$ and $V_{K} = V \cap f(B_{N}(0))$, $f\lvert_K$ is non-constant and $\dim(V_{K}) > \beta$. Since there are only countably many possible choices of $K$, \Cref{radii-product-lower-bound-lemma} holds almost surely simultaneously for every such $K$. Applying \Cref{radii-product-lower-bound-lemma} to $V_K \subseteq f(K)$, there exists a (random) $C_{V_K,\beta}>0$ such that for any cover of $\Ssym{f\lvert_K}{V_{K}}\subseteq \Ssym{f}{V}$,
  \begin{equation}
    \label{eq:80}
    C_{V_K,\beta} \leq \sum_{R \in \cR}\operatorname{diam}_{D_{h}}(R)^\beta\cdot\operatorname{diam}_{\mathrm{Euc}}(R).
  \end{equation}
  Specifically, note that $\Ssym{f\lvert_K}{V_K}\subseteq K$ is bounded since $K$ is bounded. Now, by applying Propositions \ref{euclidean-dim-lemma} and \ref{lqg-dim-lemma}, we immediately obtain that $\dim_{\mathrm{Euc}}(\Ssym{f\lvert_K}{V_K})\geq 1+\DEuc{\gamma}{\beta}$ and $\dim_{\mathrm{LQG}}(\Ssym{f\lvert_K}{V_K})\geq \beta+\DLQG{\gamma}{\beta}$ and the desired lower bounds on the dimensions of $\Ssym{f}{V}$ follow immediately since $\Ssym{f\lvert_K}{V_K}\subseteq \Ssym{f}{V}$. Finally, since the desired bounds hold for all $\beta < \dim(V)$ and the functions $\DEuc{\gamma}{\beta}$ and $\DLQG{\gamma}{\beta}$ are continuous in $\beta$, the bounds also extend to $\beta = \dim(V)$ by sending $\beta \uparrow \dim(V)$.
\end{proof}

\subsection{Non-constancy sets on the boundary}
\label{sec:restr-bdy}
In this section, we give the proof of \Cref{lqg_line_lower_bounds_thm}, which is an analogue of \Cref{lqg_lower_bounds_thm} for non-constancy sets of Lipschitz functions on the boundary. While the lower bound for the LQG dimension is a simple consequence of Theorem \ref{intersection_universal_bounds}, we will need to convert this to a Euclidean dimension lower bound using LQG-specific techniques. For this, we will first develop the following boundary version of the worst-case KPZ estimate Proposition \ref{prop:worstkpz}. 

\begin{lem}
  \label{bdy-kpz}
  Let $h$ be a Neumann GFF and let $\gamma \in (0,2)$. Then almost surely, for every bounded Borel set $H\subseteq \RR$ with $\mathrm{dim}_{\mathrm{LQG}}(H)> \beta$, we have
  \begin{equation}
    \label{eq:73}
    \dim_{\mathrm{Euc}}(H)\geq \DbKPZ{\gamma}{\beta}.
  \end{equation}
\end{lem}
Our strategy to prove the above will be essentially a simpler version of the proof of \Cref{euclidean-dim-lemma}, and thus we will be brief. First note that because every bounded $H \subset \mathbb{R}$ is contained in $[-N,N]$ for some $N \in \mathbb{N}$, to prove Lemma \ref{bdy-kpz}, it will suffice to prove the desired bound under the additional assumption that $H \subseteq [-N,N]$. We will therefore fix such an $N$ and work with the interval $I = [-N,N]$ throughout the section. We will require the following lemma, which will play the role of the Frostman-style estimate given by the much more involved Lemma \ref{radii-product-lower-bound-lemma}.

\begin{lem}
  \label{frostman-bdy}
Let $h$ be a Neumann GFF and let $\gamma \in (0,2)$. Almost surely, for any Borel set $H\subseteq \mathbb{R}$ satisfying $\dim_{\mathrm{LQG}}(H)> \beta$, there exists a constant $C_{H,\beta}>0$ such that for every countable cover $\cR$ of $H$ by subsets of $\mathbb{R}$, we have
  \begin{equation}
    \label{eq:74}
    C_{H,\beta}\leq \sum_{R\in \cR}\mathrm{diam}_{D_h}(R)^\beta.
  \end{equation}
\end{lem}

\begin{proof}
    This is an immediate consequence of the definition of Hausdorff dimension.
\end{proof}

In order to prove Lemma \ref{bdy-kpz}, the goal now is to convert the above Frostman-type estimate for the $D_h$-metric to one for the Euclidean metric; we will use the same strategy as in Section \ref{subsec:Euc} to do so. To start, we will require the following lemma from \cite{hm-metric-gluing} which relates the LQG metric to the LQG boundary measure for points on the boundary.

\begin{lem}[{{\cite{hm-metric-gluing} Proposition 4.3}}]\label{hughes-miller-boundary-lemma}
Let $h$ be a Neumann GFF and consider the $\gamma$-LQG surface $(\overline{\HH},D_h)$. Fix $\eta>0$. Almost surely, for all $x,y\in I$ with $|x-y|$ being small enough, we have $D_h(x,y)\leq \nu_h([x,y])^{2/d_\gamma-\eta}$.
\end{lem}

We will also need the following lemma bounding the expectation of the LQG boundary measure of an interval of Euclidean width $\epsilon$.

\begin{lem}\label{boundary-measure-expectation-lemma}
    Let $h$ be a Neumann GFF and fix $p\in (0,4/\gamma^2)$. There exist positive constants $C,c$ such that for each $x\in I$ and $\epsilon\in (0,1)$, we have
    \begin{displaymath}
        c\epsilon^{(1+\gamma^2/4)p-\gamma^2p^2/4} \leq \mathbb{E}[ \nu_h([x-\epsilon,x+\epsilon])^{p}]\leq C\epsilon^{(1+\gamma^2/4)p-\gamma^2p^2/4} 
    \end{displaymath}
\end{lem}
\begin{proof}
Since the restriction of a Neumann GFF to $\mathbb{R}$ is a Gaussian free field with covariance $-2\log|x-y| + O(1)$ (see e.g. \cite[Section 6.2]{berestycki-lqg-notes}), the field $h/\sqrt{2}$ has the standard log-correlated normalization. By \cite[Theorem 6.37]{berestycki-lqg-notes}, the boundary measure is given by
\begin{equation}
    \nu_{h}(dx) = e^{\frac{\gamma}{2}h(x)}dx = e^{\frac{\gamma}{\sqrt{2}}(h/\sqrt{2})(x)}dx,
\end{equation}
i.e. the Gaussian multiplicative chaos (GMC) measure associated with $h/\sqrt{2}$ and parameter $\gamma/\sqrt{2}$. The result then follows by applying \cite[Theorem 3.27]{berestycki-lqg-notes} to this GMC measure.
\end{proof}

The above lemma, when combined with a basic Markov inequality argument, yields the following result.
\begin{lem}\label{thick-point-bound-boundary-mn}
 Fix $\alpha\in (0,\xi Q)$ and $\delta > 0$. Then for all $x\in I$ and $\epsilon$ small enough, we have
    \begin{displaymath}
    \PP(\nu_h([x-\epsilon,x+\epsilon])^{2/d_\gamma}\geq \epsilon^\alpha)\leq \epsilon^{(Q-\alpha/\xi)^2/4 - \delta}.
    \end{displaymath}
\end{lem}

\begin{proof}
    By Markov's inequality and Lemma \ref{boundary-measure-expectation-lemma}, there exists a constant $C$ such that for all $p\in (0,4/\gamma^2)$ and all $\epsilon$ small enough, we have
    \begin{align}
         \PP(\nu_h([x-\epsilon,x+\epsilon])\geq \epsilon^\alpha)&\leq  \epsilon^{-\alpha p}\mathbb{E}[ \nu_h([x-\epsilon,x+\epsilon])^{p}] \nonumber\\
         &\leq C\epsilon^{(1+\gamma^2/4)p-\gamma^2p^2/4-\alpha p}.
    \end{align}
    Define $f(p)=(1+\gamma^2/4)p-\gamma^2p^2/4-\alpha p$, and note that $f(p)$ is maximized at $p_*= 1/2 + 2(1-\alpha)/\gamma^2$ and that $f(p_*)=(1+\gamma^2/4-\alpha)^2/\gamma^2=(Q\gamma/2-\alpha)^2/\gamma^2$. So as long as $p_*\in (0,4/\gamma^2)$, or in other words, for all $\alpha\in (\gamma^2/4-1,\gamma^2/4+1)$, we have
    \begin{align}
        \PP(\nu_h([x-\epsilon,x+\epsilon])\geq \epsilon^\alpha)\leq C\epsilon^{(Q\gamma/2-\alpha)^2/\gamma^2}.        
    \end{align}
    Now replacing $\alpha$ by $\alpha/(2/d_\gamma)$, we get the desired inequality. As a result, as long as $\alpha/(2/d_\gamma)\in (\gamma^2/4-1,\gamma^2/4+1)$ or equivalently, as long as $\alpha\in (\xi(\gamma/2-2/\gamma),\xi Q)$, we have the desired inequality
    \begin{align}
      \label{eq:66}
      \PP(\nu_h([x-\epsilon,x+\epsilon])^{2/d_\gamma}\geq \epsilon^\alpha)&\leq C\epsilon^{(Q-\alpha/\xi)^2/4}\nonumber\\
      &\leq \epsilon^{(Q-\alpha/\xi)^2/4-\delta}
    \end{align}
    for small enough $\epsilon$. Finally, we note that since $\gamma\in (0,2)$, we have $(0,\xi Q)\subseteq (\xi(\gamma/2-2/\gamma),\xi Q)$ and this completes the proof.
\end{proof}
  
We now state a simple consequence of the above lemma which will be more convenient for our application.
  \begin{lem}\label{thick-point-bound-boundary}
Fix $\alpha\in (0,\xi Q)$ and $\delta > 0$. Then there exists $\eta_0 > 0$ such that for any choice $\eta\in (0,\eta_0)$, for all $x\in I$ and $\epsilon$ small enough, we have
    \begin{displaymath}
    \PP(\nu_h([x-\epsilon,x+\epsilon])^{2/d_\gamma-\eta}\geq \epsilon^\alpha)\leq \epsilon^{(Q-\alpha/(\xi-\gamma\eta/2))^2/4-\delta}.
    \end{displaymath}
  \end{lem}
  \begin{proof}
  For this, we invoke Lemma \ref{thick-point-bound-boundary-mn} with $\alpha$ therein replaced by $\alpha \times (2/d_\gamma)/(2/d_\gamma-\eta)$ and we choose $\eta$ small enough such this quantity is at most $\xi Q$.
  \end{proof}

As in Section \ref{subsec:Euc}, we now define $\underline{\alpha}=\xi(Q-2), \overline{\alpha}=\xi(Q+2)$, fix a ``mesh size'' $\delta>0$ (which will eventually be sent to zero), and define $\Adeltasymlink{\delta} = [\underline{\alpha}-\delta,\overline{\alpha}+\delta]\cap \delta\ZZ$ as in the previous sections. We also define
  \begin{equation}\label{def:g-tilde}
    \tilde{\cG}_k := \left\{[z_{j}-2^{-k},z_{j}+2^{-k}] \colon z_j \in 2^{-k}\mathbb Z \cap I
    \right\},
  \end{equation}
  and for $\alpha\in \Adeltasymlink{\delta}$, we define
  \begin{equation}\label{def:gka-tilde}
      \tilde{\cG}_{k}^{\alpha} := \{U \in \Gtildesymlink{k} : \text{diam}_{D_{h}}(U) \in (2^{- \alpha k},2^{-(\alpha-\delta) k}] \}.
  \end{equation}
We also define the family
\begin{equation}\label{def:ggeq-tilde}
    \tilde{\cG}_{>k_0}:=\bigcup_{k>k_0}\Gtildesymlink{k}.
\end{equation}  
We will consider the quadratic expression $\tilde{g}$ defined by
  \begin{equation}
  \label{eq:boundary-polynomial}
\tilde{g}(\alpha) = -1+\beta\alpha+(Q-\alpha/\xi)^{2}/4,
\end{equation}
and as with \eqref{eq:quad} in the bulk case, the bound of \Cref{bdy-kpz} will ultimately arise as a root of $\tilde{g}$. By using the same argument as in Lemma \ref{lem:split}, we can obtain the following result.  

\begin{lem}
  \label{lem:split-bd}
Almost surely, for all $k_0$ large enough, and any cover $\cR$ of $H$ which consists only of elements of $\Gtildegeqsymlink{k_0}=\bigcup_{k>k_0}\Gtildesymlink{k}$, and any $\alpha_*\in \Adeltasymlink{\delta}$, we have
  \begin{equation}\label{eq:split-boundary}
    \sum_{R\in \cR}\operatorname{diam}_{D_{h}}(R)^\beta\leq \sum_{R\in \cR}\operatorname{diam}_{\mathrm{Euc}}(R)^{\beta (\alpha_*-\delta)} + \sum_{\alpha\in \Adeltasymlink{\delta}, \alpha<\alpha_*} \sum_{k>k_0}\# \Gkatildesymlink{k}{\alpha} (2^{-k})^{\beta (\alpha-\delta)}.
  \end{equation}
\end{lem}

\begin{proof}
    The proof is essentially identical to that of \Cref{lem:split}, so we omit it.
\end{proof}

Further, by using Lemmas \ref{hughes-miller-boundary-lemma}, \ref{lem:split-bd} and arguing as in Lemma \ref{tk-alpha-bound-lemma}, we have the following.
\begin{lem}\label{tk-alpha-bound-lemma-bdy} There exists a constant $C > 0$ such that for all $\alpha\in \Adeltasymlink{\delta}$ satisfying $\alpha\in (0,\xi Q)$ and all $k$ large enough, we have
\begin{equation}\label{tkalpha_bound_bdy}
    \#\Gkatildesymlink{k}{\alpha} < C\cdot 2^{k}\cdot 2^{-k\left((Q-\alpha/\xi)^{2}/4-\delta\right)}.
\end{equation}
\end{lem}
\begin{proof}
Since $\alpha\in (0,\xi Q)$, by Lemma \ref{thick-point-bound-boundary}, there exists an $\eta_0>0$ such that for any choice of $\eta\in (0,\eta_0)$, for all sufficiently small $\epsilon > 0$ and all $x \in I$, we have
\begin{align*}
  \mathbb{P}\left[\nu_{h}([x-\epsilon,x+\epsilon])^{2/d_\gamma-\eta} > \epsilon^{\alpha} \right] \leq \epsilon^{(Q-\alpha/(\xi-\gamma\eta/2))^2/4-\delta/3}.
\end{align*}
As an immediate consequence, for some constant $C > 0$, we have
\begin{align*}
  \EE[ \#\{U\in \Gtildesymlink{k}: \nu_h(U)^{2/d_\gamma-\eta}\geq 2^{-\alpha k}\}] \leq C 2^{k} \times 2^{-k((Q-\alpha/(\xi-\gamma\eta/2))^2/4-\delta/3)},
\end{align*}
where the term $C 2^{k}$ simply upper bounds the cardinality of the set $I\cap 2^{-k}\ZZ$. Thus, by a basic Markov and Borel-Cantelli argument, for any $\eta\in (0,\eta_0)$, for all large enough $k$, we almost surely have
\begin{equation}
  \label{eq:67}
  \#\{U\in \Gtildesymlink{k}: \nu_h(U)^{2/d_\gamma-\eta}\geq 2^{-\alpha k}\}\leq  C 2^{k} \times 2^{-k((Q-\alpha/(\xi-\gamma\eta/2))^2/4 - 2\delta/3)}.
\end{equation}
Combining the above with Lemma \ref{hughes-miller-boundary-lemma}, we obtain that for any $\eta\in (0,\eta_0)$, for all large enough $k$, we almost surely have
\begin{equation}
  \label{eq:68}
  \#\{U\in \Gtildesymlink{k}: \mathrm{diam}_{D_h}(U)\geq 2^{-\alpha k}\}\leq  C 2^{k} \times 2^{-k((Q-\alpha/(\xi-\gamma\eta/2))^2/4 - 2\delta/3)}.
\end{equation}
As a consequence, by choosing $\eta$ small enough compared to $\delta$, we obtain that for all $k$ large enough, 
\begin{equation}
  \label{eq:69}
  \# \Gkatildesymlink{k}{\alpha} \leq C2^k\times 2^{-k\left((Q-\alpha/\xi)^2/4-\delta\right)}.
\end{equation}
This completes the proof.
\end{proof}

We can now use the above to show that for an appropriate choice of $\alpha_*$, the second term on the right hand side in \eqref{eq:split-boundary} is necessarily small.

\begin{lem}\label{euclidean-sum-control-lemma-bdy}
  Fix an $\alpha_*\in \Adeltasymlink{\delta}\cap (0,\xi Q)$ such that $\min \tilde{g}( (-\infty,\alpha_*]\cap \Adeltasymlink{\delta}) > (2+\beta)\delta$. Then, almost surely, for any $\epsilon>0$, we can choose $k_0$ to be large enough such that
  \begin{equation}\label{negative-set-lower-bound-bdy}
    \sum_{\alpha\in \Adeltasymlink{\delta}, \alpha<\alpha_*} \sum_{k>k_0}\# \Gkatildesymlink{k}{\alpha} (2^{-k})^{\beta (\alpha-\delta)}<\epsilon.
    \end{equation}
\end{lem}

\begin{proof}
The result follows by essentially the same proof as that of Lemma \ref{euclidean-sum-control-lemma}.
\end{proof}

We now combine the preceding lemmas to complete the proof of \Cref{bdy-kpz}.

\begin{proof}[Proof of Lemma \ref{bdy-kpz}]
  To begin, as mentioned earlier, it suffices to fix a closed interval $I = [-N,N] \subseteq \RR$ and assume that $H\subseteq I$. Fix $\delta>0$ and use this to define $\Adeltasymlink{\delta}$; we will send this to $0$ at the end of the proof. Now, on combining Lemma \ref{frostman-bdy} with Lemma \ref{lem:split-bd}, we obtain that almost surely, for all $k_0$ large enough, for any cover $\cR$ of $H$ consisting only of elements of $\Gtildegeqsymlink{k_0}$ and any $\alpha_*\in \Adeltasymlink{\delta}$,
  \begin{equation}
    \label{eq:75}
   0< C_{H,\beta}\leq \sum_{R\in \cR}\operatorname{diam}_{\mathrm{Euc}}(R)^{\beta (\alpha_*-\delta)} + \sum_{\alpha\in \Adeltasymlink{\delta}, \alpha<\alpha_*} \sum_{k>k_0}\# \Gkatildesymlink{k}{\alpha} (2^{-k})^{\beta (\alpha-\delta)}.
  \end{equation}
  Now, consider the quadratic equation $x\mapsto \tilde{g}(x)$ defined in \ref{eq:boundary-polynomial}, whose roots are given by
  \begin{equation}
    \label{eq:76}
    \underline{r}_\partial,\overline{r}_\partial= \frac{d_{\gamma}(\gamma^{2}+4) -4\beta\gamma^{2} \pm 2\gamma\sqrt{4d_{\gamma}^{2} - (8\beta + 2\beta\gamma^{2})d_{\gamma} + 4\beta^{2}\gamma^{2}}}{2d_{\gamma}^{2}}.
  \end{equation}
  Note that $\beta\underline{r}_\partial = \DbKPZ{\gamma}{\beta}$. We define $\alpha_{\delta,*}$ to be the largest value in $[\underline{\alpha},\underline{r}_\partial]\cap \Adeltasymlink{\delta}$ such that $\tilde{g}(\alpha_{\delta,*})>(2+\beta)\delta$, which is well-defined for small enough $\delta$ since $\tilde{g}$ is a quadratic with positive leading coefficient. Now, by invoking Lemma \ref{euclidean-sum-control-lemma-bdy} with $\epsilon= C_{H,\beta}/2$, we have
  \begin{equation}
    \label{eq:77}
    C_{H,\beta}/2\leq \sum_{R\in \cR}  \mathrm{diam}_{\mathrm{Euc}}(R)^{\beta(\alpha_{\delta,*}-\delta)}.
  \end{equation}
It can be checked that $\Gtildegeqsymlink{k_0}$ is a sufficient covering family in the sense of Definition \ref{def-sufficient-collection} for any $k_0$, and thus the above implies that $\operatorname{dim}_{\mathrm{Euc}}(\Ssym{f\lvert_X}{V})\geq \beta (\alpha_{\delta,*}-\delta)$. Finally, since $\tilde{g}|_{[\underline{\alpha},\underline{r}_\partial]}$ is a strictly decreasing continuous function, we also have $\lim_{\delta\rightarrow 0} \alpha_{\delta,*}=\underline{r}_\partial$, thereby yielding that $\operatorname{dim}_{\mathrm{Euc}}(H)\geq \beta \underline{r}_\partial=\DbKPZ{\gamma}{\beta}$, which completes the proof.
\end{proof}

Finally, we use Lemma \ref{bdy-kpz} to complete the proof of \Cref{lqg_line_lower_bounds_thm}.
\begin{proof}[Proof of Theorem \ref{lqg_line_lower_bounds_thm}]
    First assume that $\beta < \dim(V)$. By an immediate application of Theorem \ref{intersection_universal_bounds}, we obtain that $\dim_{\mathrm{LQG}}(\Ssym{f}{V})\geq \beta$. Thereafter, by invoking Lemma \ref{bdy-kpz} along with the continuity of the function $\DbKPZ{\gamma}{\cdot}$, we obtain $\dim_{\mathrm{Euc}}(\Ssym{f}{V})\geq \DbKPZ{\gamma}{\beta}$. Finally, while \Cref{bdy-kpz} applies for only 
    $\beta < \dim(V)$, we may extend to the $\beta = \dim(V)$ case by the same continuity argument as in the Proof of \Cref{lqg_lower_bounds_thm}, and this completes the proof.
\end{proof}

\section{\texorpdfstring{Proofs of \Cref{thm:lqgres}, \Cref{thm:lqgbbdres}, and \Cref{thm:poissonroad}}{Proofs of main results}}
\label{sec:appl-lqg-poiss}
In this section, we combine the results of the previous sections to complete the proofs of the results stated in Section \ref{sec:main-results:-summ}. As preparation, we first state a lemma which, in effect, checks that the LQG metric satisfies the regularity assumption \ref{it:as} of Theorem \ref{3-star_points_general}; this will require geodesic confluence (Proposition \ref{prop:conf1}).

\begin{lem}\label{countable_tracing_geos_lemma}
Consider the metric space $(X,D)$ given by $(\CC,D_h)$ where $h$ is a whole-plane GFF and $D_h$ is the $\gamma$-LQG metric.
Then almost surely, simultaneously for all $a,b\in \mathbb{Q}^2$, the set of $r \in \RR$ for which there exists a $D_{h}$-geodesic started from $a$ which spends a nontrivial interval of time in $\Fsym{a}{b}{r}{=}$ is at most countable.
\end{lem}

\begin{proof}
  Since $\mathbb{Q}^2$ is countable, it suffices to prove the desired statement for fixed points $a,b\in \mathbb{Q}^{2}$. The strategy now is to use geodesic confluence to construct a countable collection $S$ of geodesics such that any sub-segment of a geodesic starting at $a$ must contain a path $\eta\in S$. Indeed, we define
  \begin{equation}
    \label{eq:5}
    S=\bigcup_{u,v\in \mathbb{Q}^2}\bigcup_{0\leq s<t\leq D_h(u,v), (s,t)\in \mathbb{Q}^2}\{\Gamma_{u,v}\lvert_{[s,t]}\},
  \end{equation}
  and by \Cref{prop:conf1}, it follows that for any geodesic $P$ starting from $a$ and any nontrivial sub-segment $P'\subseteq P$, there must exist at least one $\eta\in S$ for which $\eta\subseteq P'$. Now, suppose we have a ``bad'' $r\in \RR$ for which there exists a point $z\in \CC$ and a $D_h$-geodesic $\Gamma_{a,z}$ and a nontrivial interval $I$ such that $\Gamma_{a,z}\lvert_{I}\subseteq \Fsym{a}{b}{r}{=}$. Then, by the above, there must exist an $\eta\in S$ such that $\eta\subseteq \Gamma_{a,z}\lvert_I\subseteq \Fsym{a}{b}{r}{=}$. However, note that $\Fsym{a}{b}{r}{=}\cap \Fsym{a}{b}{s}{=}$ are disjoint for $r\neq s$. Thus, since $S$ is countable, the set of bad $r$ must be at most countable, and this completes the proof.
\end{proof}

We will also require a boundary version of the above result.
\begin{lem}\label{countable_tracing_geos_lemma_bd}
Consider the metric space $(X,D)$ given by $(\overline{\HH},D_h)$ where $h$ is a Neumann GFF and $D_h$ is the $\gamma$-LQG metric. Then almost surely, simultaneously for all $a,b\in \mathbb{Q}^2\cap \overline{\HH}$, the set of $r \in \RR$ for which there exists a $D_{h}$-geodesic started from $a$ which spends a nontrivial interval of time in $\Fsym{a}{b}{r}{=}$ is at most countable.
\end{lem}
\begin{proof}
  Once Proposition \ref{prop:conf2} is substituted for Proposition \ref{prop:conf1}, the proof is verbatim the same as that of Lemma \ref{countable_tracing_geos_lemma}.
\end{proof}

We will also require the following version of Lemma \ref{countable_tracing_geos_lemma} for Kendall's Poisson roads metric.
\begin{lem}
  \label{countable_tracing_geos_lemma_poisson}
    Fix $\beta>2$ and consider the corresponding Poisson roads metric $(\mathbb{R}^2,D)$. Then almost surely, simultaneously for all $a,b\in \mathbb{Q}^2$, the set of $r \in \RR$ for which there exists a $D$-geodesic started from $a$ which spends a nontrivial interval of time in $\Fsym{a}{b}{r}{=}$ is at most countable.
  \end{lem}
  \begin{proof}
    Just as in \eqref{eq:5}, we define
\begin{equation}
  \label{eq:42}
      S=\bigcup_{u,v\in \mathbb{Q}^2}\bigcup_{0\leq s<t \leq D(u,v), (s,t)\in \mathbb{Q}^2}\{\Gamma_{u,v}\lvert_{[s,t]}\}.
    \end{equation}
    The work \cite{geosdonotpausenroute} established confluence properties for the Poisson roads metric akin to Proposition \ref{prop:conf1}. In particular, it is established in \cite[Lemma 1]{geosdonotpausenroute} that, almost surely, for any points $u,v\in \CC$, any geodesic $\Gamma_{u,v}$, and any nontrivial interval $I\subseteq [0, D(u,v)]$, there must exist an $\eta\in S$ such that $\eta\subseteq \Gamma_{u,v}\lvert_I$. Thus, the verbatim analogue of the proof of Lemma \ref{countable_tracing_geos_lemma} holds for the current setting of the Poisson roads metric as well. This completes the proof.
  \end{proof}
Finally, we now combine the results of the entirety of the paper to complete the proofs of Theorems \ref{thm:lqgres} and \ref{thm:lqgbbdres}, which concern LQG defined with the whole-plane GFF and with the free-boundary GFF respectively.

\begin{proof}[Proof of Theorem \ref{thm:lqgres}]
  We go through the proofs of each of the four items one by one.
  
  \noindent \textbf{Proof of item \ref{it:LQG-3-lb}}: By using \Cref{countable_tracing_geos_lemma}, we see that the necessary assumption \ref{it:as} of \Cref{3-star_points_general} is a.s.\ satisfied for the countable dense set $A = \mathbb{Q}^2$. We note also that the identity map from $(\mathbb{C}, D_{h})$ to $\CC$ endowed with the Euclidean metric is a homeomorphism. As a consequence $(\CC,D_h)$ is a $\sigma$-compact metric space which is not a tree and also has trivial first homology (with integer coefficients). That is, almost surely, all the conditions of \Cref{3-star_points_general} hold. We conclude that almost surely, $\dim_{D_h}\left(T_{\textup{3-star}}\right)\geq 2$. By combining this with the worst-case KPZ estimate Proposition \ref{prop:worstkpz}, we obtain $\dim_{\mathrm{Euc}}(T_{\textup{3-star}})\geq  \DKPZ{\gamma}{2}$, and this completes the proof of item \ref{it:LQG-3-lb}.

 \noindent \textbf{Proof of item \ref{it:LQG-2-lb-wholeplane}}: Fix an arbitrary pair of points $a, b \in \mathbb{C}$, and define
\begin{equation}
f(z) = D_{h}(b,z) - D_{h}(a,z),
\end{equation}
which is Lipschitz with respect to $D_h$. Note that the function $f$ cannot be globally constant since we have $f(a)=D_h(b,a)>0$ and $f(b)=-D_h(b,a)<0$. For $r\in \RR$, recall the set $\Fsym{a}{b}{r}{=}=\{z\in \CC: D(z,b)-D(z,a)=r\}$ defined earlier in \eqref{def:F}. As a consequence of Lemma \ref{lem-inside}, note that for any $r\in \RR$ and any $z\in \Fsym{a}{b}{r}{=}$, we are guaranteed that $z$ admits almost disjoint geodesics $\Gamma_{z,a},\Gamma_{z,b}$ unless the following holds:
\begin{enumerate}
\item \label{it:timeint1} A geodesic from $z$ to $a$ traces $\Fsym{a}{b}{r}{=}$ for a nontrivial time interval.
\end{enumerate}
Let $K_{a,b}$ be the set of  $r\in \RR$ for which condition \ref{it:timeint} above holds. By Lemma \ref{countable_tracing_geos_lemma}, the set $K_{a,b}$ is necessarily at most countable, which implies that  $\dim (\RR\setminus K_{a,b})=1$. We now observe that $f^{-1}(\RR\setminus K_{a,b})\subseteq T_{\textup{2-star}}$. As a consequence, by using Theorem \ref{lqg_lower_bounds_thm} with $\beta=1$, we immediately obtain the desired lower bounds for the LQG and Euclidean dimensions of the set $\Ssym{f}{\RR\setminus K_{a,b}}\subseteq f^{-1}(\RR\setminus K_{a,b})\subseteq T_{\textup{2-star}}$. This completes the proof of item \ref{it:LQG-2-lb-wholeplane}.

\noindent \textbf{Proof of item \ref{it:LQG-1net-lb}}: We define
\begin{equation}
  \label{eq:6}
  f(x) = \min\{r : x \in \mathcal{B}^{\sbullet[0.5],z}_{r}(a)\},
\end{equation}
where we recall that $\mathcal{B}^{\sbullet[0.5],z}_{r}(a)$ is the filled LQG metric ball with radius $r$ centered at $a$ and targeted at $z$. Recall that we are justified taking a minimum rather than an infimum here by \Cref{lem:infmin}. Then $f$ is Lipschitz with respect to the LQG metric (see the proof of  \Cref{intersection_metric_nets}), and its non-constancy set $\Ssym{f}{\RR}$ is the metric net $\mathcal{N}^{z}(a)$. As a result, the lower bounds of \Cref{lqg_lower_bounds_thm} (with $\beta = 1$) apply to the LQG metric net $\mathcal{N}^{z}(a)$. This completes the proof of item \ref{it:LQG-1net-lb}.

\noindent \textbf{Proof of item \ref{it:LQG-net-int-lb}}: The hypotheses of \Cref{intersection_metric_nets_general} apply almost surely when $h$ is a whole-plane GFF and $D_{h}$ is the whole-plane LQG metric: the subcritical LQG metric a.s. induces the Euclidean topology, and \cite{lqg-metric-estimates} Lemma 3.8 proves the property $\lim_{z\rightarrow\infty}D(w,z) = \infty$ (or, equivalently, that $(\mathbb{C},D)$ is boundedly compact). Thus, the desired result follows by applying \Cref{intersection_metric_nets_general}. This completes the proof of item \ref{it:LQG-net-int-lb}.
\end{proof}

\begin{proof}[Proof of Theorem \ref{thm:lqgbbdres}]
  We now go through the proofs of both the items one by one.
  
  \noindent \textbf{Proof of item \ref{it:LQG-2-lb}}: In view of Theorem \ref{boundary_thm} (applied with $A=\RR$), we need only verify that there exist some $x,y\in \RR$ such that $(x,y)\in \Hsym{\overline{\HH}}{D_h}$. However, by Lemma \ref{countable_tracing_geos_lemma_bd}, the above holds a.s.\ simultaneously for all $x,y\in \mathbb{Q}\subseteq \RR$. This completes the proof of item \ref{it:LQG-2-lb}.

  \noindent \textbf{Proof of item \ref{it:LQG-net-lb-boundary}}: We apply \Cref{intersection_universal_bounds} with $f(z) = \min\{r: z\in \cB_r^{\bullet,\infty}(0)\}$ and $V = \mathbb{R}$. Note that we have already shown this choice of $f$ is Lipschitz in the proof of \Cref{intersection_metric_nets} and thus \Cref{intersection_universal_bounds} applies. This completes the proof of item \ref{it:LQG-net-lb-boundary}.
\end{proof}
 We now come to the proof of Theorem \ref{thm:poissonroad}, and this concerns Kendall's Poisson roads metric.
\begin{proof}[Proof of Theorem \ref{thm:poissonroad}]
  We go through the proofs of both the items one by one.
  
  \noindent \textbf{Proof of item \ref{it:pr-3-lb}}: 
  Recall that $(\mathbb{R}^{2},D)$ is homeomorphic to $(\mathbb{R}^{2},|\cdot|)$ where $|\cdot|$ denotes the Euclidean metric (\cite{kendallrandomlinesmetricspaces}, pg. 515), and thus is a $\sigma$-compact metric space which is not a tree and also has trivial first homology (with integer coefficients). Also, by \cite[Theorem 5.1]{kahnlineprocess}, metric balls for the Poisson roads metric almost surely have finite Euclidean diameter, which implies the property $\lim_{z\rightarrow\infty}D(w,z)=\infty$ (or, equivalently, that $(\mathbb{R}^{2},D)$ is boundedly compact). This verifies that the Poisson roads metric satisfies the first two assumptions in Theorem \ref{3-star_points_general} and it now remains to check assumption \ref{it:as}. For this, we set $A=\mathbb{Q}^2$ and first note that, and by the uniqueness of geodesics between fixed points (see \cite[Theorem 4.4]{kendallrandomlinesmetricspaces}), there is a.s.\ a unique geodesic $\Gamma_{u,v}$ simultaneously for all $u,v\in A$. Secondly, we invoke Lemma \ref{countable_tracing_geos_lemma_poisson} and this verifies that assumption \ref{it:as} is indeed satisfied. Thus, by \Cref{3-star_points_general}, the set of $3$-stars for the Poisson roads metric $D$ must have Hausdorff dimension (with respect to the Poisson roads metric) at least $2$. This completes the proof of item \ref{it:pr-3-lb}.

  \noindent \textbf{Proof of item \ref{it:pr-net-int-lb}}: The desired result will follow by an application of \Cref{intersection_metric_nets_general}. We have already verified in the proof of item \ref{it:pr-3-lb} that the necessary hypotheses are satisfied: in particular, that the Poisson roads metric a.s. induces the Euclidean topology and that $(\mathbb{R}^{2},D)$ is boundedly compact. Item \ref{it:pr-net-int-lb} thus follows by an application of \Cref{intersection_metric_nets_general}.
\end{proof}

\section{Appendix}

Here we provide brief proofs of our assumptions about the zeros of the functions \eqref{eq:quad} and \eqref{eq:54} whose roots we considered in Section 4. These proofs are straightforward, but we do make use of the bounds for LQG dimension $d_{\gamma}$ derived in \cite{gp-lfpp-bounds} to verify in each case that the zeros are real and fall in the required ranges. We begin with the function \eqref{eq:quad} considered in the proof of the Euclidean dimension lower bound.

\begin{prop}\label{euclidean-roots-appendix-lemma}
Given $\gamma \in (0,2)$ and $\beta \in [0,1]$, define the function $g(\alpha) = -1+\beta\alpha+(Q-\alpha/\xi)^{2}/2$. Also define the constants $\underline{\alpha} = \xi\left(Q-2\right)$, $ \overline{\alpha} = \xi\left(Q+2\right)$. Then for every such choice of $\gamma$ and $\beta$, the following are true: i) the zeros $\underline{r} \leq \overline{r}$ of $g$ are real, and ii) $\underline{\alpha} < \underline{r} < \overline{\alpha}$.
\end{prop}

\begin{proof}
$g$ is quadratic in $\alpha$ with zeros given by
\begin{align*}
    \underline{r}, \overline{r} 
    &= \frac{d_{\gamma}(\gamma^{2}+4) - 2\beta\gamma^{2} \pm 2\gamma\sqrt{2d_{\gamma}^{2}-(4\beta+\beta\gamma^{2}) d_{\gamma}+\beta^{2}\gamma^{2}}}{2d_{\gamma}^{2}},
\end{align*}
so to show i) it will suffice to show that the quantity $D(d_\gamma) = 2d_{\gamma}^{2}-(4\beta+\beta\gamma^{2}) d_{\gamma}+\beta^{2}\gamma^{2}$ is non-negative. $D$ is itself quadratic in $d_{\gamma}$ with real roots
\begin{align*}
    \underline{d}, \overline{d} &= \frac{\beta}{4}\left(4 + \gamma^{2} \pm \sqrt{16 + \gamma^{4}}\right)
\end{align*}
and positive leading coefficient, so it is non-negative on $(-\infty,\underline{d}]\cup[\overline{d},\infty)$. We will show that $d_{\gamma} \geq \overline{d}$ for all $\gamma \in (0,2)$ and $\beta \in [0,1]$, so that $D(d_{\gamma})$ is always non-negative for our choice of parameters.

For this we will make use of the bounds for $d_{\gamma}$ derived in \cite{gp-lfpp-bounds}.  When $\gamma \leq \sqrt{8/3}$, Corollary 2.5 of \cite{gp-lfpp-bounds} gives us that
\begin{align*}
    d_{\gamma} &\geq \frac{2\gamma^{2}}{4+\gamma^{2}-\sqrt{16+\gamma^{4}}}  = \frac{1}{4}\left(4+\gamma^{2}+\sqrt{16+\gamma^{4}}\right) \geq \frac{\beta}{4}\left(4 + \gamma^{2} + \sqrt{16 + \gamma^{4}}\right) = \overline{d},   
\end{align*}
where the last inequality follows because $\beta \leq 1$.
When $\gamma > \sqrt{8/3}$, the same corollary gives us that 
\begin{equation*}
d_{\gamma} \geq \frac{1}{3}\left(4+\gamma^{2}+\sqrt{16+2\gamma^{2}+\gamma^{4}}\right)
> \frac{1}{4}\left(4+\gamma^{2}+\sqrt{16+\gamma^{4}}\right) \geq \overline{d}.
\end{equation*}
So $d_{\gamma} \geq \overline{d}$ for all $\gamma \in (0,2)$ and $\beta \in [0,1]$, which by the preceding discussion proves i). \\

Now for the upper bound in ii), we observe that
\begin{align*}
    \underline{r} &= \frac{d_{\gamma}\left(\gamma^{2}+4\right) - 2\beta\gamma^{2} - 2\gamma\sqrt{D(d_{\gamma})}}{2d_{\gamma}^{2}} \leq \frac{d_{\gamma}\left(\gamma^{2}+4\right)}{2d_{\gamma}^{2}} = \frac{\gamma^{2}}{2d_{\gamma}} + \frac{2}{d_{\gamma}} < \frac{\gamma^{2}}{2d_{\gamma}} + \frac{2}{d_{\gamma}} + \frac{2\gamma}{d_{\gamma}} = \overline{\alpha},
\end{align*}
where the first inequality follows because $\beta\gamma^{2} \geq 0$ and $\gamma\sqrt{D(d_{\gamma})} \geq 0$, and the final equality follows from the definitions of $Q$ and $\xi$ \eqref{def:xi-and-Q}.

For the lower bound in ii), since $g$ is quadratic in $\alpha$ with positive leading coefficient and real roots, it will suffice to show that $g(\underline{\alpha}) > 0$ and $g'(\underline{\alpha}) < 0$. We compute
\begin{align*}
    g(\underline{\alpha}) &= -1 + \beta\underline{\alpha} + \frac{1}{2}\left(Q-\frac{\underline{\alpha}}{\xi}\right)^{2} = -1 + \beta\underline{\alpha} + \frac{1}{2}(2)^{2} = \beta\underline{\alpha} + 1 > 0,
\end{align*}
and
\begin{align*}
    g'(\underline{\alpha}) &= \beta - \frac{Q - \underline{\alpha}/\xi}{\xi} = \beta - \frac{2}{\xi} = \beta - \frac{2d_{\gamma}}{\gamma} < 0,
\end{align*}
where to obtain the final inequality we use that $\beta \leq 1$ while $d_{\gamma} > \gamma$. So it follows that $\underline{\alpha} < \underline{r}$; this completes the proof of ii).
\end{proof}

We also prove an analogous result for the function \eqref{eq:54} considered in the LQG lower bound.
\begin{prop}\label{lqg-roots-appendix-lemma}
Given $\gamma \in (0,2)$ and $\beta \in [0,1]$, define the function $\ell(\alpha)= \beta + 1/\alpha + \frac{1}{2\alpha}\left( \frac{2}{\gamma} - \frac{\gamma}{2} - \frac{\alpha}{\xi}\right)^{2} - d_\gamma$. Define the constants $\underline{\alpha} = \xi\left(Q-2\right)$, $ \overline{\alpha} = \xi\left(Q+2\right)$. Then for every such choice of $\gamma$ and $\beta$, the following are true: i) the zeros $\underline{r} \leq \overline{r}$ of $\ell$ are real, ii) $\underline{\alpha}< \xi(Q-\gamma)< \overline{r}< \overline{\alpha}$, and iii) $\ell$ is strictly increasing and non-negative on $[\overline{r}, \overline{\alpha}]$.
\end{prop}

\begin{proof}
    Solving $\ell(\alpha) = 0$ yields two real zeros $\underline{r} \leq \overline{r} \in \mathbb{R} \setminus \{0\}$ if and only if $D(d_{\gamma}) = 2d_{\gamma}^{2}-(4\beta+\beta\gamma^{2}) d_{\gamma}+\beta^{2}\gamma^{2}$ is non-negative; these roots are given by
    \begin{align*}
        \underline{r},\overline{r}&=\frac{16+\gamma^{4}}{8d_{\gamma} + (2d_{\gamma} - 4\beta)\gamma^{2} \pm 4\gamma\sqrt{D(d_{\gamma})}}.
    \end{align*}
We have already shown in the proof of \Cref{euclidean-roots-appendix-lemma} that $D(d_{\gamma}) \geq 0$ when $\gamma \in (0,2), \beta \in [0,1]$, so it follows that $\ell$ has two real zeros; this gives i).  

Now we will prove ii). The larger zero $\overline{r}$ is given by
    \begin{align*}
        \overline{r}&=\frac{16+\gamma^{4}}{8d_{\gamma} + (2d_{\gamma} - 4\beta)\gamma^{2} - 4\gamma\sqrt{D(d_{\gamma})}} = \frac{4d_{\gamma} + (d_{\gamma}-2\beta)\gamma^{2} + 2\gamma\sqrt{D(d_{\gamma})}}{2d_{\gamma}^{2}}.
    \end{align*}
We compute
\begin{align*}
    \overline{r} - \xi(Q-\gamma) &= \frac{4d_{\gamma} + (d_{\gamma}-2\beta)\gamma^{2} + 2\gamma\sqrt{D(d_{\gamma})}}{2d_{\gamma}^{2}} - \frac{4 - \gamma^{2}}{2d_{\gamma}} = \frac{\gamma\left((d_{\gamma}-\beta)\gamma + \sqrt{D(d_{\gamma})}\right)}{d_{\gamma}^{2}} > 0,
\end{align*}
where the positivity is immediate because $d_{\gamma} \geq 2 > \beta$ and $\sqrt{D(d_{\gamma})} \geq 0$ as discussed above.
We also compute,
\begin{align*}
    \overline{\alpha} - \overline{r} &= \frac{(\gamma+2)^{2}}{2d_{\gamma}} - \frac{4d_{\gamma} + (d_{\gamma}-2\beta)\gamma^{2} + 2\gamma\sqrt{D(d_{\gamma})}}{2d_{\gamma}^{2}} = \frac{\gamma\left(2d_{\gamma} + \beta\gamma -\sqrt{D(d_{\gamma})}\right)}{d_{\gamma}^{2}}.
\end{align*}
Since $2d_{\gamma} + \beta\gamma > 0$, to show that $\overline{\alpha} - \overline{r} > 0$ it thus suffices to show $\left(2d_{\gamma} + \beta\gamma\right)^{2} > D(d_{\gamma})$; this is indeed the case as $D(d_{\gamma}) \leq 2d_{\gamma}^{2} + \beta^{2}\gamma^{2}$, while we have $(2d_{\gamma}+\beta\gamma)^{2} \geq 4d_{\gamma}^{2} + \beta^{2}\gamma^{2}  > 2d_{\gamma}^{2} + \beta^{2}\gamma^{2}$. This concludes the proof of ii).

Finally, for iii), we compute
\begin{align*}
    \ell'(\alpha) = \frac{d_{\gamma}^{2}}{2\gamma^{2}} - \frac{\gamma^{4}+16}{8\gamma^{2}\alpha^{2}},
\end{align*}
so $\ell'(\alpha) > 0$ whenever $\alpha > \alpha_{*} = \frac{\sqrt{\gamma^{4}+16}}{2d_{\gamma}}$. We have (since $\beta \leq 1$ and $D(d_{\gamma}) \geq 0$)
\begin{align*}
    \overline{r} \geq \frac{4d_{\gamma} + (d_{\gamma}-2)\gamma^{2}}{2d_{\gamma}^{2}},
\end{align*}
so to show that $\overline{r} \geq \alpha_{*}$ it suffices to show that
\begin{align*}
4d_{\gamma} + (d_{\gamma}-2)\gamma^{2} \geq d_{\gamma}\sqrt{\gamma^{4}+16}.
\end{align*}
Rearranging, this condition is equivalent to 
\begin{align*}
d_{\gamma} \geq \frac{1}{4}\left(4+\gamma^{2}+\sqrt{\gamma^{4}+16}\right),
\end{align*}
which holds for all $\gamma \in (0,2)$ by Corollary 2.5 of \cite{gp-lfpp-bounds}. So $\overline{r} \geq \alpha_{*}$, and so $\ell'(\overline{r}) \geq 0$ and $\ell'(\alpha) > 0$ for all $\alpha > \overline{r}$. So $\ell$ is strictly increasing and non-negative on $[\overline{r},\overline{\alpha}]$.
\end{proof}

An analogous proposition also holds for the function \eqref{eq:boundary-polynomial} considered in the boundary case; the proof is very similar to that of \Cref{euclidean-roots-appendix-lemma} so we omit it here.

\bibliography{references}

@article{gwynne2022geodesicsmetricballboundaries,
  title={Geodesics and metric ball boundaries in {Liouville} quantum gravity},
  author={Gwynne, Ewain and Pfeffer, Joshua and Sheffield, Scott},
  journal={Probability Theory and Related Fields},
  volume={182},
  number={3},
  pages={905--954},
  year={2022},
  publisher={Springer}
}

@article{kendallrandomlinesmetricspaces,
      title={From Random Lines to Metric Spaces}, 
      author={Kendall, Wilfrid S.},
      year={2017},
      journal={The Annals of Probability},
      volume={45(1)},
      pages={467-517},
      url={https://arxiv.org/abs/1403.1156}, 
}

@article{kahnlineprocess,
      title={Improper {Poisson} line process as {SIRSN} in any dimension}, 
      author={Kahn, Jonas},
      year={2016},
      journal={The Annals of Probability},
      volume={44(4)},
      pages={2694-2725},
      url={https://arxiv.org/pdf/1503.03976}, 
}

@article{geosdonotpausenroute,
author = {Blanc, Guillaume and Curien, Nicolas and Kahn, Jonas},
title = {Geodesics in planar Poisson road random metric},
journal = {Proceedings of the London Mathematical Society},
volume = {131},
number = {1},
pages = {e70070},
doi = {https://doi.org/10.1112/plms.70070},
url = {https://londmathsoc.onlinelibrary.wiley.com/doi/abs/10.1112/plms.70070},
eprint = {https://londmathsoc.onlinelibrary.wiley.com/doi/pdf/10.1112/plms.70070},
year = {2025}
}

@article{Dau25,
 author = {Dauvergne, Duncan},
 title = {The 27 geodesic networks in the directed landscape},
 fjournal = {Inventiones Mathematicae},
 journal = {Invent. Math.},
 issn = {0020-9910},
 volume = {242},
 number = {1},
 pages = {123--220},
 year = {2025},
 language = {English},
 doi = {10.1007/s00222-025-01355-8},
 zbMATH = {8090104}
}

@book{pw-gff-notes,
    AUTHOR = {Werner, Wendelin and Powell, Ellen},
     TITLE = {Lecture notes on the {G}aussian free field},
    SERIES = {Cours Sp\'{e}cialis\'{e}s [Specialized Courses]},
    VOLUME = {28},
 PUBLISHER = {Soci\'{e}t\'{e} Math\'{e}matique de France, Paris},
      YEAR = {2021},
     PAGES = {vi+171},
      ISBN = {978-2-85629-952-4},
   MRCLASS = {60-02 (60G15 60G60 60J67 82B20 82B21)},
  MRNUMBER = {4466634},
       eprint = {\arxiv{2004.04720}},
}

@article{shef-gff,
      AUTHOR = {Sheffield, Scott},
     TITLE = {Gaussian free fields for mathematicians},
   JOURNAL = {Probab. Theory Related Fields},
  FJOURNAL = {Probability Theory and Related Fields},
    VOLUME = {139},
      YEAR = {2007},
    NUMBER = {3-4},
     PAGES = {521--541},
      ISSN = {0178-8051},
     CODEN = {PTRFEU},
   MRCLASS = {60K35 (60J65 81T10 82B31)},
  MRNUMBER = {2322706 (2008d:60120)},
MRREVIEWER = {Ofer Zeitouni},
       eprint= {\arxiv{math/0312099}}
       }

@article{hm-metric-gluing,
  title={Equivalence of metric gluing and conformal welding in $\gamma$-{L}iouville quantum gravity for $\gamma \in (0, 2)$},
  author={Hughes, Liam and Miller, Jason},
  journal={Annales de l'Institut Henri Poincar\'e (B) Probabilit\'es et Statistiques},
  volume={61},
  number={3},
  pages={1961--2003},
  year={2025},
}

@incollection{ddg-metric-survey,
    AUTHOR = {Ding, Jian and Dub\'{e}dat, Julien and Gwynne, Ewain},
     TITLE = {Introduction to the {L}iouville quantum gravity metric},
 BOOKTITLE = {I{CM}---{I}nternational {C}ongress of {M}athematicians. {V}ol.
              6. {S}ections 12--14},
     PAGES = {4212--4244},
 PUBLISHER = {EMS Press, Berlin},
      YEAR = {[2023] \copyright 2023},
   MRCLASS = {60D05 (60G60)},
  MRNUMBER = {4680401},
       eprint = {\arxiv{2109.01252}},
}

@ARTICLE{dg-confluence,
    AUTHOR = {Ding, Jian and Gwynne, Ewain},
     TITLE = {Regularity and confluence of geodesics for the supercritical
              {L}iouville quantum gravity metric},
   JOURNAL = {Probab. Math. Phys.},
  FJOURNAL = {Probability and Mathematical Physics},
    VOLUME = {5},
      YEAR = {2024},
    NUMBER = {1},
     PAGES = {1--54},
      ISSN = {2690-0998},
   MRCLASS = {60D05 (60G60 83C45)},
  MRNUMBER = {4696081},
       DOI = {10.2140/pmp.2024.5.1},
       URL = {https://doi.org/10.2140/pmp.2024.5.1},
       eprint = {\arxiv{2104.06502}}, 
}

@article{gwynne-geodesic-network,
    AUTHOR = {Gwynne, Ewain},
     TITLE = {Geodesic networks in {L}iouville quantum gravity surfaces},
   JOURNAL = {Probab. Math. Phys.},
  FJOURNAL = {Probability and Mathematical Physics},
    VOLUME = {2},
      YEAR = {2021},
    NUMBER = {3},
     PAGES = {643--684},
      ISSN = {2690-0998},
   MRCLASS = {83C45 (60D05 60G60)},
  MRNUMBER = {4408022},
       DOI = {10.2140/pmp.2021.2.643},
       URL = {https://doi.org/10.2140/pmp.2021.2.643},
       eprint = {\arxiv{2010.11260}}
}

@book{mq-strong-confluence,
  title={Geodesics in the Brownian map: Strong confluence and geometric structure},
  author={Miller, Jason and Qian, Wei},
  series={Memoirs of the American Mathematical Society},
  number={1602},
  year={2025},
  publisher={American Mathematical Society}
}

@article{afs-metric-ball, 
    AUTHOR = {Ang, Morris and Falconet, Hugo and Sun, Xin},
     TITLE = {Volume of metric balls in {L}iouville quantum gravity},
   JOURNAL = {Electron. J. Probab.},
  FJOURNAL = {Electronic Journal of Probability},
    VOLUME = {25},
      YEAR = {2020},
     PAGES = {Paper No. 160, 50},
   MRCLASS = {60D05 (60G15 60G60)},
  MRNUMBER = {4193901},
       DOI = {10.1214/20-ejp564},
       URL = {https://doi.org/10.1214/20-ejp564},
       eprint = {\arxiv{2001.11467}},
}

@ARTICLE{vargas-dozz-notes,
       author = {{Vargas}, Vincent},
        title = "{Lecture notes on Liouville theory and the DOZZ formula}",
      journal = {ArXiv e-prints},
         year = "2017",
        month = "Dec", 
archivePrefix = {arXiv},
       eprint = {\arxiv{1712.00829}},
 primaryClass = {math.PR},
       adsurl = {https://ui.adsabs.harvard.edu/abs/2017arXiv171200829V}
}

@book{berestycki-lqg-notes,
 author = {Berestycki, Nathana{\"e}l and Powell, Ellen},
 title = {Gaussian free field and {Liouville} quantum gravity},
 fseries = {Cambridge Studies in Advanced Mathematics},
 series = {Camb. Stud. Adv. Math.},
 volume = {220},
 isbn = {978-1-00-940550-8; 978-1-00-940549-2},
 year = {2026},
 publisher = {Cambridge: Cambridge University Press},
 language = {English},
 doi = {10.1017/9781009405492},
 zbMATH = {8097122}
}

@article{KPZ88,
    author = "Knizhnik, V. G. and Polyakov, Alexander M. and Zamolodchikov, A. B.",
    editor = "Khalatnikov, I. M. and Mineev, V. P.",
    title = "{Fractal Structure of 2D Quantum Gravity}",
    reportNumber = "PRINT-88-0812 (LANDAU)",
    doi = "10.1142/S0217732388000982",
    journal = "Mod. Phys. Lett. A",
    volume = "3",
    pages = "819",
    year = "1988"
}

@article{gm-uniqueness, 
    AUTHOR = {Gwynne, Ewain and Miller, Jason},
     TITLE = {Existence and uniqueness of the {L}iouville quantum gravity
              metric for {$\gamma\in(0,2)$}},
   JOURNAL = {Invent. Math.},
  FJOURNAL = {Inventiones Mathematicae},
    VOLUME = {223},
      YEAR = {2021},
    NUMBER = {1},
     PAGES = {213--333},
      ISSN = {0020-9910},
   MRCLASS = {83C45 (30F99 58 60)},
  MRNUMBER = {4199443},
       DOI = {10.1007/s00222-020-00991-6},
       URL = {https://doi.org/10.1007/s00222-020-00991-6},
   eprint = {\arxiv{1905.00383}},
}

@article{gm-confluence,
    AUTHOR = {Gwynne, Ewain and Miller, Jason},
     TITLE = {Confluence of geodesics in {L}iouville quantum gravity for
              {$\gamma \in (0,2)$}},
   JOURNAL = {Ann. Probab.},
  FJOURNAL = {The Annals of Probability},
    VOLUME = {48},
      YEAR = {2020},
    NUMBER = {4},
     PAGES = {1861--1901},
      ISSN = {0091-1798},
   MRCLASS = {60J67 (60G52)},
  MRNUMBER = {4124527},
       DOI = {10.1214/19-AOP1409},
       URL = {https://doi.org/10.1214/19-AOP1409},
   eprint = {\arxiv{1905.00381}}, 
}

@article{lqg-metric-estimates,
    AUTHOR = {Dub\'{e}dat, Julien and Falconet, Hugo and Gwynne, Ewain and
              Pfeffer, Joshua and Sun, Xin},
     TITLE = {Weak {LQG} metrics and {L}iouville first passage percolation},
   JOURNAL = {Probab. Theory Related Fields},
  FJOURNAL = {Probability Theory and Related Fields},
    VOLUME = {178},
      YEAR = {2020},
    NUMBER = {1-2},
     PAGES = {369--436},
      ISSN = {0178-8051},
   MRCLASS = {60D05 (60G60)},
  MRNUMBER = {4146541},
       DOI = {10.1007/s00440-020-00979-6},
       URL = {https://doi.org/10.1007/s00440-020-00979-6},
   eprint = {\arxiv{1905.00380}}, 
}

@article{dddf-lfpp,
    AUTHOR = {Ding, Jian and Dub\'{e}dat, Julien and Dunlap, Alexander and
              Falconet, Hugo},
     TITLE = {Tightness of {L}iouville first passage percolation for
              {$\gamma \in (0,2)$}},
   JOURNAL = {Publ. Math. Inst. Hautes \'{E}tudes Sci.},
  FJOURNAL = {Publications Math\'{e}matiques. Institut de Hautes \'{E}tudes
              Scientifiques},
    VOLUME = {132},
      YEAR = {2020},
     PAGES = {353--403},
      ISSN = {0073-8301},
   MRCLASS = {60K35 (60G15 83C45)},
  MRNUMBER = {4179836},
       DOI = {10.1007/s10240-020-00121-1},
       URL = {https://doi.org/10.1007/s10240-020-00121-1},
       eprint = {\arxiv{1904.08021}},
}

@article{gp-lfpp-bounds,
       author = {{Gwynne}, Ewain and {Pfeffer}, Joshua},
        title = "{Bounds for distances and geodesic dimension in Liouville first passage percolation}",
      journal = {{E}lectronic {C}ommunications in {P}robability}, 
      volume = {24},
      pages= {no. 56, 12},
         year = "2019",  
       eprint = {\arxiv{1903.09561}},   
}

@article{DS11,
       AUTHOR = {Duplantier, Bertrand and Sheffield, Scott},
     TITLE = {Liouville quantum gravity and {KPZ}},
   JOURNAL = {Invent. Math.},
  FJOURNAL = {Inventiones Mathematicae},
    VOLUME = {185},
      YEAR = {2011},
    NUMBER = {2},
     PAGES = {333--393},
      ISSN = {0020-9910},
     CODEN = {INVMBH},
   MRCLASS = {81T40 (60K35)},
  MRNUMBER = {2819163 (2012f:81251)},
MRREVIEWER = {Lee-Peng Teo},
       DOI = {10.1007/s00222-010-0308-1},
       URL = {http://dx.doi.org/10.1007/s00222-010-0308-1},
       eprint={\arxiv{1206.0212}}
       }

@article{rhodes-vargas-review,
    AUTHOR = {Rhodes, R{\'e}mi and Vargas, Vincent},
     TITLE = {Gaussian multiplicative chaos and applications: {A} review},
   JOURNAL = {Probab. Surv.},
  FJOURNAL = {Probability Surveys},
    VOLUME = {11},
      YEAR = {2014},
     PAGES = {315--392},
      ISSN = {1549-5787},
   MRCLASS = {60G57 (28A80 60G15)},
  MRNUMBER = {3274356},
       DOI = {10.1214/13-PS218},
       URL = {http://dx.doi.org/10.1214/13-PS218},
eprint={\arxiv{1305.6221}}
}

@article{kahane,
AUTHOR = {Kahane, Jean-Pierre},
     TITLE = {Sur le chaos multiplicatif},
   JOURNAL = {Ann. Sci. Math. Qu\'ebec},
  FJOURNAL = {Annales des Sciences Math\'ematiques du Qu\'ebec},
    VOLUME = {9},
      YEAR = {1985},
    NUMBER = {2},
     PAGES = {105--150},
      ISSN = {0707-9109},
   MRCLASS = {60G57 (60G42)},
  MRNUMBER = {829798 (88h:60099a)},
MRREVIEWER = {S. D. Chatterji}
}

@article{DG20,
 author = {Ding, Jian and Gwynne, Ewain},
 title = {The fractal dimension of {Liouville} quantum gravity: universality, monotonicity, and bounds},
 fjournal = {Communications in Mathematical Physics},
 journal = {Commun. Math. Phys.},
 issn = {0010-3616},
 volume = {374},
 number = {3},
 pages = {1877--1934},
 year = {2020},
 language = {English},
 doi = {10.1007/s00220-019-03487-4},
 zbMATH = {7187082},
 Zbl = {1436.83024}
}

@article{RV11,
 author = {Rhodes, R{\'e}mi and Vargas, Vincent},
 title = {{KPZ} formula for log-infinitely divisible multifractal random measures},
 fjournal = {European Series in Applied and Industrial Mathematics (ESAIM): Probability and Statistics},
 journal = {ESAIM, Probab. Stat.},
 issn = {1292-8100},
 volume = {15},
 pages = {358--371},
 year = {2011},
 language = {English},
 doi = {10.1051/ps/2010007},
 zbMATH = {6157522},
 Zbl = {1268.60070}
}

@article{tbm-characterization,
    AUTHOR = {Miller, Jason and Sheffield, Scott},
     TITLE = {An axiomatic characterization of the {B}rownian map},
   JOURNAL = {J. \'{E}c. polytech. Math.},
  FJOURNAL = {Journal de l'\'{E}cole polytechnique. Math\'{e}matiques},
    VOLUME = {8},
      YEAR = {2021},
     PAGES = {609--731},
      ISSN = {2429-7100},
   MRCLASS = {60D05 (83C45)},
  MRNUMBER = {4225028},
       DOI = {10.5802/jep.15},
       URL = {https://doi.org/10.5802/jep.15},
   eprint = {\arxiv{1506.03806}}, 
}

@ARTICLE{lqg-tbm1,
    AUTHOR = {Miller, Jason and Sheffield, Scott},
     TITLE = {Liouville quantum gravity and the {B}rownian map {I}: the
              {${\mathrm QLE}(8/3,0)$} metric},
   JOURNAL = {Invent. Math.},
  FJOURNAL = {Inventiones Mathematicae},
    VOLUME = {219},
      YEAR = {2020},
    NUMBER = {1},
     PAGES = {75--152},
      ISSN = {0020-9910},
   MRCLASS = {60D05 (60J67 60J80 83C45)},
  MRNUMBER = {4050102},
       DOI = {10.1007/s00222-019-00905-1},
       URL = {https://doi.org/10.1007/s00222-019-00905-1},
   eprint = {\arxiv{1507.00719}}, 
}

@ARTICLE{legall-geodesic-stars,
    AUTHOR = {{Le Gall}, Jean-Fran{\c{c}}ois},
     TITLE = {Geodesic stars in random geometry},
   JOURNAL = {Ann. Probab.},
  FJOURNAL = {The Annals of Probability},
    VOLUME = {50},
      YEAR = {2022},
    NUMBER = {3},
     PAGES = {1013--1058},
      ISSN = {0091-1798},
   MRCLASS = {60D05 (05C80)},
  MRNUMBER = {4413211},
MRREVIEWER = {Zhongyang Li},
       DOI = {10.1214/21-aop1553},
       URL = {https://doi.org/10.1214/21-aop1553},
       eprint = {\arxiv{2102.00489}},
}

@ARTICLE{gwynne-miller-uihpq,
    AUTHOR = {Ewain Gwynne and Jason Miller},
     TITLE = {Scaling limit of the uniform infinite half-plane quadrangulation in the {G}romov-{H}ausdorff-{P}rokhorov-uniform topology},
   JOURNAL = {Electron. J. Probab.},
  FJOURNAL = {Electronic Journal of Probability},
      YEAR = {2017},
    VOLUME = {22},
       PNO = {84},
     PAGES = {1-47},
      ISSN = {1083-6489},
       DOI = {10.1214/17-EJP102},
      SICI = {1083-6489(2017)22:84<1:SLOTUI>2.0.CO;2-E},
   eprint = {\arxiv{1608.00954}},
}

@article{legall-geodesics,
AUTHOR = {{Le Gall}, Jean-Fran{\c{c}}ois},
     TITLE = {Geodesics in large planar maps and in the {B}rownian map},
   JOURNAL = {Acta Math.},
  FJOURNAL = {Acta Mathematica},
    VOLUME = {205},
      YEAR = {2010},
    NUMBER = {2},
     PAGES = {287--360},
      ISSN = {0001-5962},
     CODEN = {ACMAA8},
   MRCLASS = {60J65 (60D05)},
  MRNUMBER = {2746349 (2012b:60272)},
MRREVIEWER = {Xiaowen Zhou},
       DOI = {10.1007/s11511-010-0056-5},
       URL = {http://dx.doi.org/10.1007/s11511-010-0056-5},
eprint={\arxiv{0804.3012}}
}

@article{miermont-brownian-map,
AUTHOR = {Miermont, Gr{\'e}gory},
     TITLE = {The {B}rownian map is the scaling limit of uniform random
              plane quadrangulations},
   JOURNAL = {Acta Math.},
  FJOURNAL = {Acta Mathematica},
    VOLUME = {210},
      YEAR = {2013},
    NUMBER = {2},
     PAGES = {319--401},
      ISSN = {0001-5962},
   MRCLASS = {60D05 (05Cxx 52C17 60F05 60G57 60J65)},
  MRNUMBER = {3070569},
       DOI = {10.1007/s11511-013-0096-8},
       URL = {http://dx.doi.org/10.1007/s11511-013-0096-8},
eprint={\arxiv{1104.1606}}
}

@article {epw-real-tree,
    AUTHOR = {Evans, Steven N. and Pitman, Jim and Winter, Anita},
     TITLE = {Rayleigh processes, real trees, and root growth with
              re-grafting},
   JOURNAL = {Probab. Theory Related Fields},
  FJOURNAL = {Probability Theory and Related Fields},
    VOLUME = {134},
      YEAR = {2006},
    NUMBER = {1},
     PAGES = {81--126},
      ISSN = {0178-8051},
   MRCLASS = {60B05 (60B99 60J27 60J80)},
  MRNUMBER = {2221786},
MRREVIEWER = {Jochen Geiger},
       DOI = {10.1007/s00440-004-0411-6},
       URL = {https://doi.org/10.1007/s00440-004-0411-6},
       eprint={\arxiv{0402293}}
}

@article {dov-dl,
    AUTHOR = {Dauvergne, Duncan and Ortmann, Janosch and Vir\'{a}g, B\'{a}lint},
     TITLE = {The directed landscape},
   JOURNAL = {Acta Math.},
  FJOURNAL = {Acta Mathematica},
    VOLUME = {229},
      YEAR = {2022},
    NUMBER = {2},
     PAGES = {201--285},
      ISSN = {0001-5962},
   MRCLASS = {60K35 (60B20 82B43)},
  MRNUMBER = {4554223},
       DOI = {10.4310/acta.2022.v229.n2.a1},
       URL = {https://doi.org/10.4310/acta.2022.v229.n2.a1},
       eprint = {\arxiv{1812.00309}},
}

@preamble{
   "\def\cprime{$'$} "
}

@book {bbi-metric-geometry,
    AUTHOR = {Burago, Dmitri and Burago, Yuri and Ivanov, Sergei},
     TITLE = {A course in metric geometry},
    SERIES = {Graduate Studies in Mathematics},
    VOLUME = {33},
 PUBLISHER = {American Mathematical Society, Providence, RI},
      YEAR = {2001},
     PAGES = {xiv+415},
      ISBN = {0-8218-2129-6},
   MRCLASS = {53C23},
  MRNUMBER = {1835418},
MRREVIEWER = {Mario Bonk},
       DOI = {10.1090/gsm/033},
       URL = {http://dx.doi.org/10.1090/gsm/033},
}

@book{DMS14,
 author = {Duplantier, Bertrand and Miller, Jason and Sheffield, Scott},
 title = {Liouville quantum gravity as a mating of trees},
 fseries = {Ast{\'e}risque},
 series = {Ast{\'e}risque},
 issn = {0303-1179},
 volume = {427},
 isbn = {978-2-85629-941-8},
 year = {2021},
 publisher = {Paris: Soci{\'e}t{\'e} Math{\'e}matique de France (SMF)},
 language = {English},
 zbMATH = {7453876},
 Zbl = {1503.60003}
}

@article{kpz-fluctuation,
author={M. Kardar and G. Parisi and Y.-C. Zhang},
title={{D}ynamic scaling of growing interfaces},
journal={{Phys. Rev. Lett.}},
volume={56},
pages={889--892}, 
month=mar,
year={1986}
}

@article{BSS19,
 author = {Basu, Riddhipratim and Sarkar, Sourav and Sly, Allan},
 title = {Coalescence of geodesics in exactly solvable models of last passage percolation},
 fjournal = {Journal of Mathematical Physics},
 journal = {J. Math. Phys.},
 issn = {0022-2488},
 volume = {60},
 number = {9},
 pages = {093301, 22},
 year = {2019},
 language = {English},
 doi = {10.1063/1.5093799},
 zbMATH = {7116349},
 Zbl = {1480.60284}
}

@article{DEP24,
 author = {Dembin, Barbara and Elboim, Dor and Peled, Ron},
 title = {Coalescence of geodesics and the {BKS} midpoint problem in planar first-passage percolation},
 fjournal = {Geometric and Functional Analysis. GAFA},
 journal = {Geom. Funct. Anal.},
 issn = {1016-443X},
 volume = {34},
 number = {3},
 pages = {733--797},
 year = {2024},
 language = {English},
 doi = {10.1007/s00039-024-00672-z},
 zbMATH = {7862379},
 Zbl = {1554.60126}
}

@article{BGH21,
 author = {Basu, Riddhipratim and Ganguly, Shirshendu and Hammond, Alan},
 title = {Fractal geometry of {{\(\text{Airy}_2\)}} processes coupled via the {Airy} sheet},
 fjournal = {The Annals of Probability},
 journal = {Ann. Probab.},
 issn = {0091-1798},
 volume = {49},
 number = {1},
 pages = {485--505},
 year = {2021},
 language = {English},
 doi = {10.1214/20-AOP1444},
 zbMATH = {7310797},
 Zbl = {1457.82165}
}

@misc{Bha22,
 author = {Bhatia, Manan},
 title = {Atypical stars on a directed landscape geodesic},
 year = {2022},
 howpublished = {Preprint, {arXiv}:2211.05734 [math.{PR}] (2022)},
 url = {https://arxiv.org/abs/2211.05734},
 arXiv = {arXiv:2211.05734}
}

@article{BGH22,
 author = {Bates, Erik and Ganguly, Shirshendu and Hammond, Alan},
 title = {Hausdorff dimensions for shared endpoints of disjoint geodesics in the directed landscape},
 fjournal = {Electronic Journal of Probability},
 journal = {Electron. J. Probab.},
 issn = {1083-6489},
 volume = {27},
 pages = {44},
 note = {Id/No 1},
 year = {2022},
 language = {English},
 doi = {10.1214/21-EJP706},
 zbMATH = {7478703},
 Zbl = {1496.60115}
}

@misc{BK25,
      title={Strong confluence of geodesics in {L}iouville quantum gravity}, 
      author={Manan Bhatia and Konstantinos Kavvadias},
      year={2025},
      eprint={2512.09219},
      archivePrefix={arXiv},
      primaryClass={math.PR},
      url={https://arxiv.org/abs/2512.09219}, 
}

@article{Bha25,
 author = {Bhatia, Manan},
 title = {The {{\(d_{\gamma} /2\)}}-variation of distance profiles in {{\(\gamma\)}}-{Liouville} quantum gravity},
 fjournal = {Communications in Mathematical Physics},
 journal = {Commun. Math. Phys.},
 issn = {0010-3616},
 volume = {406},
 number = {3},
 pages = {58},
 note = {Id/No 61},
 year = {2025},
 language = {English},
 doi = {10.1007/s00220-024-05206-0},
 zbMATH = {7986052},
 Zbl = {1558.60029}
}

@article{GP22,
 author = {Gwynne, Ewain and Pfeffer, Joshua},
 title = {{KPZ} formulas for the {Liouville} quantum gravity metric},
 fjournal = {Transactions of the American Mathematical Society},
 journal = {Trans. Am. Math. Soc.},
 issn = {0002-9947},
 volume = {375},
 number = {12},
 pages = {8297--8324},
 year = {2022},
 language = {English},
 doi = {10.1090/tran/8085},
 zbMATH = {7612765},
 Zbl = {1501.60007}
}

@article{MQ20,
 author = {Miller, Jason and Qian, Wei},
 title = {The geodesics in {Liouville} quantum gravity are not {Schramm}-{Loewner} evolutions},
 fjournal = {Probability Theory and Related Fields},
 journal = {Probab. Theory Relat. Fields},
 issn = {0178-8051},
 volume = {177},
 number = {3-4},
 pages = {677--709},
 year = {2020},
 language = {English},
 doi = {10.1007/s00440-019-00949-7},
 zbMATH = {7227776},
 Zbl = {1450.60009}
}
\bibliographystyle{hmralphaabbrv}

\end{document}